\documentclass[10pt,reqno]{amsart}
\usepackage{amsmath}
\usepackage{amsfonts}
\usepackage{amssymb}
\usepackage{graphicx}
\usepackage{amsthm,graphicx,color,yfonts}
\usepackage{pdfsync}
\usepackage{epstopdf}%
\usepackage{pifont}
\usepackage{bbm}
\usepackage{a4wide}
\usepackage{kotex}

\usepackage[colorlinks=true]{hyperref}
\hypersetup{linkcolor=red,citecolor=blue,filecolor=dullmagenta,urlcolor=blue}

\usepackage{wrapfig}
\usepackage{tikz}
\usetikzlibrary{arrows,calc,decorations.pathreplacing}
\definecolor{light-gray1}{gray}{0.90}
\definecolor{light-gray2}{gray}{0.80}
\definecolor{light-gray3}{gray}{0.60}

\newtheorem{thm}{Theorem}[section]
\newtheorem{lem}[thm]{Lemma}
\newtheorem{prop}[thm]{Propsition}
\newtheorem{cor}[thm]{Corollary}
\newtheorem{defn}[thm]{Definition}
\newtheorem{rem}[thm]{Remark}

\numberwithin{equation}{section}

\newcommand{\les}{\lesssim}
\newcommand{\lam}{{\lambda}}

\newcommand{\vp}{{\varphi}}

\newcommand{\de}{{\delta}}

\def\normo#1{\left\|#1\right\|}
\def\normb#1{\Big\|#1\Big\|}
\def\norm#1{\|#1\|}
\def\bra#1{\langle#1\rangle}
\def\wt#1{\widetilde{#1}}
\def\wh#1{\widehat{#1}}
\def\set#1{\{#1\}}

\newcommand{\T}{{\mathbb T}}
\newcommand{\R}{{\mathbb R}}
\newcommand{\Z}{{\mathbb Z}}
\newcommand{\ft}{{\mathcal{F}}}

\newcommand{\N}{{\mathcal{N}}}
\newcommand{\NT}{\mathcal{NT}}

\newcommand{\Sch}{{\mathcal{S}}}
\newcommand{\supp}{{\mbox{supp}}}

\newcommand{\noi}{{\noindent}}

\newcommand{\px}{\partial_x}
\newcommand{\pt}{\partial_t}

\begin{document}

	\title[Fifth-order mKdV equations]{Energy solutions for the fifth-order modified Korteweg de-Vries equations}
	
	\author{Chulkwang Kwak}
		\address{Department of Mathematics, Ewha Womans University, Seoul, 03760, Republic of Korea}
	\email{ckkwak@ewha.ac.kr}

\author{Kiyeon Lee}
\address{Stochastic Analysis and Application Research Center(SAARC), Korea Advanced Institute of Science and Technology, 291 Daehak-ro, Yuseong-gu, Daejeon, 34141, Republic of	Korea}
\email{kiyeonlee@kaist.ac.kr}

    \thanks{$*$All authors agree that this work is fairly contributed.}

\begin{abstract}
We consider the Cauchy problem for the fifth-order modified Korteweg-de Vries equation (mKdV) under the periodic boundary condition. The fifth-order mKdV is an asymptotic model for shallow surface waves, and (in the perspective of integrable systems) the second equation in the mKdV hierarchy as well. In strong contrast with the non-periodic case, periodic solutions for dispersive equations do not have a (local) smoothing effect, and this observation becomes a major obstacle to considering the Cauchy problem for dispersive equations under the periodic condition, consequently, the periodic fifth-order mKdV shows a quasilinear phenomenon, while the non-periodic case can be considered as a semilinear equation.

\medskip

In this paper, we mainly establish the global well-posedness of the fifth-order mKdV in the energy space ($H^2(\T)$), which is an improvement of the former result by the first author \cite{Kwak2018-2}. The main idea to overcome the lack of (local) smoothing effect is to introduce a suitable (frequency dependent) short-time space originally motivated by the work by Ionescu, Kenig, and Tataru \cite{IKT2008}. The new idea is to combine the (frequency) localized modified energy with additional weight in the spaces, which eventually handles the logarithmic divergence appearing in the energy estimates. Moreover, by using examples localized in low and very high frequencies, we show that the flow map of the fifth-order mKdV equation is not $C^3$, which implies that the Picard iterative method is not available for the local theory. This weakly concludes the quasilinear phenomenon of the periodic fifth-order mKdV. 
\end{abstract}
	
		\maketitle

	\tableofcontents	
		
\section{Introduction}

\subsection{Model description}
The Korteweg-de Vries (KdV) equation has been well-understood as an asymptotic model for shallow water waves and an completely integrable system for last decades. 

\medskip

In a perspective of an asymptotic model for surface waves, the incompressible, irrotational Euler equation is considered as a free boundary problem. With suitable non-dimensional parameters denoted by $\delta := \frac{h}{\lambda}$ and $\varepsilon := \frac{a}{h}$, where $h, \lambda$ and $a$ describe the water depth, the wave length and the amplitude of the free surface, respectively, the physical condition $\delta \ll 1$ is supposed to characterize water waves as long waves or shallow water waves. Moreover, with subspecialized conditions between $\varepsilon$ and $\delta$, there are several long wave approximations. We introduce three typical long wave regimes.

\medskip

\begin{enumerate}
\item Shallow water wave: $\varepsilon = 1$ and $\delta \ll 1$.
\item Korteweg-de Vries (KdV): $\varepsilon = \delta^2 \ll 1$ and $\mu \neq \frac13$.
\item Kawahara: $\varepsilon = \delta^4 \ll 1$ and $\mu = \frac13 + \nu\varepsilon^{\frac12}$.
\end{enumerate}
Here, another non-dimensional parameter $\mu$ is called the Bond number, which comes from the surface tension on the free surface.

\medskip

In the regime of Item (1), the (so-called) shallow water equations are obtained as $\delta \to 0$. It is known that the shallow water equations are analogous to one-dimensional compressible Euler equations for an isentropic flow of a gas of the adiabatic index $2$, and thus its solutions generally have a singularity in finite time, even if the initial data are sufficiently smooth. Therefore, this long wave regime is used to explain breaking waves of water waves. In the regime of Item (2), the following well-known, notable KdV equation can be derived from capillary-gravity waves \cite{KdV1895}: 
\[\pm2 u_t + 3uu_x +\left( \frac13 - \mu\right)u_{xxx} = 0.\]
It is worth to notice that when the Bond number $\mu = \frac13$, the equation degenerates to the inviscid Burgers equation. In connection with this critical Bond number, Hasimoto \cite{Hasimoto1970} derived a higher-order KdV equation of the form
\[\pm2 u_t + 3uu_x - \nu u_{xxx} +\frac{1}{45}u_{xxxxx} = 0,\]
in the regime of Item (3), which is nowadays called the Kawahara equation. See, for instance, \cite{BBM, BCS1, BCL, BLS} and references therein for more derivations of long wave approximations for water waves. 

\medskip

As a completely integrable system, Gardner, Greene, Kruskal and Miura \cite{GGKM1967} has discovered that the spectrum of the Hill operator
\begin{equation}\label{eq:Hill}
L(t) = \frac{d^2}{dx^2} - u(t,x)
\end{equation}
on the interval $[0,4\pi]$ is independent of $t$ when the potential $u(t,x)$ solves the KdV equation given by
\begin{equation}\label{eq:KdV}
u_t + u_{xxx} = 6uu_x.
\end{equation}
Later, Lax \cite{Lax1968} reformulated such a property as the so-called \emph{Lax pair formulation}, precisely, 
\begin{equation}\label{eq:Lax}
\mbox{KdV} \equiv \partial_t L = [B,L] = BL - LB,
\end{equation}
where $L$ is given in \eqref{eq:Hill} and $B(t) = 4\partial_x^3 - 6 u \partial_x -3\partial_x u$, and $[\cdot, \cdot]$ is the standard commutator. The right-hand side of \eqref{eq:Lax} can be understood as isospectral deformation of $L(t)$, which implies $\mbox{spec}(L(t)) = \mbox{spec}(L(0))$ as mentioned before. This observation opened the theory of complete integrable systems of PDEs (Lax pairs for the nonlinear Schr\"odinger equation, the Sine-Gordon equation, the Toda lattice, etc had been found), and the KdV equation holds a central example among other integrable systems. The differential operator $B$ can be extended to higher-order ones by
\[B_n = 4^n \px^{2n+1} + \sum_{j=1}^n \{a_{nj}\px^{2j-1} + \px^{2j-1} a_{nj}\}, \hspace{1em} n=0,1,\cdots,\]
where $B_0 = \px$ and the coefficient $a_{nj} = a_{nj}(u)$ can be chosen such that the commutator $[A_n,L]$ has order zero, and then Lax also generalized its formulation in the higher-order form
\[\pt L = B_n L - LB_n,\]
which provides the \emph{KdV hierarchy} as (we only give first few equations and Hamiltonians)
\begin{equation}\label{eq:KdV Hamiltonian}
\begin{split}
\pt u - \px u = 0 \hspace{0.5em}&\hspace{0.5em} \int \frac12 u^2\\
\pt u + \px^3 u - 6u\px u = 0 \hspace{0.5em}&\hspace{0.5em} \int \frac12 (\px u)^2 +  u^3\\
\pt u - \px^5 u + 5\px(\px^2(u^2) - (\px u)^2 - 2u^3) = 0 \hspace{0.5em}&\hspace{0.5em} \int \frac12 (\px^2 u)^2 - 5u(\px u)^2 + \frac52 u^4\\
\vdots \hspace{1cm}&\hspace{1cm} \vdots\\
\end{split}
\end{equation} 
Magri \cite{Magri1978}, thereafter, found an additional structure in this complete integrable system, the so-called bi-Hamiltonian structure, and further studies on integrable systems  have been widely done by several researchers (see, for instance, \cite{AS1981, KP2003, Deift} and references therein). 

\medskip

On the other hand, the Miura transformation \cite{Miura1968}
\begin{equation}\label{eq:miura}
u = v_x + v^2
\end{equation}
explains the relation between the KdV and mKdV equations, particularly, if $v$ satisfies the defocusing mKdV equation given by
\begin{equation}\label{eq:mkdv}
v_t + v_{xxx} - 6v^2v_x = 0,
\end{equation}
then $u$ defined as in \eqref{eq:miura} satisfies \eqref{eq:KdV}\footnote{However, the converse is not true due to the presence of the operator $2v + \px$ in the relation
\[u_t + u_{xxx} - 6uu_x = (2v + \px)(v_t + v_{xxx} - 6v^2v_x).\]}. 
Moreover, Miura, Gardner and Kruskal \cite{MGK1968} discovered infinitely many conserved densities and their flux corresponding to the mKdV, and it showed that the mKdV is also a completely integrable system. Recently, Choudhuri, Talukdar and Das \cite{CTD2009} derived the Lax representation and constructed the bi-Hamiltonian structure of the equations in the mKdV hierarchy. The following is a bundle of a few equations and their associated Hamiltonians with respect to bi-Hamiltonian structures (see also \cite{Grunrock} and references therein for the relation between \eqref{eq:KdV Hamiltonian} and \eqref{eq:hamiltonian}):
\begin{equation}\label{eq:hamiltonian}
\begin{split}
\pt u - \px u = 0 \hspace{0.5em}&\hspace{0.5em} \int \frac12 u^2\\
\pt u + \px^3 u - 6u^2\px u = 0 \hspace{0.5em}&\hspace{0.5em} \int \frac12 u_x^2 + \frac12 u^4\\
\pt u - \px^5 u + \px(10(u^2\px^2 u + u(\px u)^2) - 6u^5) = 0 \hspace{0.5em}&\hspace{0.5em} \int \frac12 u_{xx}^2  + 5u^2u_x^2+u^6\\
\vdots \hspace{1cm}&\hspace{1cm} \vdots\\
\end{split}
\end{equation} 

In this paper, we discuss the initial value problem for the fifth-order equation in the modified KdV hierarchy \eqref{eq:hamiltonian} under the periodic boundary condition:
\begin{equation}\label{eq:5mkdv_integrable}
\begin{cases}
\pt u - \px^5 u + 40u\px u \px^2u + 10 u^2\px^3u + 10(\px u)^3 - 30u^4 \px u = 0, \hspace{1em} (t,x) \in \R \times \T, \\
u(0,x) = u_0(x) \in H^s(\T),
\end{cases}
\end{equation}
where $\T = \R/2\pi\Z =  [0,2\pi]$ and $u=u(t,x)$ is a real-valued function.

\subsection{Main theorem}
It is already known that, due to the inverse scattering method, the global solution for \eqref{eq:5mkdv_integrable} exists for any Schwartz initial data. However, generalizing coefficients for the nonlinear terms may break the integrable structure, thus the theory of complete integrable systems is no longer available for the following equation:
\begin{equation}\label{eq:5mkdv-gen}
	\begin{cases}
		\pt u - \px^5 u + c_1u\px u \px^2u + c_2 u^2\px^3u + c_3(\px u)^3 + c_4u^4 \px u = 0,\\
		u(0,x) = u_0(x) \in H^s(\T),
	\end{cases}
\end{equation}
where $c_j$, $j=1,\cdots,4$, are real constant. 

\medskip

Even for integrable dispersive equations, the study on the low regularity local theory based on the analytic method is regarded as an important subject in order to develop mathematical theories (particularly, Harmonic analysis, Fourier analysis, Functional analysis, and so on for dispersive equations). Besides, the integrable structures under the non-periodic setting have been ignored (but the integrable structures, recently, were essentially used to prove the low regularity well-posedness of some particular equations, see Remark \ref{rem:Killip} below). This work, however, is to clarify that, under the periodic setting, partial properties of integrability (for instance, not all but some of conservation laws) are essentially needed to study the low regularity well-posedness problem. The following three Hamiltonians 
\begin{equation}\label{eq:hamiltonian-2}
	\begin{aligned}
\mathcal H_0 = \mathcal H_0[u](t) &:=  \int \frac12 u^2 (t,x) \; dx,\\
\mathcal H_1 = \mathcal H_1[u](t) &:=\int \left(\frac12 \left( \px u\right)^2 + \frac{c_1}{80} u^4 \right)(t,x) \; dx,\\
\mathcal H_2 = \mathcal H_2[u](t) &:= \int \left(\frac12 \left( \px^2 u \right)^2+ \frac{c_1}{8}u^2\left(\px u\right)^2 + \frac{c_1^2}{1600}u^6 \right) (t,x) \; dx,
	\end{aligned}
\end{equation}
are particularly required for the $H^2(\T)$ global well-posedness, and it restricts the coefficients in \eqref{eq:5mkdv-gen} to the following relations:
\begin{equation}\label{eq:coefficient constraint}
\frac{c_1}4 = c_2 = c_3 \quad \mbox{and} \quad -\frac3{160}c_1^2 =c_4.
\end{equation}

\begin{rem}
One can see that all quantities in \eqref{eq:hamiltonian-2} are conserved in time by the flow of \eqref{eq:5mkdv-gen} with \eqref{eq:coefficient constraint}.
\end{rem}

\begin{rem}
Taking $c_1 = 40$, one has \eqref{eq:5mkdv_integrable} from the equation \eqref{eq:5mkdv-gen} with \eqref{eq:coefficient constraint}.
\end{rem}

\begin{rem}\label{rem:Hamiltonian system}
The equation \eqref{eq:5mkdv-gen} with \eqref{eq:coefficient constraint} can be expressed as the Hamiltonian system with respect to $\mathcal H_2$ in \eqref{eq:hamiltonian-2} as follows:
\begin{equation*}
u_t = \partial_x \nabla_u \mathcal H_2[u(t)]  = \nabla_{\omega_{-\frac{1}{2}}} \mathcal H_2[u(t)],
\end{equation*}
where $\nabla_u$ is the $L^2$ gradient and $\omega_{-\frac{1}{2}}$ is the symplectic form in $H^{-\frac{1}{2}}$ defined as 
\[\omega_{-\frac{1}{2}}\left(v,w\right) := \int_{\T} v \partial_x^{-1} w dx,\]
for all $v,w \in H_0^{-\frac{1}{2}}$\footnote{Here $H_0^s$ space is the subspace of $H^s(\T)$ whose elements satisfy the mean zero condition}. 
\end{rem}

\begin{rem}\label{rem:GlobalEnergy}
Due to the conservations of $\mathcal H_0$, $\mathcal H_1$, and $\mathcal H_2$ in \eqref{eq:hamiltonian-2}, $H^2(\T)$ can be regarded as the energy space for the equation \eqref{eq:5mkdv-gen} with \eqref{eq:coefficient constraint}.
\end{rem}

\medskip

In this paper, we mainly prove that the equation \eqref{eq:5mkdv-gen} with \eqref{eq:coefficient constraint} is globally well-posed in the energy space $H^2(\T)$. To be precise, we introduce the notion of the well-posedness (for particularly scaling subcritical equations). The Duhamel principle says that the equation \eqref{eq:5mkdv-gen} with \eqref{eq:coefficient constraint} is equivalent to the following integral equation
\begin{equation}\label{Duhamel}
u(t) = S(t)u_0 - \int_0^t S(t-s) F(u)(s) \; ds,
\end{equation}
where $S(t)$ is the generator of the free evolution for $\partial_t u - \partial_x^5 u =0$, and $F(u)$ represents the nonlinear terms in \eqref{eq:5mkdv-gen} with \eqref{eq:coefficient constraint} (will be precisely introduced). The well-known notion of well-posedness was initially proposed by Hadamard \cite{Hadamard}, and now it is precisely stated as
\begin{defn}[Local well-posedness]\label{def:WP}
Let $u_0 \in H^s(\T)$ be given. We say that the initial value problem for \eqref{eq:5mkdv-gen} with \eqref{eq:coefficient constraint} is locally well-posed in $H^s(\T)$ if the following holds: 
\begin{enumerate}
\item \emph{(Existence)} There exist a time $T = T(\norm{u_0}_{H^s(\T)}) > 0$ and a subset $X_T^s$ of $C([0,T];H^s(\T))$ such that $u$ satisfying \eqref{Duhamel} exists in $X_T^s$.

\medskip
 
\item \emph{(Uniqueness)} The solution is unique in $X_T^s$.

\medskip 

\item \emph{(Continuous dependence on the data)} The map $u_0 \mapsto u$ is continuous from a ball $B \subset H^s(\T)$ to $X_T^s$ (with the $H^s(\T)$ topology).
\end{enumerate}
\end{defn}

\begin{rem}\label{rem:WP}
When $T$ can be taken arbitrary large, we say that the initial value problem for \eqref{eq:5mkdv-gen} with \eqref{eq:coefficient constraint} is globally (in time) well-posed in $H^s(\T)$.
\end{rem}

\begin{rem}\label{rem:ill-posed}
We say that given equation is $C^k$ well-posed, if the continuity condition of (3) in Definition \ref{def:WP} can be replaced by $C^k$-differentiability. Otherwise, we call $C^k$ ill-posedness. 
\end{rem}

\begin{thm}\label{thm:main}
	Let $s \ge 2$. Let $u_0 \in H^s(\T)$ be such that
	\begin{equation}\label{eq:level set}
		\int_{\T} (u_0(x))^2 \; dx = \gamma_1, \hspace{2em} \int_{\T} (\px u_0(x))^2 + (u_0(x))^4 \; dx = \gamma_2
	\end{equation}
	for some $\gamma_1, \gamma_2 \ge 0$. Then, the equation \eqref{eq:5mkdv-gen} with \eqref{eq:coefficient constraint} is locally well-posed in $H^s(\T)$ in the sense of Definition \ref{def:WP}. Particularly, the solution $u$ satisfies
	\begin{align}
		\begin{aligned}\label{eq:result}
		&u(t,x) \in C([-T,T];H^s(\T)) \quad \mbox{and}\\
		&\eta(t)\sum_{n \in \Z} e^{i(nx - 20n\int_0^t \norm{u(s)}_{L^4}^4 \; ds)}\wh{u}(t,n) \in  F^s(T) \cap C([-T,T];H^s(\T)).
		\end{aligned}
	\end{align}
Moreover, the flow map $S_T : H^s(\T) \to C([-T,T];H^s(\T))$ is continuous on the level set in $H^s(\T)$ satisfying \eqref{eq:level set}.
\end{thm}

\begin{rem}
A function $\eta$ in \eqref{eq:result} is any cut-off function in $C^{\infty}(\R)$ with $\mathrm{supp}(\eta) \subset [-T,T]$, and the resolution space $F^s(T)$ will be defined in Section \ref{sec:ftn-space}.
\end{rem}

As a consequence of Theorem \ref{thm:main} in addition to Remark \ref{rem:GlobalEnergy}, we have the global well-posedness of \eqref{eq:5mkdv-gen} with \eqref{eq:coefficient constraint} in the energy space.
\begin{cor}\label{cor:GWP}
The Cauchy problem of \eqref{eq:5mkdv-gen} with \eqref{eq:coefficient constraint} is globally well-posed in the energy space $H^2(\T)$.
\end{cor}

\begin{rem}
The Miura transform and suitable scaling-rescaling arguments ensure that Corollary \ref{cor:GWP} implies the global well-posedness of the fifth-order KdV equation
\begin{equation}\label{eq:5kdv}
	\pt u - \px^5 u + a_1 u_x \partial_x^2u + a_2 u\px^3u + a_3  u^2 u_x= 0
\end{equation}
in $H^1(\T)$, for the particular cases when $a_1 = \frac{c_1}{2}$, $a_2 = \frac{c_1}{4}$ and $a_3 = -\frac3{160}c_1^2$, $c_1 \ge 0$. 
%
%
%
%
%
%
%
This improves the previous results \cite{Kwak2016, Kwak2018-2}, and also extends the result in \cite{KM2018} to more general nonlinearities of the fifth-order KdV equation (but in the higher regularity Sobolve space).
\end{rem}

\begin{rem}
In the proof of Theorem \ref{thm:main}, two Hamiltonians $\mathcal H_0$ and $\mathcal H_1$ have to be conserved in time, thus coefficients in \eqref{eq:5mkdv-gen} are restricted only on the first one ($\frac{c_1}4 = c_2 = c_3$) in \eqref{eq:coefficient constraint}. This allows us to extend Theorem \ref{thm:main} for more general equations than \eqref{eq:5mkdv-gen} with \eqref{eq:coefficient constraint}.
\end{rem}

\begin{rem}\label{rem:integrable is enough}
The main analysis tool for the proof of Theorem \ref{thm:main} is not relied on variations of coefficients, thus it suffices to consider the Cauchy problem for the integrable equation \eqref{eq:5mkdv_integrable} which corresponds to the coefficients as \eqref{eq:coefficient constraint} with $c_1 = 40$.
\end{rem}

\begin{rem}
The full restriction \eqref{eq:coefficient constraint} of coefficients are necessary to obtain Corollary \ref{cor:GWP}. 
\end{rem}

Note that the result in Theorem \ref{thm:main} is based on a compactness method: one passes to the limit by compactness arguments on some approximate solutions (see Appendix \ref{sec:Appendix A}). On the other hand, another natural way to solve the Cauchy problem is to perform the Picard iterative method on the integral equation \eqref{Duhamel}, which provides a stronger result than one in Theorem \ref{thm:main} (particularly, the flow map becomes real analytic). However, due to the presence of non-trivial resonances, one can not apply the Picard iterative method on this equation. 
\begin{thm}\label{thm:C^3 illposed}
Let $s > 0$. Then, there does not exist $T > 0$ such that \eqref{eq:5mkdv-gen} admits a unique local solution on $[-T, T]$ and the flow map $u_0 \mapsto u(t)$, $t \in [T,T]$, for \eqref{eq:5mkdv-gen} is $C^3$-differentiable at the origin from $H^s(\T)$ to $H^s(\T)$.
\end{thm}

The failure of the $C^3$-differentiability of the flow map is due to the presence of non-trivial resonances as mentioned. On the other hand, renormalizing the equation \eqref{eq:5mkdv-gen} with \eqref{eq:coefficient constraint} removes all non-trivial resonances. Indeed, using conservation laws ($\mathcal H_0$ and $\mathcal H_1$) and the nonlinear transformation defined by 
\begin{equation*}
	\NT(u)(t,x) = v(t,x) := \frac{1}{\sqrt{2\pi}}\sum_{n \in \Z} e^{i(nx - 20n \int_0^t \norm{u(s)}_{L^4}^4 \; ds)}\wh{u}(t,n),
\end{equation*}
the equation \eqref{eq:5mkdv_integrable} is renormalized\footnote{See Section \ref{sec:reformulation} for the details.} by
\begin{equation}\label{eq:5mkdv3_0}
	\begin{split}
		\pt\wh{v}(n) - i(n^5 + d_1n^3 + d_2n)\wh{v}(n)=&~{}-20in^3|\wh{v}(n)|^2\wh{v}(n)\\
		&+10in \sum_{\N_{3,n}} \wh{v}(n_1)\wh{v}(n_2)n_3^2\wh{v}(n_3) \\
		&+10in \sum_{\N_{3,n}} \wh{v}(n_1)n_2\wh{v}(n_2)n_3\wh{v}(n_3)\\
		&+6i n\sum_{\N_{5,n}} \wh{v}(n_1)\wh{v}(n_2)\wh{v}(n_3)\wh{v}(n_4)\wh{v}(n_5),
	\end{split}
\end{equation}
with some constants $d_1, d_2$ (depending on conserved quantities). Nevertheless, since the nonlinearity is still strong than its dispersion effects, we cannot solve the renormalization of \eqref{eq:5mkdv3_0} by using Picard iteration. Hence, we get the following results.
\begin{thm}\label{thm:failure}
Let $s > 0$. Then, there does not exist $T > 0$ such that \eqref{eq:5mkdv3_0} admits a unique local solution on $[-T, T]$ and the flow map $v_0 \mapsto v(t)$, $t \in [T,T]$, for \eqref{eq:5mkdv3_0} is $C^5$-differentiable at the origin from $H^s(\T)$ to $H^s(\T)$.
\end{thm}

\begin{rem}
The proof of Theorem \ref{thm:C^3 illposed} is analogous to and simpler than the proof of Theorem \ref{thm:failure}, thus we only provide a precise proof for Theorem \ref{thm:failure} in Section \ref{sec:illposed} and Appendix \ref{appendix:calculation}, but we will introduce an example to break the $C^3$-differentiability of the flow map (see Remark \ref{rem:C^3diffexample}).
\end{rem}

\subsection{About literature}

There are numerous results regarding the well-posedness of KdV and mKdV equations (including higher order equations) not only on the periodic domain but also the non-periodic domain, and here we only introduce few but important works via analytic methods in historical order.

\medskip

The KdV equation \eqref{eq:KdV} on $\R$ has been firstly studied by Bona and Smith \cite{BS1975}, Bona and Scott \cite{BS1976}, Saut and Temam \cite{ST1976}, and Kato \cite{Kato1979}. The authors (independently) proved that the KdV equation is locally and globally well-posed in $H^s(\R)$, $s > \frac32$ and $s \ge 2$, respectively, and up to this moment, dispersive structures of the KdV equation were ignored (this argument is also available for mKdV as well as higher-order equations regardless of boundary conditions). Later, Kenig, Ponce, and Vega \cite{KPV1991} handled $u_x$ via the maximal function estimate, and thus proved the local well-posedness in $H^s(\R)$, $s > \frac34$, which, in addition to the energy conservation law, guarantees the global well-posedness in the energy space $H^1(\R)$. By developing the dispersive nature introduced by Bourgain \cite{Bourgain1993}, Kenig, Ponce and Vega \cite{KPV1996} established the sharp bilinear estimate which guarantees that the KdV equation is locally (resp. globally) well-posed in the negative Sobolev space $H^s(\R)$, $s > -\frac34$ (resp. in $L^2(\R)$). The Sobolev regularity index $s = -\frac34$ is optimal in the sense that the KdV equation no longer shows a semilinear nature below this regularity. The global well-posedness result has been extended to the same regularity as the local one by Colliander, Keel, Staffilani, Takaoka, and Tao \cite{CKSTT2003}. We also refer to \cite{Guo2009, Kishimoto2009, Liu2015} and references therein for more results.

\medskip

Compared to the KdV equation, the mKdV and higher-order equations received less attention. Kenig, Ponce, and Vega \cite{KPV1993} (we do not mention here former local well-posedness results via compactness method) proved that the mKdV equation \eqref{eq:mkdv} is locally (resp. globally) well-posed in $H^s(\R)$, $s \ge \frac14$ (resp. in $H^1(\R)$). Later, its simple proof was given in \cite{KPV1996}, and the global well-posedness with the same regularity as the local one was obtained in \cite{CKSTT2003}. The fifth-order KdV equation \eqref{eq:5kdv} on $\R$ has first been studied by Ponce \cite{Ponce1993}. Ponce proved that \eqref{eq:5kdv} is globally well-posed in $H^s(\R)$, $s \ge 4$, and Kwon \cite{Kwon2008-1} improved the local (resp. global) result for $s > \frac52$ (resp. $s \ge 3$). Later, the global well-posedness result in the energy space was proven in \cite{GKK2013, KP2015}. For the fifth-order mKdV \eqref{eq:5mkdv-gen} on $\R$, Kwon \cite{Kwon2008-2} proved the local (resp. global) well-posedness in $H^s(\R)$, $s \ge \frac34$ (resp. $s \ge 1$). We also refer to \cite{KP2016} for higher order equations on $\R$.

\medskip

\begin{rem}\label{rem:Killip}
Recently, Killip and Visan \cite{KV2019} proved the global well-posedness of the KdV equation in $H^{-1}(\R)$. They provided a new, general method of low regularity local theory for integrable equations, informed by the existence of commuting flows. With the strong ill-posedness result established by Molinet \cite{Molinet2011}, the initial value problem for the KdV equation has been completely resolved. Nevertheless, the well-posedness problem via analytic methods (without using integrable structures) in $H^s(\R)$, $-1 \le s < -\frac34$, is still left an interesting open problem. An analogous argument for the fifth-order KdV can be found in \cite{BKV2021}.
\end{rem}

In a strong contrast to the non-periodic problems, the periodic problems have less developed due to the lack of local smoothing properties. To overcome this obstacle, Bourgain \cite{Bourgain1993} introduced the \emph{Fourier restriction norm method} (or dispersive smoothing properties) for periodic solutions, and proved that KdV \eqref{eq:KdV} and mKdV \eqref{eq:mkdv} are globally well-posed in $L^2(\T)$ and $H^1(\T)$, respectively. Later, Kenig et al. \cite{KPV1996} improved the local results for KdV and mKdV in $H^{-\frac12}(\T)$ and $H^{\frac12}(\T)$, respectively, and Colliander et al. \cite{CKSTT2003} showed the global well-posedness of these equations at the same regularities. Takaoka and Tsutsumi \cite{TT2004} (see also \cite{NTT2010}) particularly improved the local theory for mKdV in $H^s(\T)$, $s > \frac13$, and Molinet, Pilod, and Vento \cite{MPV2018} recently obtained unconditional well-posedness in $H^s(\T)$, $s \ge \frac13$.

\medskip 

The fifth-order KdV and mKdV equations under the periodic setting have firstly (as far as we know) studied by the first author \cite{Kwak2016, Kwak2018-2}. In these works, it is shown that the fifth-order KdV and mKdV are locally well-posed in $H^2(\T)$ and $H^s(\T)$, $s > 2$, respectively. In particular, the different nature (quasilinear nature) of the fifth-order mKdV equation under the periodic setting from the non-periodic problem (semilinear nature) has been observed. 

\begin{rem}\label{rem:Killip2}
Kappeler and Topalov \cite{KT2006, KT2005} proved that KdV and mKdV are globally well-posed in $H^{-1}(\T)$ and $L^2(\T)$, respectively, via theory of complete integrability. 
\end{rem}

\bigskip

\subsection{Idea of proofs}\label{sec:idea}

In the view of Theorem 1.5 of \cite{Kwak2018-2}, the trilinear estimates
\begin{align}\label{eq:failure-xsb}
\|uv \px^3 w\|_{X^{s,b-1}_{\tau - n^5}} \les \|u\|_{X^{s,b}_{\tau-n^5}}\|v\|_{X^{s,b}_{\tau-n^5}}\|w\|_{X^{s,b}_{\tau-n^5}}
\end{align}
fail for any $s,b \in \R$. Note that estimates \eqref{eq:failure-xsb} fail only periodic domain due to the lack of the local smoothing effect for linear profiles. Among the contributions of frequencies, {\em high$\times$low$\times$low $\Rightarrow$ high} interaction components, where the third order derivative is taken in the highest frequency, is critical in the counterexample of \eqref{eq:failure-xsb}. This observation shows that the standard Fourier restriction norm method is not available for the Cauchy problem for \eqref{eq:5mkdv3_0}. Moreover, it leads to the lack of uniform continuity of the flow (see Theorems \ref{thm:C^3 illposed} and \ref{thm:failure} for partial results). 
To overcome these obstacles, we employ the short time Fourier restriction norm method to prove Theorem \ref{thm:main}. Ionescu, Kenig, and Tataru \cite{IKT2008} firstly introduced this idea to establish the global well-posedness of the KP-I equation in the energy space. The main ingredients of the proof are the following estimates:
\begin{align}
	\|v\|_{F^s(T)} &\les \|v\|_{E^s(T)} + \|\mathfrak{N} (v)\|_{N^s(T)},\label{eq:linear}\\
	\|\mathfrak{N}(v)\|_{N^s(T)} &\les \|v\|_{F^s(T)}^3 + \|v\|_{F^s(T)}^5, \label{eq:nonlinear-esti}\\
	\|v\|_{E^s(T)} &\les \left( 1+ \|v_0\|_{H^s(\T)}^2\right)\|v_0\|_{H^s(\T)}^2 + \left( 1+ \|v\|_{F^s(T)}^2 + \|v\|_{F^s(T)}^4\right) \|v\|_{F^s(T)}^4,\label{eq:energy-esti}
\end{align}
in appropriate function spaces $F^s(T), N^s(T)$, $E^s(T)$, which will be precisely defined in Section \ref{sec:ftn-space}, where $\mathfrak{N}(v)$ is a collection of nonlinearities in \eqref{eq:5mkdv3_0}. 

\medskip

An appropriate time scale for \eqref{eq:5mkdv3_0} can be taken comparable to $(\mbox{frequency})^{-2}$. To see this, we give an intuitive observation:
Let $N \in \Z_+$ and $s \in \R$ be fixed, and let $f := f_1 + f_N$ be given as an initial data, where $\hat{f}_j(n) = n^{-s}\delta(n-j)$\footnote{Here $\hat{f}(n)$ is the Fourier coefficient of $f$ precisely defied in Section \ref{sec:2}.}, for $j=1, N$. Note that $\|f\|_{H^s(\T)} \sim 1$. Under the philosophy of the local theory, suppose for a suitable time cut-off function $\rho$  that  $u(t,x) \approx u_1(t,x) + u_N(t,x) =: \rho(t/T)e^{t\px^5}f_1 + \rho(t/T)e^{t\px^5}f_N$ is a solution that evolves as a linear solution in a very short time interval $[0,T]$. Then we must have (particularly, we focus on the worst case)
\[\left\| \rho(t/T)\int_0^t e^{(t-s)\px^5}(u_1^2 \px^3u_N)(s) \; ds \right\|_{H^s(\T)} \lesssim 1.\]
The left-hand side can be further (but roughly) reduced by
\begin{equation}\label{eq:heuristic}
N^{s+3}\sum_{j \ge 1} 2^{j/2}\left\|\eta_j(\tau - n^5) (\tau - n^5 + iT^{-1})^{-1}\widehat{u_1^2 \px^3u_N} \right\|_{L_{\tau}^2\ell_n^2},
\end{equation}
due to \eqref{eq:small data1.2} (see Appendix A in \cite{GKK2013} for its proof)\footnote{Here the function $\eta_j$ is a Littlewood-Paley multiplier for modulations, see Section \ref{sec:2.1} for the precise definition.}. By Lemma \ref{lem:prop of Xk}, in addition to Lemma \ref{lem:tri-L2} (particularly, the worst case appears when $j_{sub} \le j_{max} - 10$, but we always have $2^{j_{sub}} \ge T^{-1}$), we conclude 
\[\eqref{eq:heuristic} \lesssim N^{s+3}N^{-2}T^{\frac12}N^{-s} \lesssim 1,\]
as long as $T \sim N^{-2}$. 

\medskip

To compensate the nonlinear estimates \eqref{eq:nonlinear-esti}, we need to control the energy of solutions in \eqref{eq:linear} (particularly, we need energy-type estimates \eqref{eq:energy-esti}). A direct computation, however, cannot give such estimates \eqref{eq:energy-esti} due to the derivative loss in the {\em high$\times$low$\times$low $\Rightarrow$ high} interaction component, for instance, we have
\begin{equation}\label{eq:energyfail}
\begin{aligned}
	&\sum_{0\le k_1,k_2 \le k-10} \left| \sum_{n, \overline{\N}_{3,n}}  \int_0^{t_k} \chi_k(n) [ \chi_{k_1}(n_1) \wh{v}(n_1)\chi_{k_2}(n_2) \wh{v}(n_2) n_3^3 \wh{v}(n_3) ] \chi_k(n)\wh{v}(n)\right| \\
	&\hspace{1cm}\les \|v\|_{F^{\frac14+}(T)}^2 \sum_{|k'-k|\le 5} 2^{2k}\|P_{k'}v\|_{F_k(T)}^2.
\end{aligned}
\end{equation}
To recover it, we introduce a modified energy (see \eqref{eq:new energy1-5}). The modified energy in addition to the commutator estimates enables us to move two derivatives taken in the high frequency regime to the low frequency regime, so thus we completely improve the above estimate \eqref{eq:energyfail}. This idea in its current form was first described by Kwon \cite{Kwon2008-1} and developed as localize versions in \cite{KP2015, Kwak2018-1}.

\medskip

In the previous work \cite{Kwak2018-2}, the first author also employed the short time Fourier restriction norm method to obtain \eqref{eq:energy-esti}, but failed in the energy space (only in $H^s(\T)$, $s>2$) due to the presence of the logarithmic divergence appearing in the high-low interaction case, particularly, the case when the low frequency mode has the largest modulation. To get rid of this obstacle and then reduce the regularity $s$ down, we adopt an additional weight in the function spaces, which was first described in \cite{IK2007JMAS} and also used in \cite{GPWW2011, GKK2013} (it is the first attempt to use an additional weight in its current form under the periodic setting as far as we know). This weight exactly attack the issued case in the previous work, but no further advantages (see Lemmas \ref{lem:energy1-1} and \ref{lem:commutator1}, particularly, the case when third-maximum frequency is strictly greater than the minimum frequency). 

\medskip

The trade-off of using a short time structure as well as a weight is to worsen nonlinear interaction components particularly whose resulting frequency is smaller than at least two of the others, for instance,
\[P_{low}(P_{low} u \cdot P_{high} v \cdot P_{high} w_{xxx}).\]
However, nonlinearties under coefficient conditions \eqref{eq:coefficient constraint} can be rewritten as a divergence form, thus one can take a weight to balance nolinear and energy estimates.

\medskip

On the other hand, inspired by the work \cite{Tzvetkov1999} for KdV equation (see also \cite{MST2001} for Benjamin-Ono equation), we are able to prove Theorems \ref{thm:C^3 illposed} and \ref{thm:failure}. Let us be precise. According to the philosophy of local theory, a solution behaves like a linear profile within its lifespan. By taking a suitable initial data $\phi$ consisting of two components localized low and high frequency modes, respectively, one has to establish that there is no function space $X_T$ continuously embedded in $C([-T,T],H^s(\T))$ such that
\begin{align*}
	\normo{\int_0^t e^{(t-s)\px^5} (v^2 v_{xxx}) ds}_{X_T} \les \|v\|_{X_T}^3,
\end{align*}
where $v \sim e^{t\px^5}\phi$. Crucial interation components for the proofs of Theorems \ref{thm:C^3 illposed} and \ref{thm:failure} are \emph{high-low-low to high} resonant ($(n_1,n_2,n_3,n) = (1,-1,N,N)$) and non-resonant ($(n_1,n_2,n_3,n) = (2,-1,N-1,N)$)\footnote{Integration by parts under non-resonant interactions guarantees derivative gains, but distributed terms occur quintic resonance.} components, respectively. See Section \ref{sec:illposed} for the details.

\subsection*{Organization} The paper is organized as follows: In Section \ref{sec:2}, we provide some preliminaries, particularly, renormalization of the equation \eqref{eq:5mkdv_integrable} and introduction to short-time spaces and their properties. 	In Section \ref{sec:3}, multilinear and energy estimates are established, which are essential improvements of the formal work \cite{Kwak2018-2}. In Section \ref{sec:illposed}, we give counter-examples for $C^5$ ill-posedness of the renormalized equation. In Appendices, we provide the standard argument of the local well-posedness for the sake of self-containedness, and we also give concrete computations for error terms appearing in Section \ref{sec:illposed}. 

\subsection*{Acknowledgement}
C. K. was supported by Young Research Program, National Research Foundation of Korea(NRF) grant funded by the Korea government(MSIT) (No. RS-2023-00210210) and the Basic Science Research Program through the National Research Foundation of Korea (NRF) funded by the Ministry of Education (No. 2019R1A6A1A11051177). K. Lee was supported by  the Basic Science Research Program through the National Research Foundation of Korea (NRF) funded by the Ministry of Education (No. NRF-2022R1I1A1A01056408) and the National Research Foundation of Korea(NRF) grant funded by the Korea government(MSIT) (No. NRF-2019R1A5A1028324).

\section{Preliminaries}\label{sec:2}

\subsection{Notations}\label{sec:2.1}
$\;$

\noindent $\bullet$ For $x,y \in \R_+$, $x \lesssim y$ means that there exists $C>0$ such that $x \le Cy$, and $x \sim y$ means $x \lesssim y$ and $y\lesssim x$. 

\medskip

\noindent $\bullet$ Let $a_1,a_2,a_3,a_4 \in \R$. The quantities $a_{max} \ge a_{sub} \ge a_{thd} \ge a_{min}$ can be conveniently defined to be the maximum, sub-maximum, third-maximum and minimum values of $a_1,a_2,a_3,a_4$ respectively.

\medskip

\noindent $\bullet$ For $f \in \Sch(\R \times \T) $ we denote by $\wt{f}$ or $\ft (f)$ the Fourier transform of $f$ with respect to both spatial and time variables,
\[\wt{f}(\tau,n)=\frac{1}{\sqrt{2\pi}}\int_{\R}\int_{0}^{2\pi} e^{-i xn}e^{-it\tau}f(t,x)\; dx dt .\]
Moreover, we use $\ft_x$ (or $\wh{\;}$ ) and $\ft_t$ to denote the Fourier transform with respect to space and time variable respectively.

\medskip

\noindent $\bullet$ We denote $\bra{\cdot} = \sqrt{ 1+ |\cdot|^2}$. Let $\eta_0: \R \to [0,1]$ denote a smooth bump function supported in $ [-2,2]$ and equal to $1$ in $[-1,1]$ with the following property of regularities:
\begin{equation*}
	\partial_n^{j} \eta_0(n) = O(\eta_0(n)/\bra{n}^j), \hspace{1em} j=0,1,2,
\end{equation*}
as $n$ approaches end points of the support of $\eta$.

\medskip

\noindent $\bullet$ Let $\Z_+ = \Z \cap [0,\infty]$. For $k \in \Z_+$, we set
\[I_0 = \set{n \in \Z : |n| \le 2} \hspace{1em} \mbox{ and } \hspace{1em} I_k = \set{n \in \Z : 2^{k-1} \le |n| \le 2^{k+1}}, \hspace{1em} k \ge 1.\]

\medskip

\noindent $\bullet$ For $k,j \in \Z_+$, let
\[D_{k,j}=\{(\tau,n) \in \R \times \Z : \tau - \mu(n) \in I_j, n \in I_k \}, \hspace{2em} D_{k,\le j}=\cup_{l\le j}D_{k,l}.\]

\medskip

\noindent $\bullet$ For $k \in \Z_+ $, let 
\begin{equation}\label{eq:cut-off1}
	\chi_0(n) = \eta_0(n), \hspace{1em} \mbox{and} \hspace{1em} \chi_k(n) = \eta_0(n/2^k) - \eta_0(n/2^{k-1}), \hspace{1em} k \ge 1,
\end{equation}
which is supported in $I_k$, and
\[\chi_{[k_1,k_2]}=\sum_{k=k_1}^{k_2} \chi_k \quad \mbox{ for any} \ k_1 \le k_2 \in \Z_+ .\]
$\{ \chi_k \}_{k \in \Z_+}$ is the inhomogeneous decomposition function sequence to the frequency space. For $k\in \Z_+$, let $P_k$ denote the
operators on $L^2(\T)$ defined by $\widehat{P_kv}(n)=\chi_k(n)\wh{v}(n)$. For $l\in \Z_+$, let
\[P_{\le l}=\sum_{k \le l}P_k, \quad P_{\ge l}=\sum_{k \ge l}P_k.\]
For the time-frequency decomposition, we use the cut-off function $\eta_j$, but the same as $\eta_j(\tau-\mu(n)) = \chi_j(\tau-\mu(n))$.

\medskip

\noindent $\bullet$ For $Z = \R$ or $\Z$, let $\Gamma_k(Z)$ denote $(k-1)$-dimensional hyperplane by 
\[\set{\overline{x} = (x_1,x_2,...,x_k) \in Z^k : x_1 +x_2 + \cdots +x_k= 0}.\] 

\subsection{Reformulation of the main equations}\label{sec:reformulation}
As mentioned in Remark \ref{rem:integrable is enough}, in what follows, we only deal with \eqref{eq:5mkdv_integrable}. Taking the Fourier coefficient of \eqref{eq:5mkdv_integrable} with respect to the spatial variable, one has
\begin{equation}\label{eq:5mkdv1}
	\begin{split}
		\pt\wh{u}(n) - in^5\wh{u}(n) =&10i \sum_{n_1+n_2+n_3=n} \wh{u}(n_1)\wh{u}(n_2)n_3^3\wh{u}(n_3) \\
		&+10i \sum_{n_1+n_2+n_3=n} n_1\wh{u}(n_1)n_2\wh{u}(n_2)n_3\wh{u}(n_3)\\
		&+40i \sum_{n_1+n_2+n_3=n} \wh{u}(n_1)n_2\wh{u}(n_2)n_3^2\wh{u}(n_3)\\
		&+  30i \sum_{n_1+n_2+n_3+n_4+n_5=n} \wh{u}(n_1)\wh{u}(n_2)\wh{u}(n_3)\wh{u}(n_4)n_5\wh{u}(n_5) \\
	\end{split}
\end{equation}
Note that cubic resonance in the right-hand side of \eqref{eq:5mkdv1} appears  when $H =0$, where
\begin{equation*}
	\begin{aligned}
		H =& H(n_1,n_2,n_3)\\ 
		:=& (n_1+n_2+n_3)^5 - n_1^5 - n_2^5 - n_3^5 \\
		=& \frac52(n_1+n_2)(n_1+n_3)(n_2+n_3)(n_1^2+n_2^2+n_3^2+n^2).
	\end{aligned}
\end{equation*}

\begin{rem}
Due to $n_1^2+n_2^2+n_3^2+n^2 > 0$ for non-zero integers $n_1, n_2, n_3, n$, we know
\[
H=0 \quad \Leftrightarrow \quad (n_1+n_2)(n_1+n_3)(n_2+n_3)=0.
\]
\end{rem}
Moreover, one can also detect the quintic resonance $\norm{u}_{L^4}^4 n\wh{u}(n)$ when frequencies allow the following relation: 
\[	n_\ell + n_j + n_m +n_p = 0 \quad \mbox{for} \quad 1 \le \ell < j < m < p \le 5.\]
Most of resonances can be controlled by the first and the second conserved quantities in  \eqref{eq:hamiltonian}, thus we rewrite \eqref{eq:5mkdv1} as follows: 
\begin{equation}\label{eq:5mkdv2}
	\begin{split}
		\pt\wh{u}(n) - i(n^5 + d_1n^3 + d_2n)\wh{u}(n) =&~{}d_3 i \norm{u(t)}_{L^4}^4n\wh{u}(n) -20in^3|\wh{u}(n)|^2\wh{u}(n)\\
		&+10in \sum_{\N_{3,n}} \wh{u}(n_1)\wh{u}(n_2)n_3^2\wh{u}(n_3) \\
		&+10in \sum_{\N_{3,n}} \wh{u}(n_1)n_2\wh{u}(n_2)n_3\wh{u}(n_3)\\
&+6in \sum_{\N_{5,n}} \wh{u}(n_1)\wh{u}(n_2)\wh{u}(n_3)\wh{u}(n_4)\wh{u}(n_5),
	\end{split}
\end{equation}
where $d_1 = 10\norm{u_0}_{L^2}^2$, $d_2 = 10 (\norm{u_0}_{\dot{H^1}}^2 + \norm{u_0}_{L^4}^4)$, $d_3 =20$, 
\begin{equation*}
	\N_{3,n} = \left\{(n_1,n_2,n_3) \in \Z^3 : \begin{array}{ll} &n_1+n_2+n_3=n, \\ 
		&(n_1+n_2)(n_1+n_2)(n_2+n_3) \neq 0
	\end{array}
	\right\}.
\end{equation*}
and
\begin{align*}
\begin{aligned}
	&\N_{5,n} = \left\{
	(n_1,n_2,n_3,n_4,n_5) \in \Z^5 : \;\;\begin{aligned} &n_1+n_2+n_3+n_4+n_5=n,\\
		&n_\ell + n_j + n_m +n_p \neq 0,\\
		& \; 1 \le \ell <j<m<p \le 5
	\end{aligned}
	\right\}.
	\end{aligned}
\end{align*}

\begin{rem}
Even if we choose the non-integrable equation  \eqref{eq:5mkdv-gen} with \eqref{eq:coefficient constraint}, taking $\mathcal H_0$ and $\mathcal H_1$ in \eqref{eq:hamiltonian-2}, we still have \eqref{eq:5mkdv2} with different coefficients $d_j$, $j=1,2,3$. 
\end{rem}

\begin{rem}
	$\overline{\N}_{\ell,n}$, $\ell=3,5$, can be similarly defined with the frequency relation constraints 
\[\displaystyle \sum_{j=1}^\ell n_j + n = 0.\]
\end{rem}
Due to the presence of $d_3\norm{u(t)}_{L^4}^4n\wh{u}(n)$ in the right-hand side of \eqref{eq:5mkdv2}, we, similarly as in \cite{Kwak2018-2} further take the nonlinear transformation defined by 
\begin{equation}\label{eq:modified solution}
	v(t,x) := \NT(u)(t,x) :=  \frac{1}{\sqrt{2\pi}}\sum_{n \in \Z} e^{i(nx - 20n \int_0^t \norm{u(s)}_{L^4}^4 \; ds)}\wh{u}(t,n).
\end{equation}
Staffilani \cite{Staffilani1997} first introduced this-type of nonlinear transformation for the generalized KdV equations, and showed its bi-continuity in $H^s(\T)$ for $s > 1/2$. See also \cite{CKSTT2004} for its application.
With this new variable, we finally reformulate the equation \eqref{eq:5mkdv2} by
\begin{equation}\label{eq:5mkdv3}
	\begin{split}
		\pt\wh{v}(n) - i(n^5 + d_1n^3 + d_2n)\wh{v}(n)=&~{}-20in^3|\wh{v}(n)|^2\wh{v}(n)\\
		&+10in \sum_{\N_{3,n}} \wh{v}(n_1)\wh{v}(n_2)n_3^2\wh{v}(n_3) \\
		&+10in \sum_{\N_{3,n}} \wh{v}(n_1)n_2\wh{v}(n_2)n_3\wh{v}(n_3)\\
      &+6i n\sum_{\N_{5,n}} \wh{v}(n_1)\wh{v}(n_2)\wh{v}(n_3)\wh{v}(n_4)\wh{v}(n_5) \\
		=:&~{} \wh{N}_1(v) + \wh{N}_2(v) + \wh{N}_3(v) + \wh{N}_4(v)\\
		v(0,x) = u(0,x) =: v_0.\hspace{4.5em}&
	\end{split}
\end{equation}
We abbreviate  $N_j(v)$ to $N_j(u,v,w)$ for the cubic terms, $j=1,2,3$, or $N_4(v_1,v_2,v_3,v_4,v_5)$ for the quintic term in the sequel, and $\wh{N_j}$ denotes the Fourier transform of $N_j$.

\subsection{Function spaces}\label{sec:ftn-space}

For $s,b \in \R$, $X^{s,b}$-norm associated to \eqref{eq:5mkdv3} is given by

\[\norm{u}_{X^{s,b}}=\norm{\bra{ \tau - \mu(n)}^b\bra{n}^s \ft(u)}_{L_{\tau}^2(\R;\ell_n^2(\Z))},\]
where 
\[
\mu(n) = n^5 + d_1n^3 + d_2n.
\]

The $X^{s,b}$ space is now widely used to study the low-regularity theory for the dispersive equations. The Fourier restriction norm method was first introduced in its current form by Bourgain \cite{Bourgain1993} and further developed by Kenig, Ponce and Vega \cite{KPV1996} and Tao \cite{Tao2001}. Since then, it becomes a standard way to study semilinear dispersive equations.

For $k \in \Z_+$, we define the $X^{0,\frac12,1}$-type space $X_k$ for frequency localized functions,
\begin{align*}
	X_k=\left\{
	\begin{aligned}
		f\in L^2(\R \times \Z): f(\tau,n) \mbox{ is supported in } \R \times I_k  \; \mbox{ and }\;
		\norm{f}_{X_k}<\infty
	\end{aligned}
	\right\}
\end{align*}
where
$$
\norm{f}_{X_k}:=\sum_{j=0}^\infty 2^{j/2}\beta_{j,k}\norm{\eta_j(\tau-\mu(n))\cdot f(\tau,n)}_{L_{\tau}^2\ell_n^2}
$$
with
\begin{equation}\label{eq:weight}
\beta_{j,k} = \left\{ \begin{aligned} 1  \qquad&  \;\;(k=0), \\ 1+ 2^{\frac14(j-5k)} & \;\;(k\ge 1). \end{aligned} \right.
\end{equation}
\begin{rem}\label{rem:weight}
The weight introduced in \eqref{eq:weight} can be generalized as $1+2^{\gamma(j-5k)}$ for $k \ge 1$, where $0 < \gamma \le \frac14$. Non-negativity of $\gamma$ is essential to handle logarithmic divergence appearing in the energy estimates as mentioned. Note that the heavier weight, the worse nonlinear estimates. The possible range of $\gamma$ allowing nonlinear estimates in $H^2(\T)$ level is $0 < \gamma \le \frac14$, and thus we fix $\gamma = \frac14$ for computational simplicity.

\begin{rem}
As mentioned in Section \ref{sec:idea}, the weight handles the logarithmic divergence issue in the energy estimates, particularly, the high-low interaction case when the low frequency mode has the largest modulation. Under the periodic setting, however, high and very low ($< 1$) frequency interactions cannot appear so that the non-singular weight \eqref{eq:weight} is itself applicable. This is a big contrast of the choice of the weight with the work by Guo, Kwon, and the first author \cite{GKK2013}.
\end{rem}

%
%
%
\end{rem}

As in \cite{IKT2008}, at frequency $2^k$ we use the $X_k$-norm, uniformly on the $2^{-2k}$ time scale (see Section \ref{sec:idea} for the motivation of such a time scale). Precisely, for $k\in \Z_+$, we define function spaces
\begin{eqnarray*}
	&& F_k=\left\{
	\begin{array}{l}
		f\in L^2(\R \times \T): \widehat{f}(\tau,n) \mbox{ is supported in } \R \times I_k \mbox{ and } \\
		\norm{f}_{F_k}=\sup\limits_{t_k\in \R}\norm{\ft[f\cdot\eta_0(2^{2k}(t-t_k))]}_{X_k}<\infty
	\end{array}
	\right\},
	\\
	&&N_k=\left\{
	\begin{array}{l}
		f\in L^2(\R \times \T): \widehat{f}(\tau,n) \mbox{ is supported in } \R \times I_k \mbox{ and }  \\
		\norm{f}_{N_k}=\sup\limits_{t_k\in \R}\norm{(\tau-\mu(n)+i2^{2k})^{-1}\ft[f\cdot\eta_0(2^{2k}(t-t_k))]}_{X_k}<\infty
	\end{array}
	\right\}.
\end{eqnarray*}
In a standard way, we localize these spaces in time $T\in (0,1]$ as
\begin{align*}
	F_k(T)=&\left\{f\in C([-T,T]:L^2(\T)): \norm{f}_{F_k(T)}=\inf_{\wt{f}=f \mbox{ in } [-T,T] \times \T }\norm{\wt f}_{F_k}\right\},\\
	N_k(T)=& \left\{f\in C([-T,T]:L^2(\T)): \norm{f}_{N_k(T)}=\inf_{\wt{f}=f \mbox{ in } [-T,T] \times \T }\norm{\wt f}_{N_k}\right\}.
\end{align*}
For $s\geq 0$ and $T\in (0,1]$, we collect dyadic pieces in a Littlewood-Paley argument to define function spaces for solutions and
nonlinear terms as
\begin{eqnarray*}
	&&F^{s}(T)=\left\{ u: \norm{u}_{F^{s}(T)}^2=\sum_{k=0}^{\infty}2^{2sk}\norm{P_k(u)}_{F_k(T)}^2<\infty \right\},
	\\
	&&N^{s}(T)=\left\{ u: \norm{u}_{N^{s}(T)}^2=\sum_{k=0}^{\infty}2^{2sk}\norm{P_k(u)}_{N_k(T)}^2<\infty \right\}.
\end{eqnarray*}
We finally define the energy space analogously as above, i.e., for $s\geq 0$ and $u\in C([-T,T]:H^\infty(\T))$
\begin{eqnarray*}
	\norm{u}_{E^{s}(T)}^2=\norm{P_{0}(u(0))}_{L^2}^2+\sum_{k\geq 1}\sup_{t_k\in [-T,T]}2^{2sk}\norm{P_k(u(t_k))}_{L^2}^2.
\end{eqnarray*}

\begin{lem}[Properties of $X_k$]\label{lem:prop of Xk}
	Let $k, l\in \Z_+$ with $l \le 5k$ and $f_k\in X_k$. Then
	\begin{equation}\label{eq:prop1}
		\begin{aligned}
			&\sum_{j=l+1}^\infty 2^{j/2}\beta_{j,k}\normo{\eta_j(\tau-\mu(n))\int_{\R}|f_k(\tau',n)|2^{-l}(1+2^{-l}|\tau-\tau'|)^{-4}d\tau'}_{L_{\tau}^2\ell_n^2}\\
			&\qquad+2^{l/2}\beta_{l,k}\normo{\eta_{\leq l}(\tau-\mu(n)) \int_{\R}|f_k(\tau',n)| 2^{-l}(1+2^{-l}|\tau-\tau'|)^{-4}d\tau'}_{L_{\tau}^2\ell_n^2}\\
&\lesssim\norm{f_k}_{X_k}.					\end{aligned}
	\end{equation}
	In particular, if $t_0\in \R$ and $\gamma\in \Sch(\R)$, then
	\begin{eqnarray}\label{eq:prop2}
		\norm{\ft[\gamma(2^l(t-t_0))\cdot \ft^{-1}(f_k)]}_{X_k}\lesssim
		\norm{f_k}_{X_k}.
	\end{eqnarray}
	Moreover, from the definition of $X_k$-norm,
	\[\normo{\int_{\R}|f_k(\tau',n)|\; d\tau'}_{\ell_n^2} \lesssim \norm{f_k}_{X_k}.\]
\end{lem}
\begin{proof}
	We refer to \cite{GKK2013} for the proof. 
\end{proof}

\begin{rem}
	The restriction $l \le 5k$ in Lemma \ref{lem:prop of Xk} is required to control the weight $\beta_{j,k}$ in the left-hand side of \eqref{eq:prop1}. See this with Lemma 2.1 in \cite{Kwak2018-2} for the comparison. 
\end{rem}

As in \cite{IKT2008}, for any $k\in \Z_+$ we define the set $S_k$ of $k$-\emph{acceptable} time multiplication factors 
\[
S_k=\left\{m_k:\R\rightarrow \R: \norm{m_k}_{S_k}=\sum_{j=0}^{10} 2^{-2jk}\norm{\partial^jm_k}_{L^\infty}< \infty \right\}.
\] 
 Using the definitions and \eqref{eq:prop2}, Direct estimates show that for any $s\geq 0$ and $T\in (0,1]$
\[\begin{cases}
	\normb{\sum\limits_{k\in \Z_+} m_k(t)\cdot P_k(u)}_{F^{s}(T)}\lesssim (\sup_{k\in \Z_+}\norm{m_k}_{S_k})\cdot \norm{u}_{F^{s}(T)};\\
	\normb{\sum\limits_{k\in \Z_+} m_k(t)\cdot P_k(u)}_{N^{s}(T)}\lesssim (\sup_{k\in \Z_+}\norm{m_k}_{S_k})\cdot \norm{u}_{N^{s}(T)};\\
	\normb{\sum\limits_{k\in \Z_+} m_k(t)\cdot P_k(u)}_{E^{s}(T)}\lesssim (\sup_{k\in \Z_+}\norm{m_k}_{S_k})\cdot \norm{u}_{E^{s}(T)}.
\end{cases}\]

\subsection{Block estimates and trilinear nonlinear estimates}
This section devotes to providing key ingredients for nonlinear and energy estimates, which will be addressed in Section \ref{sec:3}. 

Let $\zeta_i = \tau_i - \mu(n_i)$. For compactly supported functions $f_i \in L^2(\R \times \Z)$, $i=1,2,3,4$, we define 
\begin{align*}
&J(f_1,f_2,f_3,f_4) \\
&= \sum_{n_4, \N_{3,n_4}}\int_{\overline{\zeta}\in \Gamma_4(\R)}f_1(\zeta_1,n_1)f_2(\zeta_2,n_2)f_3(\zeta_3,n_3)f_4(\zeta_4 + G(n_1,n_2,n_3),n_4),
\end{align*}
where  $\overline{\zeta} = (\zeta_1,\zeta_2,\zeta_3,-\zeta_4-G(n_1,n_2,n_3))$ with the resonance function   
\[
G(n_1,n_2,n_3) = \mu(n_1 + n_2 + n_3)- \mu(n_1) - \mu(n_2) - \mu(n_3).
\] 
The function $G$ plays an crucial role in $X^{s,b}$ analysis. From the identities
\[n_1+n_2+n_3 = n_4\]
and
\[\zeta_1+\zeta_2+\zeta_3 = \zeta_4 + G(n_1,n_2,n_3)\]
on the support of $J(f_1,f_2,f_3,f_4)$, we see that $J(f_1,f_2,f_3,f_4)$ vanishes unless
\begin{equation}\label{eq:support property}
	\begin{array}{c}
		2^{k_{max}} \sim 2^{k_{sub}}\\
		2^{j_{max}} \sim \max(2^{j_{sub}}, |G|),
	\end{array}
\end{equation}
where $|n_i| \sim 2^{k_i}$ and $|\zeta_i| \sim 2^{j_i}$, $i=1,2,3,4$. Note that a direct computation gives
\begin{equation}\label{eq:modulation}
	|G| = \frac52(n_1+n_2)(n_2+n_3)(n_3+n_1)\left(n_1^2+n_2^2 + n_3^2 + (n_1+n_2+n_3)^2 + \frac{6d_1}{5}\right).
\end{equation}
Moreover, a direct computation in addition to a fact that the convolution operator allows the commutative law, ensures
\[|J(f_1,f_2,f_3,f_4)| = |J(f_2,f_1,f_3,f_4)| = |J(f_3,f_2,f_1,f_4)| = |J(f_1^*,f_2^*,f_4,f_3)|.\]

\begin{lem}[Block estimates, Lemma 4.1 in \cite{Kwak2018-2}]\label{lem:tri-L2}
	Let $k_i, j_i\in \Z_+$, $i=1,2,3,4$. Let $f_{i} \in L^2(\R \times \Z) $ be nonnegative functions supported in $D_{k_i,j_i}$:
	
	\noi(a) For any $k_i,j_i \in \Z_+$, $i=1,2,3,4$, we
	have
	\begin{align}\label{eq:tri-block estimate-a1}
		J(f_{1},f_{2},f_{3},f_{4}) \lesssim 2^{(j_{min}+j_{thd})/2}2^{(k_{min}+k_{thd})/2}\prod_{i=1}^4 \|f_{i}\|_{L^2}.
	\end{align}
	
	\noi(b) Let $k_{thd} \le k_{max}-10$. Then, the following holds:
	
 If $(k_i,j_i) = (k_{thd},j_{max})$ for $i=1,2,3,4$, we have
\begin{align}\label{eq:tri-block estimate-b1.2}
	J(f_{1},f_{2},f_{3},f_{4}) \lesssim 
		 2^{(j_1+j_2+j_3+j_4)/2}2^{-(j_{sub}+j_{max})/2}2^{k_{thd}/2}  \prod_{i=1}^4 \|f_{i}\|_{L^2}. 
\end{align}
Otherwise, we get
\begin{align*}
	J(f_{1},f_{2},f_{3},f_{4}) \lesssim 2^{(j_1+j_2+j_3+j_4)/2}2^{-(j_{sub}+j_{max})/2}2^{k_{min}/2}  \prod_{i=1}^4 \|f_{i}\|_{L^2}.
\end{align*}
\end{lem} 

\begin{rem}\label{rem:duality}
	By the duality, $J(f_1,f_2,f_3,f_4)$ in Lemma \ref{lem:tri-L2} can be replaced by $\norm{\mathbf{1}_{D_{k_4,j_4}}(f_{1}\ast f_{2}\ast f_{3})}_{L^2}$. See Corollary 4.2 in \cite{Kwak2018-2}.
\end{rem}

The following lemma gives a significant different estimates by exploiting the weighted function spaces compared to those in previous results \cite{Kwak2018-2}.

\begin{lem}\label{lem:energy1-1}
	Let $T \in (0,1]$, $k_i\in \Z_+$, and $v_i \in F_{k_i}(T)$, $i=1,2,3,4$. We further assume $k_1 \le k_2 \le k_3 \le k_4$ with $k_4 \ge 10$. Then we have
%
	
%
%
%
\begin{align*}
		\left| \sum_{n_4,\overline{\N}_{3,n_4}} \int_0^T\wh{v}_1(n_1)\wh{v}_2(n_2)\wh{v}_3(n_3)\wh{v}_4(n_4) \; dt\right| \lesssim   C(\textbf{k}) \prod_{i=1}^4\norm{v_i}_{F_{k_i}(T)}
\end{align*}
where
\begin{align*}\label{eq:energy1-1.}
C(\textbf{k})= \left\{ \begin{aligned}
	2^{k_4/2}           &\qquad \mbox{ if } |k_1 -k_4| \le 5,\\
	2^{-k_4}2^{k_1/2}   & \qquad\mbox{ if } |k_2-k_4| \le 5 \mbox{ and } k_1 \le k_4 -10,\\	
	2^{-k_4}2^{k_1/2}   &\qquad \mbox{ if } |k_3-k_4| \le 5, k_2 \le k_4 -10, \mbox{ and } |k_1-k_2| \le 5,\\	
	2^{-(k_4-k_2)}2^{-k_4}   &\qquad \mbox{ if } |k_3-k_4| \le 5, k_2 \le k_4 -10, \mbox{ and } k_1 \le k_2 -10,\\	
		\end{aligned}    \right.
\end{align*}
with $\textbf{k}=(k_1,k_2,k_3,k_4)$.
\end{lem}
 
\begin{proof}
	Comparing with Lemma 6.4 in \cite{Kwak2018-2}, the only difference appears in the case when $|k_3-k_4| \le 5, k_2 \le k_4 -10$, and $k_1 \le k_2 -10$, thus we focus on this case.

	We follow the standard extensions $\wt{v}_i \in F_{k_i}$ so that $\norm{\wt{v}_{i}}_{F_{k_i}} \le 2 \norm{v_i}_{F_{k_i}(T)}$, $i=1,2,3,4$. Let $\rho : \R \to [0,1]$ be a smooth partition of unity function with $\sum_{m \in \Z}\rho^4(x-m) = 1$, $x \in \R$. Then, we obtain 
	\begin{equation}\label{eq:energy1-1.5}
		\begin{aligned}
			&\Big| \sum_{n_4,\overline{\N}_{3,n_4}} \int_0^T\wh{\wt{v}}_1(n_1)\wh{\wt{v}}_2(n_2)\wh{\wt{v}}_3(n_3)\wh{\wt{v}}_4(n_4) \;dt \Big|\\
			&\lesssim \sum_{|m| \lesssim 2^{2k_4}} \Big| \sum_{n_4,\overline{\N}_{3,n_4}} \int_{\R}\left(\rho(2^{2k_4}t-m)\mathbf{1}_{[0,T]}(t)\wh{\wt{v}}_1(n_1)\right)  \\
& \hspace{6em}\cdot \left(\rho(2^{2k_4}t-m)\wh{\wt{v}}_2(n_2)\right)\cdot \left(\rho(2^{2k_4}t-m)\wh{\wt{v}}_3(n_3)\right) \cdot \left(\rho(2^{2k_4}t-m)\wh{\wt{v}}_4(n_4)\right) \;dt \Big|
		\end{aligned}
	\end{equation}

		Then, for $m$ such that $\emptyset \neq \supp \rho(2^{2k_4}t-m) \cap [0,T] \neq \supp \rho(2^{2k_4}t-m)$, one can readily handle the right-hand side of \eqref{eq:energy1-1.5} (see \cite{guo2012} for the details), thus we only consider the case 
			\begin{equation}\label{eq:A^c}
				\rho(2^{2k_4}t-m)\mathbf{1}_{[0,T]}(t)\wh{\wt{v}}_1(n_1) = \rho(2^{2k_4}t-m)\wh{\wt{v}}_1(n_1).
				\end{equation} 
			Note that the number of $m$ for which \eqref{eq:A^c} is vaild is $\lesssim 2^{2k_4}$.
	Let $f_{k_i} = \ft[\rho(2^{2k_4}t-m)\wh{\wt{v}}_i(n_i)]$ and $f_i = \eta_{j_i}(\tau - \mu(n))f_{k_i}$, $i=1,2,3,4$. By parseval's identity and \eqref{eq:prop1}, the right-hand side of \eqref{eq:energy1-1.5} is dominated by
	\begin{equation}\label{eq:energy1-1.6}
	 \sup_m 2^{2k_4} \sum_{j_1,j_2,j_3,j_4 \ge 2k_4} |J(f_1,f_2,f_3,f_4)|,
	\end{equation}
where the supremum is taken over $m$ satisfying \eqref{eq:A^c}.

	Using the frequency relations $|k_3 - k_4| \le 5$, $k_2 \le k_4 -10$, and $k_1  \le k_2 - 10$ in addition to \eqref{eq:support property} and \eqref{eq:modulation}, we see that $j_{max} \ge 4k_4 + k_2$. The worst bound of $|J(f_{1},f_{2},f_{3},f_{4})|$ should appear when $j_2 = j_{max}$ and $j_{sub} \le 4k_4 + k_2 - 10$ hold. By \eqref{eq:tri-block estimate-b1.2}, we have 

	\begin{align*}
		\eqref{eq:energy1-1.6} &\lesssim 2^{2k_4} \sum_{\substack{j_1,j_2,j_3,j_4 \ge 2k_4\\ j_2 = j_{max}\\j_{sub} \le 4k_4 + k_2 - 10}}2^{(j_1+j_2+j_3+j_4)/2}2^{-(j_{max}+j_{sub})/2}2^{k_2/2} \norm{f_2}_{L_{\tau}^2\ell_n^2}\prod_{i=1,3,4} \norm{f_i}_{L_{\tau}^2\ell_n^2}\\
&\les  2^{-(k_4-k_2)} 2^{-k_4} \left(\sum_{j_2 \ge 4k_4 +k_2}2^{j_2/2} \beta_{j_2,k_2}\norm{f_2}_{L_{\tau}^2\ell_n^2}\right)\prod_{i=1,3,4} \norm{v_{i}}_{F_{k_i}(T)}\\
		&\lesssim 2^{-(k_4-k_2)} 2^{-k_4} \prod_{i=1}^4 \norm{v_{i}}_{F_{k_i}(T)},
	\end{align*}
which proves Lemma \ref{lem:energy1-1}.
\end{proof}

For commutator estimates, we define  $\psi(n):=n\chi'(n)$ and $\psi_k(n) = \psi(2^{-k}n)$ for $k \ge 1$, where $\chi$ is defined in \eqref{eq:cut-off1}. Then we get
\[
\psi_k(n) = n\chi_k'(n).
\]

\begin{rem}\label{rem:even real}
It is easy to see that $\chi_k$ and $\psi_k$ are even and real-valued functions. These properties will be used to remove quintic resonances appeared in modified energies, see Remark \ref{rem:resonant2}.
\end{rem}
\begin{rem}
A  new cut-off function $\psi_k$ is required to be defined only for the second-order Mean Value Theorem in the commutator estimates (see Lemma \ref{lem:commutator1}), otherwise, it plays only a role of frequency support.
\end{rem}

\begin{lem}\label{lem:commutator1}
	Let $T \in (0,1]$, $k,k_1,k_2 \in \Z_+$ satisfying $k_1,k_2 \le k -10$, $u_i \in F_{k_i}(T)$, $i=1,2$, and $v \in F^0(T)$. Then, we have
\begin{equation*}
		\begin{aligned}
			\Bigg|\sum_{n,\overline{\N}_{3,n}}&\int_0^T  \chi_k(n)n[\chi_{k_1}(n_1)\wh{u}_1(n_1)\chi_{k_2}(n_2)\wh{u}_2(n_2)n_3^2\wh{v}(n_3)]\chi_k(n)\wh{v}(n) \;dt\\
			&\hspace{-2em}+ \frac12\sum_{n,\overline{\N}_{3,n}}\int_0^T (n_1+n_2)\chi_{k_1}(n_1)\wh{u}_1(n_1)\chi_{k_2}(n_2)\wh{u}_2(n_2)\chi_k(n_3)n_3\wh{v}(n_3)\chi_k(n)n\wh{v}(n) \;dt\\
			&\hspace{-2em}- \sum_{n,\overline{\N}_{3,n}}\int_0^T (n_1+n_2)\chi_{k_1}(n_1)\wh{u}_1(n_1)\chi_{k_2}(n_2)\wh{u}_2(n_2)\psi_k(n_3)n_3\wh{v}(n_3)\chi_k(n)n\wh{v}(n) \;dt \Bigg|\\
			&~{} \hspace{6em}\lesssim  C(k,k_1,k_2) 
			 \norm{P_{k_1}u_1}_{F_{k_1}(T)}\norm{P_{k_2}u_2}_{F_{k_2}(T)}\sum_{|k-k'|\le 5} \norm{P_{k'}v}_{F_{k'}(T)}^2
		\end{aligned}
	\end{equation*}
and
\begin{align*}
	\begin{aligned}
		\Bigg|\sum_{n,\overline{\N}_{3,n}}&\int_0^T \chi_k(n)[(n_1+n_2)\chi_{k_1}(n_1)\wh{u}_1(n_1)\chi_{k_2}(n_2)\wh{u}_2(n_2)n_3\wh{v}(n_3)]\chi_k(n)n\wh{v}(n)\; dt\\
		&\hspace{-2em}- \sum_{n,\overline{\N}_{3,n}}\int_0^T (n_1+n_2)\chi_{k_1}(n_1)\wh{u}_1(n_1)\chi_{k_2}(n_2)\wh{u}_2(n_2)\chi_k(n_3)n_3\wh{v}(n_3)\chi_k(n)n\wh{v}(n)\;dt \Bigg|\\
		&~{} \hspace{7em}\lesssim  C(k,k_1,k_2) 
			 \norm{P_{k_1}u_1}_{F_{k_1}(T)}\norm{P_{k_2}u_2}_{F_{k_2}(T)}\sum_{|k-k'|\le 5} \norm{P_{k'}v}_{F_{k'}(T)}^2,
	\end{aligned}
\end{align*}
	where
\[
C(k,k_1,k_2) = \begin{cases}
	2^{2k_2}  & \mbox{if }\;\; |k_1 -k_2| \le 5,\\
	2^{-(k-k_2)}2^{2k_2}& \mbox{if }\;\; k_1 \le k_2 -10.
\end{cases}
\]
\end{lem}
\begin{rem}
	Commutator estimates above are slightly improved compared with Lemma 6.5 in \cite{Kwak2018-2} thanks to the refined version of Lemma \ref{lem:energy1-1}, but the proof of Lemma \ref{lem:commutator1} is analogous.
\end{rem}

\section{Refined nonlinear and energy estimates}\label{sec:3}

\subsection{Refined nonlinear estimates}\label{sec:nonli}
In this section, we introduce the following nonlinear estimates.
\begin{prop}\label{prop:nonlinear1}
	(a) If $s \ge 1$\footnote{Compared with the previous nonlinear estimates in \cite{Kwak2018-2}, the weight condition $\beta_{j,k}$ exhibits the worse estimates in terms of regularity. Considering the function spaces $N^s(T),\,F^s(T)$ without the weight condition, we have \eqref{eq:nonlinear1} up to $s> \frac12$.}, $T \in (0,1]$ and $u,v,w, v_\ell \in F^s(T)$, $\ell=1,2,3,4,5$, then
	\begin{equation}\label{eq:nonlinear1}
		\begin{aligned}
			\sum_{m=1,2,3} \norm{N_{m}(u,v,w)}_{N^s(T)} &+ \norm{N_4(v_1,v_2,v_3,v_4,v_5)}_{N^s(T)} \\
			&\hspace{1em}\lesssim \norm{u}_{F^s(T)}\norm{v}_{F^s(T)}\norm{w}_{F^s(T)} + \prod_{\ell=1}^{5}\norm{v_\ell}_{F^s(T)}.
		\end{aligned}
	\end{equation}
	
	(b) If $T \in (0,1]$, $u,v, v_\ell \in F^2(T)$, $\ell=1,2,3,4$, and $w, v_5 \in F^0(T)$,  then
	\begin{equation}\label{eq:nonlinear2}
		\begin{aligned}
			\sum_{m=1,2,3}& \norm{N_{m}(u,v,w)}_{N^0(T)} + \norm{N_4(v_1,v_2,v_3,v_4,v_5)}_{N^0(T)}\\
			&\hspace{3em} \lesssim \norm{u}_{F^{1+}(T)}\norm{v}_{F^{2}(T)}\norm{w}_{F^0(T)} + \prod_{\ell=1}^{4}\norm{v_\ell}_{F^{1+}(T)}\norm{v_5}_{F^0(T)}.
		\end{aligned}
	\end{equation}
\end{prop}

Comparing with Proposition 5.11 in \cite{Kwak2018-2}, nonlinear estimates get more delicate from the fact that the derivative loss arising in the (cubic) high-low interactions where the output component has low frequency, but large modulation. Thus, we only focus on these regimes.
%

\begin{rem}\label{rem:NonRed}
Obviously, the case in which two derivatives are taken in the high frequency modes exhibits the most delicate. Thus, it suffices to control 
\begin{equation}\label{eq:nonres_0}
	\norm{P_{k_4}N_2(P_{k_1}u,P_{k_2}v,P_{k_3}w)}_{N_{k_4}},
\end{equation}
where $k_4 \ll k_3 = k_{max}$. Since $N_{k_4}$-norm is taken on the time intervals of length $2^{-2k_4}$, while each $F_{k_i}$-norm is defined on shorter (time) intervals of length $2^{-2k_i}$, $i=1,2,3$, we further decompose the time interval into $2^{2k_3 - 2k_4}$ sub-intervals of length $2^{-2k_3}$. For this purpose, let take $\sigma: \R \to [0,1]$, as a partition of unity, supported in $[-1,1]$ with $ \ \sum\limits_{m\in \Z} \sigma^3(x-m) \equiv 1$. Then we apply the definition of $N_{k_4}$-norm to \eqref{eq:nonres_0}, obtaining the following term
\begin{equation*}
	\begin{split}
		\sup_{t_k\in \R}&2^{k_4}2^{2k_3}\Big\|(\tau_4-\mu(n_4) +i 2^{2k_4})^{-1}\mathbf{1}_{I_{k_4}}\\
		&\qquad\qquad\times  \sum_{|m| \le C 2^{2k_3-2k_4}} \ft[\eta_0(2^{2k_4}(t-t_k))\sigma (2^{2k_3}(t-t_k)-m)P_{k_1}u]\\
		&\hspace{10em}\ast \ft[\eta_0(2^{2k_4}(t-t_k))\sigma (2^{2k_3}(t-t_k)-m)P_{k_2}v]\\
		&\hspace{10em}\ast \ft[\eta_0(2^{2k_4}(t-t_k))\sigma (2^{2k_3}(t-t_k)-m)P_{k_3}w]\Big\|_{X_{k_4}}.
	\end{split}
\end{equation*}
Let us abbreviate 
\[
\begin{aligned}
&u_{k_1} := \ft\left[\eta_0(2^{2k_4}(t-t_k))\sigma (2^{2k_3}(t-t_k))P_{k_1}u\right], \\
&v_{k_2} := \ft\left[\eta_0(2^{2k_4}(t-t_k))\sigma (2^{2k_3}(t-t_k))P_{k_2}v\right], \\
&w_{k_3} := \ft\left[\eta_0(2^{2k_4}(t-t_k))\sigma (2^{2k_3}(t-t_k))P_{k_3}w\right].
\end{aligned}
\]
We further decompose the modulation of $u_{k_1},v_{k_2}$ and $w_{k_3}$ into dyadic pieces as $f_{1}(\tau_1,n_1) = u_{k_1}(\tau_1,n_1)\eta_{j_1}(\tau_1 - \mu(n_1))$, $f_{2}(\tau_2,n_2) = v_{k_2}(\tau_2,n_2)\eta_{j_2}(\tau_2 - \mu(n_2))$ and $f_{3}(\tau_3,n_3) = w_{k_3}(\tau_3,n_3)\eta_{j_3}(\tau_3 - \mu(n_3))$, respectively.
\end{rem}

\begin{lem}[High-high-high $\Rightarrow$ low]\label{lem:nonres3} Let $k_3 \ge 20$, $|k_1-k_3|, |k_2-k_3| \le 5$ and $k_4 \le k_3-10$. Then, we have
\begin{align}\label{eq:nonres3-1}
\norm{P_{k_4}N_2(P_{k_1}u,P_{k_2}v,P_{k_3}w)}_{N_{k_4}} \lesssim 2^{\frac94 k_3}2^{- \frac74 k_4}\norm{P_{k_1}u}_{F_{k_1}}\norm{P_{k_2}v}_{F_{k_2}}\norm{P_{k_3}w}_{F_{k_3}}.
\end{align}
\end{lem}
\begin{proof}
	By \eqref{eq:tri-block estimate-a1} and Lemma \ref{lem:prop of Xk}, we get
	\[\begin{aligned} \mbox{LHS of }\eqref{eq:nonres3-1} \lesssim &2^{4k_3}2^{-k_4}\sum_{j_4 \ge 0} \frac{2^{j_4/2}}{\max(2^{j_4},2^{2k_4})}\beta_{j_4,k_4}\\
		&\qquad\times \sum_{j_1,j_2,j_3 \ge 2k_3} 2^{(j_{min}+j_{thd})/2}2^{(k_3+k_4)/2}\norm{f_1}_{L_{\tau}^2\ell_n^2}\norm{f_2}_{L_{\tau}^2\ell_n^2}\norm{f_3}_{L_{\tau}^2\ell_n^2}.
	\end{aligned}\]
	Note that \eqref{eq:modulation} and \eqref{eq:support property} yield
\begin{equation}\label{eq:high modulation}
2^{j_{max}} \gtrsim |G| \gtrsim |(n_1+n_2)(n_1+n_3)(n_2+n_3)|(n_1^2+n_2^2+n_3^2+n^2).
\end{equation}
Then, we have
\begin{align}\label{eq:case1}	
	(\max(2^{j_4},2^{2k_4}))^{-1}2^{(j_{min}+j_{thd})/2} \lesssim 
	\begin{cases}
		2^{(j_1+j_2+j_3)/2}2^{-7k_3/2}2^{-k_4}, \quad &0 \le j_4 < 2k_4,\\
		2^{(j_1+j_2+j_3)/2}2^{-7k_3/2}2^{-k_4}, \quad &2k_4 \le j_4 < 2k_3,\\
		2^{-j_4}2^{(j_1+j_2+j_3)/2}2^{-5k_3/2}, \quad &2k_3 \le j_4, \;\; j_4 \neq j_{max},\\
		2^{-j_4}2^{(j_1+j_2+j_3)/2}2^{-k_3}, \quad &2k_3 \le j_4, \;\; j_4 = j_{max}.
	\end{cases}
\end{align}
Among \eqref{eq:case1}, the most delicate bound appears when $j_4 =j_{\max} \ge 5 k_3$, Otherwise, we can handle the effect which weight gives by $2^{-j_{sub}/2}$. If $j_4 = j_{\max}$, we have
	\begin{align*}
		&2^{4k_3}2^{-k_4}\sum_{j_4 \ge  5k_3}\frac{2^{j_4/2}}{\max(2^{j_4},2^{2k_4})}\beta_{j_4,k_4}\sum_{j_1,j_2,j_3 \ge 2k_3}2^{(j_{min}+j_{thd})/2}2^{(k_3+k_4)/2}\prod_{i=1}^{3}\norm{f_i}_{L_{\tau}^2\ell_n^2}\\
		&\lesssim 2^{4k_3}2^{-k_4} 2^{-\frac54(k_3+k_4)}\sum_{j_1,j_2,j_3 \ge 2k_3}2^{(j_1+j_2+j_3)/2}2^{-k_3}2^{(k_3+k_4)/2}\prod_{i=1}^{3}\norm{f_{i}}_{L_{\tau}^2\ell_n^2}\\
		&\lesssim 2^{\frac94 k_3}2^{- \frac74 k_4}\norm{P_{k_1}u}_{F_{k_1}}\norm{P_{k_2}v}_{F_{k_2}}\norm{P_{k_3}w}_{F_{k_3}}.
	\end{align*}
	This finishes the proof of  Lemma \ref{lem:nonres3}.
\end{proof}

\begin{lem}[High-high-low $\Rightarrow$ low]\label{lem:nonres5}
	Let $k_3 \ge 20$, $|k_2-k_3| \le 5$ and $k_1,k_4 \le k_3 -10$. Then, we have
\begin{equation}\label{eq:nonres5-0}
	\norm{P_{k_4}N_2(P_{k_1}u,P_{k_2}v,P_{k_3}w)}_{N_{k_4}} \lesssim 2^{2k_3}C(k_1,k_4)\norm{P_{k_1}u}_{F_{k_1}}\norm{P_{k_2}v}_{F_{k_2}}\norm{P_{k_3}w}_{F_{k_3}},
	\end{equation}
	where
	\[C(k_1,k_4)= 
	\begin{cases}
		2^{-2 k_4} \hspace{1em}  &   {\rm if }\;\;  k_1 \le k_4 -10, \\
		2^{-\frac74 k_4}2^{-\frac14 k_1} \hspace{1em}  & {\rm if }\;\; k_4 \le k_1 - 10,\\
		2^{-\frac74 k_4} \hspace{1em}  &{\rm if }\;\; |k_1 -k_4| <10.
	\end{cases}
	\]
\end{lem}

\begin{proof}
	As we observed in the proof of Lemma \ref{lem:nonres3}, the left-hand side of \eqref{eq:nonres5-0} is bounded by
	\begin{equation}\label{eq:nonres5-1}
		2^{4k_3}2^{-k_4}\sum_{j_4 \ge 0}\frac{2^{j_4/2}\beta_{j_4,k_4}}{\max(2^{j_4},2^{2k_4})}\sum_{j_1,j_2,j_3 \ge 2k_3}\norm{\mathbf{1}_{D_{k_4,j_4}}\cdot(f_1 \ast f_2 \ast f_3)}_{L_{\tau_4}^2\ell_{n_4}^2}.
	\end{equation}

	\textbf{Case (a): $k_1 \le k_4-10$.} In this case, we also have $j_{max} \ge 4k_3 + k_4$ by \eqref{eq:modulation} and \eqref{eq:support property}. Note that there is no effect of the weight when $j_4 \le 5k_4$. Thus we reduce the summation over $j_4$ on $j_4 \ge 5k_4$. By Lemma \ref{lem:tri-L2} (b), we have
\[\eqref{eq:nonres5-1} \lesssim 2^{4k_3}2^{-k_4}\sum_{5k_4 \le j_4} \beta_{j_4,k_4}\sum_{j_1,j_2,j_3 \ge 2k_3}2^{(j_1+j_2+j_3)/2} 2^{-(j_{max}+j_{sub})/2}2^{k_4/2}\prod_{i=1}^{3}\norm{f_{i}}_{L_{\tau}^2\ell_n^2}.\]
Here $2^{k_4/2}$ can be replaced by $2^{k_1/2}$, when $j_4 \neq j_{max}$, but there is no notable effect on our result. The most delicate case happens when $j_4 = j_{max}$ similarly to the proof of Lemma \ref{lem:nonres3}, then  we get
\begin{align*}
\eqref{eq:nonres5-1} &\lesssim 2^{4k_3}2^{-k_4}\sum_{4k_3+k_4 -5 \le j_4} \beta_{j_4,k_4}\sum_{j_1,j_2,j_3 \ge 2k_3}2^{(j_1+j_2+j_3)/2} 2^{-(j_{max}+j_{sub})/2}2^{k_4/2}\prod_{i=1}^{3}\norm{f_{i}}_{L_{\tau}^2\ell_n^2}\\
&\lesssim 2^{2k_3}2^{-2 k_4}\norm{P_{k_1}u}_{F_{k_1}}\norm{P_{k_2}v}_{F_{k_2}}\norm{P_{k_3}w}_{F_{k_3}}.
\end{align*}
		\textbf{Case (b): $k_4 \le k_1-10$.} Similarly to \textbf{Case (a)}, we only focus on the case 
		\begin{equation*}
			j_4 = j_{max} \ge 4k_3 + k_1 - 5,
		\end{equation*}
and therefore we obtain
		\begin{align*}
			\begin{aligned}
				\eqref{eq:nonres5-1} &\lesssim 2^{4k_3}2^{-k_4}\sum_{4k_3 + k_1 -5 \le j_4}\beta_{j_4,k_4}\sum_{j_1,j_2,j_3 \ge 2k_3}2^{(j_1+j_2+j_3)/2} 2^{-(j_{max}+j_{sub})/2}2^{k_4/2}\prod_{i=1}^{3}\norm{f_{i}}_{L_{\tau}^2\ell_n^2}\\
				&\lesssim 2^{2k_3}2^{-\frac74 k_4}2^{-\frac14 k_1}\norm{P_{k_1}u}_{F_{k_1}}\norm{P_{k_2}v}_{F_{k_2}}\norm{P_{k_3}w}_{F_{k_3}}.
			\end{aligned}
		\end{align*}
		
		\textbf{Case (c): $|k_1 - k_4| < 10$.} Note from \eqref{eq:high modulation} that $j_{max} \ge 4k_3$. Similarly to above cases, the case when $j_4 = j_{max}$ is the most delicate one, and we only deal with this case. Note also that the case $(k_4j_4) = (k_{thd}, j_{max})$ never happens. By Lemma \ref{lem:tri-L2} (b) with $j_{max} \ge 4k_3$, we obtain
		\begin{equation*}
			\begin{aligned}
				\eqref{eq:nonres5-1} &\lesssim 2^{4k_3}2^{-k_4}\sum_{4k_3 \le j_{max}=j_4} \beta_{j_4,k_4}\sum_{j_1,j_2,j_3 \ge 2k_3}2^{(j_1+j_2+j_3)/2} 2^{-(j_{max}+j_{sub})/2}2^{k_4/2}\prod_{i=1}^{3}\norm{f_{i}}_{L_{\tau}^2\ell_n^2}\\
&\lesssim 2^{2k_3}2^{-\frac74 k_4}\norm{P_{k_1}u}_{F_{k_1}}\norm{P_{k_2}v}_{F_{k_2}}\norm{P_{k_3}w}_{F_{k_3}},
			\end{aligned}
		\end{equation*}
		which completes the proof of Lemma \ref{lem:nonres5}.
	\end{proof}
	
	\subsection{Energy estimates}\label{sec:energy-esti}
	
	In this section, we only deal with the energy for the difference of two solutions $v_1$ and $v_2$ to the following equation: 
	\begin{equation}\label{eq:5mkdv4}
		\begin{split}
			\pt\wh{v}(n) - i(n^5 + c_1n^3 + c_2n)\wh{v}(n)=&-20in^3|\wh{v}(n)|^2\wh{v}(n)\\
			&+10in \sum_{\N_{3,n}} \wh{v}(n_1)\wh{v}(n_2)n_3^2\wh{v}(n_3) \\
			&+5in \sum_{\N_{3,n}} (n_1+n_2)\wh{v}(n_1)\wh{v}(n_2)n_3\wh{v}(n_3)\\
			&+6i n\sum_{\N_{5,n}} \wh{v}(n_1)\wh{v}(n_2)\wh{v}(n_3)\wh{v}(n_4)\wh{v}(n_5).
		\end{split}
	\end{equation}
Let $w = v_1 - v_2$. Then $w$ satisfies
	\begin{align}
		\begin{aligned}\label{eq:5mkdv8}
			&\pt\wh{w}(n) - i\mu(n)\wh{w}(n)= \sum_{j=1}^4\wh{N}_j(v_1,v_2,w), \qquad w(0,x) = w_0(x) = v_{1,0}(x) - v_{2,0}(x),
		\end{aligned}
	\end{align}
where
\begin{align*}
&\wh{N}_1(v_1,v_2,w) = -20in^3\left(|\wh{v}_1(n)|^2\wh{w}(n) + \wh{v}_1(n)\wh{v}_2(n)\wh{w}(-n) + |\wh{v}_2(n)|^2\wh{w}(n)\right),\\ 
&\begin{aligned}
		\wh{N}_2(v_1,v_2,w) =&~{} 10in \sum_{\N_{3,n}} \wh{w}(n_1)\wh{v}_1(n_2)n_3^2\wh{v}_1(n_3) + 10in \sum_{\N_{3,n}} \wh{v}_2(n_1)\wh{w}(n_2)n_3^2\wh{v}_1(n_3)\\
		&\;\;\;+ 10in \sum_{\N_{3,n}} \wh{v}_2(n_1)\wh{v}_2(n_2)n_3^2\wh{w}(n_3),
	\end{aligned}\\
&\begin{aligned}
		\wh{N}_3(v_1,v_2,w) =&~{} 5in \sum_{\N_{3,n}} (n_1+n_2)\wh{w}(n_1)\wh{v}_1(n_2)n_3\wh{v}_1(n_3) + 5in \sum_{\N_{3,n}} (n_1+n_2)\wh{v}_2(n_1)\wh{w}(n_2)n_3\wh{v}_1(n_3)\\
		&\;\;\;+ 5in \sum_{\N_{3,n}} (n_1+n_2)\wh{v}_2(n_1)\wh{v}_2(n_2)n_3\wh{w}(n_3),
	\end{aligned}
\end{align*}
and
	\[\begin{aligned}
		\wh{N}_4(v_1,v_2,w) =&~{} 6i n\sum_{\N_{5,n}} \wh{w}(n_1)\wh{v}_1(n_2)\wh{v}_1(n_3)\wh{v}_1(n_4)\wh{v}_1(n_5) + 6i n\sum_{\N_{5,n}}\wh{v}_2(n_1)\wh{w}(n_2)\wh{v}_1(n_3)\wh{v}_1(n_4)\wh{v}_1(n_5)\\
		&+ 6i n\sum_{\N_{5,n}}\wh{v}_2(n_1)\wh{v}_2(n_2)\wh{w}(n_3)\wh{v}_1(n_4)\wh{v}_1(n_5) + 6i n\sum_{\N_{5,n}}\wh{v}_2(n_1)\wh{v}_2(n_2)\wh{v}_2(n_3)\wh{w}(n_4)\wh{v}_1(n_5)\\
		&+ 6i n\sum_{\N_{5,n}}\wh{v}_2(n_1)\wh{v}_2(n_2)\wh{v}_2(n_3)\wh{v}_2(n_4)\wh{w}(n_5).
	\end{aligned}\]
	We denote $\wh{N}_1(v_1,v_2,w)+\wh{N}_2(v_1,v_2,w)+\wh{N}_3(v_1,v_2,w)+\wh{N}_4(v_1,v_2,w)$ by $\wh{N}(v_1,v_2,w)$ only in the proof of Proposition \ref{prop:energy1-3} below. As we mentioned in Section \ref{sec:idea}, we define the localized modified energy by
\begin{equation}\label{eq:new energy1-5}
			\begin{aligned}
				{E}_{k}(w)(t) &=\norm{P_kw(t)}_{L^2}^2 \\
				&+\sum_{1 \le \ell \le m \le 2} \mbox{Re}\left[{\kappa}_{\ell,m} \sum_{n,\overline{\N}_{3,n}}\wh{v}_\ell(n_1)\wh{v}_m(n_2)\psi_k(n_3)\frac{1}{n_3}\wh{w}(n_3)\chi_k(n)\frac1n\wh{w}(n)\right]\\
				&+\sum_{1 \le \ell \le m \le 2} \mbox{Re}\left[{\epsilon}_{\ell,m} \sum_{n,\overline{\N}_{3,n}}\wh{v}_\ell(n_1)\wh{v}_m(n_2)\chi_k(n_3)\frac{1}{n_3}\wh{w}(n_3)\chi_k(n)\frac1n\wh{w}(n)\right]
			\end{aligned}
		\end{equation}
and
	\[{E}_{T}^s(w) = \norm{P_0w(0)}_{L^2}^2 + \sum_{k \ge 1}2^{2sk} \sup_{t_k \in [-T,T]} {E}_{k}(w)(t_k),\]
where ${\kappa}_{\ell,m}$ and ${\epsilon}_{\ell,m}$, $1 \le \ell \le m \le 2$, are real and will be chosen later.

	\begin{rem}\label{rem:modified energy}
The coefficients $\kappa_{\ell,m}$ and $\epsilon_{\ell,m}$ will be chosen appropriately for high-low interaction components to formulate commutator estimates as in Lemma \ref{lem:commutator1}. 

On the other hands, symmetrizing $\wh{N}_2(v_1,v_2,w)$ and $\wh{N}_3(v_1,v_2,w)$, one can rewrite 
	\begin{equation}\label{eq:energy-nonlinear3-1}
		\begin{aligned}
			&\wh{N}_2(v_1,v_2,w)\\
=&\frac{10}{3}in \sum_{\N_{3,n}} \Big(\wh{v}_1(n_1)\wh{v}_1(n_2)+\wh{v}_1(n_1)\wh{v}_2(n_2)+\wh{v}_2(n_1)\wh{v}_2(n_2)\Big)n_3^2\wh{w}(n_3)\\ 
			&+ \frac{20}{3}in \sum_{\N_{3,n}} n_2^2\Big(\wh{v}_1(n_1)\wh{v}_1(n_2)+\wh{v}_1(n_1)\wh{v}_2(n_2)+\wh{v}_2(n_1)\wh{v}_1(n_2)+\wh{v}_2(n_1)\wh{v}_2(n_2)\Big)\wh{w}(n_3)
		\end{aligned}
	\end{equation}	
	and
	\begin{equation*}
		\begin{aligned}
		\wh{N}_3(v_1,v_2,w)=&~{}\sum_{1 \le \ell \le m \le 2}\frac{10}{3}in \sum_{\N_{3,n}} (n_1+n_2)\wh{v}_\ell(n_1)\wh{v}_m(n_2)n_3\wh{w}(n_3)\\
			&~{}+\sum_{1 \le \ell \le m \le 2} 5in\sum_{\N_{3,n}} n_1\wh{v}_\ell(n_1)n_2\wh{v}_m(n_2)\wh{w}(n_3).
		\end{aligned}
	\end{equation*}
Then, the quintic resonant interaction components with respect to the first term in \eqref{eq:energy-nonlinear3-1} can be removed due to the choice of same values of $\kappa_{\ell, m}$ and $\epsilon_{\ell,m}$, respectively, for all $1 \le \ell \le m \le 2$, see Remark \ref{rem:resonant2} below.
	\end{rem}

Except for the cases mentioned in Remark \ref{rem:modified energy}, we only deal with
\begin{equation*}
		\wh{N}_1(v_2,w) = -20in|\wh{v}_2(n)|^2\wh{w}(n),
	\end{equation*}
	\begin{align}
		\begin{aligned}\label{eq:energy-nonlinear3-2}
			\wh{N}_2(v_2,w) &= \frac{10}{3}in \sum_{\N_{3,n}} \wh{v}_2(n_1)\wh{v}_2(n_2)n_3^2\wh{w}(n_3)+\frac{20}{3}in \sum_{\N_{3,n}} n_2^2\wh{v}_2(n_1)\wh{v}_2(n_2)\wh{w}(n_3),
		\end{aligned}
	\end{align}
	and
	\begin{align}
		\begin{aligned}\label{eq:energy-nonlinear4-2}
			\wh{N}_3(v_2,w) &= \frac{10}{3}in \sum_{\N_{3,n}} (n_1+n_2)\wh{v}_2(n_1)\wh{v}_2(n_2)n_3\wh{w}(n_3)+ 5in\sum_{\N_{3,n}} n_1\wh{v}_2(n_1)n_2\wh{v}_2(n_2)\wh{w}(n_3)
		\end{aligned}
	\end{align}
	as the cubic nonlinear terms for the sake of computational convenience, and we simply denote the relevant modified energy by
	\begin{equation*}
		\begin{aligned}
			{E}_{k}(w)(t) =&~{} \norm{P_kw(t)}_{L^2}^2 \\
			&+ \mbox{Re}\left[{\kappa} \sum_{n,\overline{\N}_{3,n}}\wh{v}_2(n_1)\wh{v}_2(n_2)\psi_k(n_3)\frac{1}{n_3}\wh{w}(n_3)\chi_k(n)\frac1n\wh{w}(n)\right]\\
			&+ \mbox{Re}\left[{\epsilon} \sum_{n,\overline{\N}_{3,n}}\wh{v}_2(n_1)\wh{v}_2(n_2)\chi_k(n_3)\frac{1}{n_3}\wh{w}(n_3)\chi_k(n)\frac1n\wh{w}(n)\right]
		\end{aligned}
	\end{equation*}
for appropriate real constants $\kappa$ and $\epsilon$.

	\begin{prop}\label{prop:energy1-3}
		Let $s \ge 2$ and $T \in (0,1]$. Then, for smooth solutions $v_1, v_2 \in F^{2s}(T) \cap C([-T,T];H^{\infty}(\T))$ to \eqref{eq:5mkdv4} and $w \in F^s(T) \cap C([-T,T];H^{\infty}(\T))$ to \eqref{eq:5mkdv8}, we have
		\begin{equation*}
			\begin{aligned}
				{E}_{T}^0(w) \lesssim&~{} (1+ \norm{v_{1,0}}_{H^s(\T)}^2+\norm{v_{1,0}}_{H^s(\T)}\norm{v_{2,0}}_{H^s(\T)}+\norm{v_{2,0}}_{H^s(\T)}^2)\norm{w_0}_{L^2(\T)}^2\\
				&+\left(\norm{v_1}_{F^{s}(T)}^2+\norm{v_1}_{F^{s}(T)}\norm{v_2}_{F^{s}(T)}+\norm{v_2}_{F^{s}(T)}^2\right)\norm{w}_{F^0(T)}^2\\
				&+\Big(\sum_{m=0}^{4}\norm{v_1}_{F^{s}(T)}^{4-m}\norm{v_2}_{F^{s}(T)}^m\Big)\norm{w}_{F^0(T)}^2\\
				&+\Big(\sum_{m=0}^{6}\norm{v_1}_{F^{s}(T)}^{6-m}\norm{v_2}_{F^{s}(T)}^m\Big)\norm{w}_{F^0(T)}^2
			\end{aligned} 
		\end{equation*}
		and
		\begin{equation*}
			\begin{aligned}
				{E}_{T}^s(w)&\les (1+ \norm{v_{1,0}}_{H^s(\T)}^2+\norm{v_{1,0}}_{H^s(\T)}\norm{v_{2,0}}_{H^s(\T)}+\norm{v_{2,0}}_{H^s(\T)}^2)\norm{w_0}_{H^s(\T)}^2\\
				&\qquad+\left(\norm{v_1}_{F^s(T)}^2+\norm{v_1}_{F^s(T)}\norm{v_2}_{F^s(T)}+\norm{v_2}_{F^s(T)}^2\right)\norm{w}_{F^s(T)}^2\\
				&\qquad +\left(\sum_{\ell,m=1,2}\norm{v_\ell}_{F^{s}(T)}\norm{v_m}_{F^{2s}(T)}\right)\norm{w}_{F^0(T)}\norm{w}_{F^s(T)}\\
				&\qquad +\Big(\sum_{m=0}^{4}\norm{v_1}_{F^{s}(T)}^{4-m}\norm{v_2}_{F^{s}(T)}^m\Big)\norm{w}_{F^s(T)}^2\\
				&\qquad +\Big(\sum_{m=0}^{3}\norm{v_1}_{F^{s}(T)}^{3-m}\norm{v_2}_{F^{s}(T)}^m\Big)(\norm{v_1}_{F^{2s}(T)}+\norm{v_2}_{F^{2s}(T)})\norm{w}_{F^0(T)}\norm{w}_{F^s(T)}\\
				&\qquad +\Big(\sum_{m=0}^{6}\norm{v_1}_{F^{s}(T)}^{6-m}\norm{v_2}_{F^{s}(T)}^m\Big)\norm{w}_{F^s(T)}^2.
			\end{aligned}
		\end{equation*} 
	\end{prop}

	\begin{proof}
		For any $k \ge 1$ and $t \in [-T,T]$, we differentiate ${E}_{k}(w)$ with respect to $t$ and deduce that 
		\[\frac{d}{dt}{E}_{k}(w) = I_1(t) + I_2(t) + I_3(t),\]
		where 
		\[\begin{aligned}
			I_1(t) &:= \frac{d}{dt}\norm{P_kw}_{L^2}^2\\
			&= 20 i \sum_{n} \chi_k^2(n) n^3 \wh{v}_1(-n)\wh{v}_2(-n)\wh{w}(n)\wh{w}(n)\\
			&\qquad+ 2\mbox{Re}\left[\sum_{n}\chi_k(n)\left(\overline{\wh{N}}_{2}(v_1,v_2,w)+\overline{\wh{N}}_{3}(v_1,v_2,w)+\overline{\wh{N}}_{4}(v_1,v_2,w)\right)\chi_k(n)\wt{w}(n)\right],
		\end{aligned}\]
		\[
		I_2(t)=  I_{2,1}(t) + I_{2,2}(t) + I_{2,3}(t),
		\]
		where
		\begin{align*}
			I_{2,1}(t) &=\mbox{Re}\Big[{\kappa}i \sum_{n,\overline{\N}_{3,n}}\big\{10n_1n_2^3(n_3+n) + 5n_1^2n_2^2(n_3+n) + 30n_1n_2^2n_3n \\
			&\;\;\;\;+ 10 n_2^3n_3n - 5(n_1+n_2)n_3^2n^2\big\} \wh{v}_2(n_1)\wh{v}_2(n_2)\psi_k(n_3)\frac{1}{n_3}\wh{w}(n_3)\chi_k(n)\frac1n\wh{w}(n)\Big],\\
			I_{2,2}(t) &=\mbox{Re}\Big[c_1{\kappa}i \sum_{n,\overline{\N}_{3,n}}\big\{3n_1n_2(n_3+n) + 6n_2n_3n\big\} \\
			&\hspace{9em}\times \wh{v}_2(n_1)\wh{v}_2(n_2)\psi_k(n_3)\frac{1}{n_3}\wh{w}(n_3)\chi_k(n)\frac1n\wh{w}(n)\Big],
		\end{align*}
		and
		\[\begin{aligned}
			I_{2,3}(t) &=\mbox{Re}\Big[{\kappa} \sum_{n,\overline{\N}_{3,n}}\wh{N}(v_2)(n_1)\wh{v}_2(n_2)\psi_k(n_3)\frac{1}{n_3}\wh{w}(n_3)\chi_k(n)\frac1n\wh{w}(n)\\
			&\hspace{5em}+\wh{v}_2(n_1)\wh{N}(v_2)(n_2)\psi_k(n_3)\frac{1}{n_3}\wh{w}(n_3)\chi_k(n)\frac1n\wh{w}(n)\\
			&\hspace{5em}+\wh{v}_2(n_1)\wh{v}_2(n_2)\psi_k(n_3)\frac{1}{n_3}\wh{N}(v_1,v_2,w)(n_3)\chi_k(n)\frac1n\wh{w}(n)\\
			&\hspace{5em}+\wh{v}_2(n_1)\wh{v}_2(n_2)\psi_k(n_3)\frac{1}{n_3}\wh{w}(n_3)\chi_k(n)\frac1n\wh{N}(v_1,v_2,w)(n)\Big],
		\end{aligned}\]
		and
		\[
		I_3(t) = I_{3,1}(t) + I_{3,2}(t)+I_{3,3}(t) ,\] 
	where
		\begin{align*}
			I_{3,1}(t) &= \mbox{Re}\Big[{\epsilon}i \sum_{n,\overline{\N}_{3,n}}\big\{10n_1n_2^3(n_3+n) + 5n_1^2n_2^2(n_3+n) + 30n_1n_2^2n_3n \\
			&\hspace{2cm}+ 10 n_2^3n_3n - 5(n_1+n_2)n_3^2n^2\big\} \wh{v}_2(n_1)\wh{v}_2(n_2)\chi_k(n_3)\frac{1}{n_3}\wh{w}(n_3)\chi_k(n)\frac1n\wh{w}(n)\Big],\\
			{I}_{3,2}(t) &=\mbox{Re}\Big[c_1{\epsilon}i \sum_{n,\overline{\N}_{3,n}}\big\{3n_1n_2(n_3+n) + 6n_2n_3n\big\}\\
			&\hspace{9em}\times \wh{v}_2(n_1)\wh{v}_2(n_2)\chi_k(n_3)\frac{1}{n_3}\wh{w}(n_3)\chi_k(n)\frac1n\wh{w}(n)\Big],
		\end{align*}
		and
		\begin{align*}
			I_{3,3}(t) &=\mbox{Re}\Big[{\epsilon} \sum_{n,\overline{\N}_{3,n}}\wh{N}(v_2)(n_1)\wh{v}_2(n_2)\chi_k(n_3)\frac{1}{n_3}\wh{w}(n_3)\chi_k(n)\frac1n\wh{w}(n)\\
			&\hspace{4cm}+\wh{v}_2(n_1)\wh{N}(v_2)(n_2)\chi_k(n_3)\frac{1}{n_3}\wh{w}(n_3)\chi_k(n)\frac1n\wh{w}(n)\\
			&\hspace{4cm}+\wh{v}_2(n_1)\wh{v}_2(n_2)\chi_k(n_3)\frac{1}{n_3}\wh{N}(v_1,v_2,w)(n_3)\chi_k(n)\frac1n\wh{w}(n)\\
			&\hspace{4cm}+\wh{v}_2(n_1)\wh{v}_2(n_2)\chi_k(n_3)\frac{1}{n_3}\wh{w}(n_3)\chi_k(n)\frac1n\wh{N}(v_1,v_2,w)(n)\Big].
		\end{align*}
		Then, it suffices to get a bound of
		\begin{equation}\label{eq:energy1-3.3}
			\left|\int_0^{t_k} {I}_{1}(t) + {I}_{2}(t) + {I}_{3}(t) \; dt \right|.
		\end{equation}
		
\noindent		\textbf{Estimates for main cubic terms in $I_1(t)$, $I_2(t)$, and $I_3(t)$.} We first consider the terms 
		\[-\mbox{Re}\left[\frac{20}{3}i \sum_{n,\overline{\N}_{3,n}}\chi_k(n)n\wh{v}_2(n_1)\wh{v}_2(n_2)n_3^2\wh{w}(n_3)\chi_k(n)\wh{w}(n)\right]\]
		and
		\[-\mbox{Re}\left[\frac{20}{3}i \sum_{n,\overline{\N}_{3,n}}\chi_k(n)n(n_1+n_2)\wh{v}_2(n_1)\wh{v}_2(n_2)n_3\wh{w}(n_3)\chi_k(n)\wh{w}(n)\right]\]
		in ${E}_{1}$,
		\[\mbox{Re}\left[- 5{\kappa}i\sum_{n,\overline{\N}_{3,n}}(n_1+n_2)n_3^2n^2\wh{v}_2(n_1)\wh{v}_2(n_2)\psi_k(n_3)\frac{1}{n_3}\wh{w}(n_3)\chi_k(n)\frac1n\wh{w}(n)\right]\]
		and
		\[\mbox{Re}\left[- 5{\epsilon}i\sum_{n,\overline{\N}_{3,n}}(n_1+n_2)n_3^2n^2\wh{v}_2(n_1)\wh{v}_2(n_2)\chi_k(n_3)\frac{1}{n_3}\wh{w}(n_3)\chi_k(n)\frac1n\wh{w}(n)\right]\]
		in $I_2(t)$ and $I_3(t)$, respectively. By choosing ${\kappa} = -\frac{4}{3}$ and ${\epsilon} = -\frac{2}{3}$ (choosing $\kappa_{\ell, m} = -\frac{4}{3}$ and $\epsilon_{\ell, m} = -\frac{2}{3}$, for $1\le \ell \le m \le 2$, in \eqref{eq:new energy1-5}), we reduce above terms to the following:
		\begin{align*}
			A_1(k)&:=\sum_{k_1,k_2 \le k - 10}\Bigg|\sum_{n,\overline{\N}_{3,n}}\int_0^{t_k} \chi_k(n)n[\chi_{k_1}(n_1)\wh{v}_2(n_1)\chi_{k_2}(n_2)\wh{v}_2(n_2)n_3^2\wh{w}(n_3)]\chi_k(n)\wh{w}(n) \,dt\\
			&+ \frac12\sum_{n,\overline{\N}_{3,n}}\int_0^{t_k} (n_1+n_2)\chi_{k_1}(n_1)\wh{v}_2(n_1)\chi_{k_2}(n_2)\wh{v}_2(n_2)\chi_k(n_3)n_3\wh{w}(n_3)\chi_k(n)n\wh{w}(n) \,dt\\
			&- \sum_{n,\overline{\N}_{3,n}}\int_0^{t_k} (n_1+n_2)\chi_{k_1}(n_1)\wh{v}_2(n_1)\chi_{k_2}(n_2)\wh{v}_2(n_2)\psi_k(n_3)n_3\wh{w}(n_3)\chi_k(n)n\wh{w}(n) \,dt \Bigg|,\\
			A_2(k)&:=\sum_{0\le k_1,k_2 \le k - 10}\Bigg|\sum_{n,\overline{\N}_{3,n}}\int_0^{t_k} \chi_k(n)\Big[(n_1+n_2)\chi_{k_1}(n_1)\wh{v}_2(n_1)\chi_{k_2}(n_2)\wh{v}_2(n_2)n_3\wh{w}(n_3)\Big]\\
			&\hspace{27em}\times\chi_k(n)n\wh{w}(n)\; dt\\
			&- \sum_{n,\overline{\N}_{3,n}}\int_0^{t_k} (n_1+n_2)\chi_{k_1}(n_1)\wh{v}_2(n_1)\chi_{k_2}(n_2)\wh{v}_2(n_2)\chi_k(n_3)n_3\wh{w}(n_3)\chi_k(n)n\wh{w}(n)\;dt \Bigg|,\\
			A_3(k)&:=\sum_{\substack{\max(k_1,k_2) \ge k-9\\k_3 \ge 0}}\Bigg|\sum_{n,\overline{\N}_{3,n}}\int_0^{t_k}\chi_{k_1}(n_1)\wh{v}_2(n_1)\chi_{k_2}(n_2)\wh{v}_2(n_2)\\
			&\hspace{6cm}\times\chi_{k_3}(n_3)n_3^2\wh{w}(n_3)\chi_k^2(n)n\wh{w}(n)\;dt\Bigg|,\\
			A_4(k)&:=\sum_{\substack{\max(k_1,k_2) \ge k-9\\k_3 \ge 0}}\Bigg|\sum_{n,\overline{\N}_{3,n}}\int_0^{t_k}(n_1+n_2)\chi_{k_1}(n_1)\wh{v}_2(n_1)\chi_{k_2}(n_2)\wh{v}_2(n_2)\\
			&\hspace{6cm}\times\chi_{k_3}(n_3)n_3\wh{w}(n_3)\chi_k^2(n)n\wh{w}(n)\;dt\Bigg|,
		\end{align*}
		and
		\begin{align*}
			A_5(k)&:=\sum_{\max(k_1,k_2) \ge k-9}\Bigg|\sum_{n,\overline{\N}_{3,n}}\int_0^{t_k}(n_1+n_2)n_3^2n^2\chi_{k_1}(n_1)\wh{v}_2(n_1)\\
			&\hspace{2cm} \times \chi_{k_2}(n_2)\wh{v}_2(n_2)(\psi_k(n_3)+\chi_k(n_3))\frac{1}{n_3}\wh{w}(n_3)\chi_k(n)\frac1n\wh{w}(n)\;dt\Bigg|.
		\end{align*}
We consider the $A_1(k)+ A_2(k)$. By Lemmas \ref{lem:commutator1}, we see that
\[\begin{aligned}
&~{}A_1(k) + A_2(k)\\
\lesssim&~{} \sum_{\substack{k_1,k_2 \le k-10\\ |k_1 - k_2| \le 5}} 2^{2k_2} \norm{P_{k_1}v_2}_{F_{k_1}(T)}\norm{P_{k_2}v_2}_{F_{k_2}(T)}\sum_{|k-k'|\le 5} \norm{P_{k'}w}_{F_{k'}(T)}^2\\
&~{}+\sum_{\substack{k_1,k_2 \le k-10\\ k_1 \le k_2 - 10}} 2^{-(k-k_2)}2^{2k_2} \norm{P_{k_1}v_2}_{F_{k_1}(T)}\norm{P_{k_2}v_2}_{F_{k_2}(T)}\sum_{|k-k'|\le 5} \norm{P_{k'}w}_{F_{k'}(T)}^2\\
\lesssim&~{} \left(\norm{v_2}_{F^0(T)}\norm{v_2}_{F^{2}(T)} + \norm{v_2}_{F^{0+}(T)}\left(\sum_{k_2 \le k - 10}2^{-(k-k_2)}2^{2k_2} \norm{P_{k_2}v_2}_{F_{k_2}(T)}\right)\right)\sum_{|k-k'|\le 5} \norm{P_{k'}w}_{F_{k'}(T)}^2\\
\lesssim&~{}\norm{v_2}_{F^{2}(T)}^2\sum_{|k-k'|\le 5} \norm{P_{k'}w}_{F_{k'}(T)}^2,
		\end{aligned}\]
which yields that for any $s \ge 0$ and any $t_k \in (0, T]$
		\begin{align*}
			\begin{aligned}
				\sum_{k \ge 1} 2^{2sk}\Big( A_1(k) + A_2(k) \Big) \lesssim \norm{v_2}_{F^{2}(T)}^2\norm{w}_{F^s(T)}^2.
			\end{aligned}
		\end{align*}

For $A_3(k)$, the frequency conditions $\max(k_1,k_2) \ge k-9$ and $k_3 \ge 0$ can be separated by the following five cases assuming without loss of generality $k_1 \le k_2$:
\begin{equation*}
\begin{aligned}
&\mbox{\bf Case I.} \;\; k_1,k_3 \le k- 10, \;\; |k_2 - k| \le 5,\;\; \mbox{and} \;\; |k_1 - k_3| \le 5\\
&\mbox{\bf Case II.} \;\; k_3 \le k- 10, \;\; k_1 \le k_3 -10, \;\; \mbox{and} \;\; |k_2 - k| \le 5\\
&\mbox{\bf Case III.} \;\; k_1 \le k- 10 \;\; \mbox{and} \;\; k_2, k_3 \ge k - 9\\
&\mbox{\bf Case IV.} \;\; k_3 \le k- 10 \;\; \mbox{and} \;\; k_1, k_2 \ge k - 9\\
&\mbox{\bf Case V.} \;\; k_1,k_2,k_3 \ge k - 9
\end{aligned}
\end{equation*}

By Lemma \ref{lem:energy1-1}, $A_3(k)$ restricted on \textbf{Case I} and \textbf{Case II} is bounded by
\begin{equation}\label{eq:A_3(k)_1,2}
\begin{aligned}
&~{}\norm{P_{k}v_2}_{F_{k}(T)}\norm{P_{k}w}_{F_k(T)}\sum_{\substack{|k_1 - k_3| \le 5 \\ k_1 \le k -10}}2^{2k_3}\norm{P_{k_1}v_2}_{F_{k_1}(T)}\norm{P_{k_3}w}_{F_{k_3}(T)}\\
&~{}+\norm{P_{k}v_2}_{F_{k}(T)}\norm{P_{k}w}_{F_k(T)}\sum_{\substack{k_1 \le k_3-10 \\ k_3 \le k-10}} 2^{-(k-k_3)}2^{2k_3} \norm{P_{k_1}v_2}_{F_{k_1}(T)}\norm{P_{k_3}w}_{F_{k_3}(T)}\\
\lesssim&~{} \norm{P_{k}v_2}_{F_{k}(T)}\norm{P_{k}w}_{F_k(T)}\norm{v_2}_{F^{2}(T)}\left(\norm{w}_{F^{0}(T)} +\sum_{k_3 \le k-10} 2^{-(k-k_3)}2^{2k_3}\norm{P_{k_3}w}_{F_{k_3}(T)}\right).
\end{aligned}
\end{equation}

Note that $A_3(k)$ restricted on \textbf{Case IV} is dominated by $A_3(k)$ restricted on \textbf{Case III}, due to the presence of two derivatives. Again, by Lemma \ref{lem:energy1-1}, $A_3(k)$ restricted on \textbf{Case III} is bounded by
\begin{equation*}
\begin{aligned}
&~{}\norm{P_{k}w}_{F_{k}(T)}\sum_{k_1 \le k-10} 2^{k_1/2}\norm{P_{k_1}v_2}_{F_{k_1}(T)}\sum_{\substack{|k_2-k|\le 5 \\ |k_3-k|\le 5}}2^{2k_3}\norm{P_{k_2}v_2}_{F_{k_2}(T)}\norm{P_{k_3}w}_{F_{k_3}(T)} \\
&~{}+\norm{P_{k}w}_{F_{k}(T)}\sum_{k_1 \le k-10} \norm{P_{k_1}v_2}_{F_{k_1}(T)}\sum_{\substack{|k_2-k_3|\le 5 \\ k_3 \ge k+9}}2^{-(k_3-k)}2^{2k_3}\norm{P_{k_2}v_2}_{F_{k_2}(T)}\norm{P_{k_3}w}_{F_{k_3}(T)} \\
\lesssim&~{} \norm{v_2}_{F^{\frac12+}(T)}\norm{v_2}_{F^{2}(T)} \sum_{|k'-k| \le 5}\norm{P_{k'}w}_{F_{k'}(T)}^2\\
&~{} + \norm{P_{k}w}_{F_{k}(T)} \norm{v_2}_{F^{0+}(T)}\sum_{\substack{|k_2-k_3|\le 5 \\ k_3 \ge k+9}}2^{-(k_3-k)}2^{2k_3}\norm{P_{k_2}v_2}_{F_{k_2}(T)}\norm{P_{k_3}w}_{F_{k_3}(T)}.
\end{aligned}
\end{equation*}

Finally, by Lemma \ref{lem:energy1-1}, $A_3(k)$ restricted on \textbf{Case V} is bounded by\footnote{$A_3(k)$ under the case when $|k-k_3| \le 5$, $|k_1-k_2| \le 5$ and $k_3 \le k_1 - 9$ is dominated by the one under the frequency conditions $|k-k_1| \le 5$, $|k_2-k_3| \le 5$ and $k_1 \le k_2 - 9$, due to the presence of two derivatives taken in the high frequency mode.} 
\begin{equation}\label{eq:A_3(k)_5}
\begin{aligned}
&~{}\norm{P_{k}w}_{F_{k}(T)}\Bigg(\sum_{|k-k'| \le 5}2^{7k/2}\norm{P_{k'}v_2}_{F_{k'}(T)}^2\norm{P_{k'}w}_{F_{k'}(T)} \\ 
&~{} \hspace{4em}+2^{3k/2}\sum_{\substack{k_3 \ge k+9\\|k_3-k'| \le 5}}2^{k_3}\norm{P_{k'}v_2}_{F_{k'}(T)}^2\norm{P_{k'}w}_{F_{k'}(T)} \\
&~{} \hspace{4em} +\sum_{|k-k_1|\le 5}2^{3k/2}\norm{P_{k_1}v_2}_{F_{k_1}(T)} \sum_{\substack{k_3 \ge k+9 \\ |k_2-k_3|\le 5}}2^{k_3}\norm{P_{k_2}v_2}_{F_{k_2}(T)}\norm{P_{k_3}w}_{F_{k_3}(T)}\\
&~{} \hspace{4em}+ 2^k\sum_{k_1 \ge k+9}\norm{P_{k_1}v_2}_{F_{k_1}(T)}\sum_{\substack{k_1 \le k_3 -10 \\ |k_2-k_3|\le 5}}2^{-(k_3 - k_1)}2^{k_3}\norm{P_{k_2}v_2}_{F_{k_2}(T)}\norm{P_{k_3}w}_{F_{k_3}(T)}\Bigg)\\
\lesssim&~{}\norm{v_2}_{F^{\frac74}(T)}^2\sum_{|k'-k|\le 5}\norm{P_{k'}w}_{F_{k'}(T)}^2 \\
&~{}+ \norm{v_2}_{F^{\frac54}(T)} 2^{3k/2}\norm{P_{k}w}_{F_{k}(T)}\sum_{\substack{k_3 \ge k+9\\|k_3-k'| \le 5}}2^{-k_3/4}\norm{P_{k'}v_2}_{F_{k'}(T)}\norm{P_{k'}w}_{F_{k'}(T)} \\
&~{}+\sum_{|k-k'|\le 5}2^{3k/2}\norm{P_{k'}v_2}_{F_{k'}(T)}\norm{P_{k'}w}_{F_{k'}(T)} \sum_{\substack{k_3 \ge k+9 \\ |k_2-k_3|\le 5}}2^{k_3}\norm{P_{k_2}v_2}_{F_{k_2}(T)}\norm{P_{k_3}w}_{F_{k_3}(T)}\\
&~{}+ \norm{v_2}_{F^{0+}(T)}2^k\norm{P_{k}w}_{F_{k}(T)}\sum_{\substack{k_3 \ge k+9 \\ |k_2-k_3|\le 5}}2^{k_3}\norm{P_{k_2}v_2}_{F_{k_2}(T)}\norm{P_{k_3}w}_{F_{k_3}(T)}.
\end{aligned}
\end{equation}
Collecting all \eqref{eq:A_3(k)_1,2} -- \eqref{eq:A_3(k)_5}, we conclude that for $s \ge 2$
\begin{equation}\label{eq:A_3(k)_1}
\sum_{k \ge 1} 2^{2sk} A_3(k) \lesssim \norm{v_2}_{F^{2}(T)}^2\norm{w}_{F^s(T)}^2
\end{equation}
and 
\begin{equation}\label{eq:A_3(k)_2}
\sum_{k \ge 1} A_3(k) \lesssim \norm{v_2}_{F^{2}(T)}^2\norm{w}_{F^0(T)}^2,
\end{equation}
for any $t_k \in (0,T]$.

Note that $A_4(k)$ is weaker than or comparable to $A_3(k)$ in the sense that one derivative taken in $P_{k_3}w$ is distributed to $P_{k_1}v_2$ or $P_{k_2}v_2$, thus $A_4(k)$ can be dealt with similarly as $A_3(k)$, and we obtain \eqref{eq:A_3(k)_1} and \eqref{eq:A_3(k)_2} for $A_4(k)$.

		
For $A_5(k)$, we readily obtain
		\[\begin{aligned}
			A_5(k) \lesssim&~{} \sum_{k_2 \ge k+10}\norm{P_{k_2}v_2}_{F_{k_2}(T)}^22^{5k/2}\norm{P_kw}_{F_k(T)}^2\\
			&+ \sum_{|k-k'|\le 5}2^{7k/2}\norm{P_{k'}v_2}_{F_{k'}(T)}^2\norm{P_{k'}w}_{F_{k'}(T)}^2  \\ 
			&+ \sum_{k_1 \le k-10}2^{k_1/2}\norm{P_{k_1}v_2}_{F_{k_1}(T)}\sum_{|k-k'|\le5}2^{2k}\norm{P_{k'}v_2}_{F_{k'}(T)}\norm{P_{k'}w}_{F_{k'}(T)}^2 \\
			\lesssim&~{} \norm{v_2}_{F^{2}(T)}^2\sum_{|k-k'|\le 5}\norm{P_{k'}w}_{F_{k'}(T)}^2, 
		\end{aligned}\]
		which implies
		\begin{align*}
			\sum_{k\ge 1}2^{2sk} A_5(k) \lesssim \norm{v_2}_{F^{2}(T)}^2\norm{w}_{F^s(T)}^2,
		\end{align*}
		whenever $s \ge 0$ and any $t_k \in (0,T]$.

Now, we focus on the rest terms of $I_1(t),I_2(t),I_3(t)$. By $F^s(T) \hookrightarrow C_TH^s(\T)$ (see Proposition \ref{prop:small data1-1}), the cubic terms in $I_1(t)$ can be easily controlled as
\[
\begin{aligned}
\sum_{k\ge 1}2^{2sk}\Big|\sum_{n}\int_0^{t_k} \chi_k^2(n) n^3 \wh{v}_1(-n)\wh{v}_2(-n)\wh{w}(n)\wh{w}(n) \; dt\Big| \lesssim \norm{v_1}_{F^{\frac32}(T)}\norm{v_2}_{F^{\frac32}(T)}\norm{w}_{F^s(T)}^2
\end{aligned}\]
		for any $s \ge 0$ and any $t_k \in (0,T]$.

For $I_1(t)$ with \eqref{eq:energy-nonlinear3-2} and \eqref{eq:energy-nonlinear4-2}, it suffices to consider
		\begin{align*}
			A_6(k):=\left|\int_0^{t_k}\sum_{n,\overline{\N}_{3,n}}\chi_{k_1}(n_1)\wh{v}_2(n_1)\chi_{k_2}(n_2)n_2^2\wh{v}_2(n_2)\chi_{k_3}(n_3)\wh{w}(n_3)\chi_k^2(n)n\wh{w}(n) \;dt \right|
		\end{align*}
		and
		\begin{equation}\label{eq:energy1-3.19}
		\left|\int_0^{t_k}\sum_{n,\overline{\N}_{3,n}}\chi_{k_1}(n_1)n_1\wh{v}_2(n_1)\chi_{k_2}(n_2)n_2\wh{v}_2(n_2)\chi_{k_3}(n_3)\wh{w}(n_3)\chi_k^2(n)n\wh{w}(n) \;dt \right|.
		\end{equation}
		
Since \eqref{eq:energy1-3.19} is weaker than or comparable to $A_6(k)$, it is enough to consider $A_6(k)$. Note that $A_6(k)$ becomes the worst when $k = \max(k_1,k_2,k_3,k)$ and $|k_2-k| \le 5$, since all derivatives are distributed to $P_{k_2}v_2$ and $P_{k}w$. Moreover, comparing to \cite{Kwak2018-2}, it is required to improve the bound of $A_6(k)$ under the case when\footnote{The bounds of $A_6(k)$ under the other cases are valid for Proposition \ref{prop:energy1-3}.}
\begin{equation}\label{eq:case}
k_1 \le k_3 - 10, \quad k_3 \le k_2 - 10, \quad \mbox{and} \quad |k_2 - k| \le 5.
\end{equation}
We (temporarily) denote the set of triple $(k_1,k_2,k_3)$ satisfying \eqref{eq:case} by $\Theta$. Lemma \ref{lem:energy1-1} leads us  that
		\[\begin{aligned}
			\sum_{ \Theta} A_6(k) \les&~{}   \sum_{\Theta}2^{-(k-k_3)}2^{2k}\norm{P_{k_1}v_2}_{F_{k_1}(T)}\norm{P_{k_3}w}_{F_{k_3}(T)}\norm{P_{k_2}v_2}_{F_{k_2}(T)}\norm{P_{k}w}_{F_k(T)}\\
			\les&~{} \|w\|_{F^{0}(T)} \sum_{\substack{|k-k_2|\le 5 \\ k_1 \le k-10}}  2^{2k}\norm{P_{k_1}v_2}_{F_{k_1}(T)}\norm{P_{k_2}v_2}_{F_{k_2}(T)}\norm{P_{k}w}_{F_k(T)},
		\end{aligned}\]
which concludes that
		\[\sum_{k\ge1}2^{2sk}\sum_{\Theta} A_6(k) \les \|v_2\|_{F^{2}(T)}\|v_2\|_{F^{2+s}(T)}\|w\|_{F^{0}(T)}\|w\|_{F^{s}(T)} \]
		for any $s \ge 0$ and $t_k \in (0,T]$.

\medskip
		
		\textbf{Estimates for the rest cubic terms in ${I}_{2,1}(t), {I}_{2,2}(t), I_{3,1}(t)$ and ${I}_{3,2}(t)$.} It immediately follows from \cite{Kwak2018-2} without any changes.

\medskip
		
\textbf{Estimates for  quintic and septic terms in \eqref{eq:energy1-3.3}.} Since the quintic and septic terms in \eqref{eq:energy1-3.3} contain less derivatives compared with cubic terms above, one can easily control these terms similarly as in \cite{Kwak2018-2}. Thus, we only point out (but already mentioned in \cite{Kwak2018-2}) the quintic resonances appeared in $I_{2,3}(t)$ and $I_{3,3}(t)$.

		\begin{rem}\label{rem:resonant2}
{\bf The case when $n_{33} + n = 0$.} Once considering the full modified energy \eqref{eq:new energy1-5} and the full nonlinear term \eqref{eq:energy-nonlinear3-1}, one can find 9-resonant interaction components in ${I}_{2,3}(t)$ as follows:\footnote{One can analogously handle the resonances appeared in $I_{3,3}(t)$, thus we only focus on $I_{2,3}(t)$.}
			\begin{equation}\label{eq:quintic resonant1}
				\sum_{n,\overline{n}\in \Gamma_4(\Z)}\wh{v}_\ell(n_1)\wh{v}_\ell(n_2)\psi_k(n-n_{31}-n_{32})\wh{v}_\ell(n_{31})\wh{v}_\ell(n_{32})\chi_k(n)n|\wh{w}(n)|^2 \hspace{2em} \ell=1,2,
			\end{equation}
			\begin{equation}\label{eq:quintic resonant2}
				\sum_{n,\overline{n}\in \Gamma_4(\Z)}\wh{v}_1(n_1)\wh{v}_2(n_2)\psi_k(n-n_{31}-n_{32})\wh{v}_1(n_{31})\wh{v}_2(n_{32})\chi_k(n)n|\wh{w}(n)|^2, \hspace{2em}
			\end{equation}
			\begin{equation}\label{eq:quintic resonant3}
				\sum_{n,\overline{n}\in \Gamma_4(\Z)}\wh{v}_1(n_1)\wh{v}_2(n_2)\psi_k(n-n_{31}-n_{32})\wh{v}_\ell(n_{31})\wh{v}_\ell(n_{32})\chi_k(n)n|\wh{w}(n)|^2 \hspace{2em} \ell=1,2,
			\end{equation}
			\begin{equation}\label{eq:quintic resonant4}
				\sum_{n,\overline{n}\in \Gamma_4(\Z)}\wh{v}_\ell(n_1)\wh{v}_\ell(n_2)\psi_k(n-n_{31}-n_{32})\wh{v}_1(n_{31})\wh{v}_2(n_{32})\chi_k(n)n|\wh{w}(n)|^2 \hspace{2em} \ell=1,2
			\end{equation}
			and for $(\ell,m)=(1,2)$ or $(2,1)$,
			\begin{equation}\label{eq:quintic resonant5}
				\sum_{n,\overline{n}\in \Gamma_4(\Z)}\wh{v}_\ell(n_1)\wh{v}_\ell(n_2)\psi_k(n-n_{31}-n_{32})\wh{v}_m(n_{31})\wh{v}_m(n_{32})\chi_k(n)n|\wh{w}(n)|^2, \hspace{2em}
			\end{equation}
			where $\overline{n} = (n_1,n_2,n_{31},n_{32})$. Under the symmetry $n_1+n_2+n_{31}+n_{32}=0$, we observe for any real-valued functions $f, g$ that
\begin{equation}\label{eq:quintic resonant}
				\begin{aligned}
					&\sum_{\overline{n} \in \Gamma_4(\Z)}\wh{f}(n_1)\wh{g}(n_2)\psi_k(n-n_{31}-n_{32})\wh{f}(n_{31})\wh{g}(n_{32})\chi_k(n)n|\wh{w}(n)|^2\\
					=&\sum_{\overline{n} \in \Gamma_4(\Z)} \overline{\wh{f}(-n_1)\wh{g}(-n_2)\psi_k(n-n_{31}-n_{32})}\overline{\wh{f}(-n_{31})\wh{g}(-n_{32})\chi_k(n)n|\wh{w}(n)|^2}\\
					=&\sum_{\overline{n} \in \Gamma_4(\Z)}\overline{\wh{f}(n_1)\wh{g}(n_2)\psi_k(n+n_{31}+n_{32})\wh{f}(n_{31})\wh{g}(n_{32})\chi_k(n)n|\wh{w}(n)|^2}\\
					=&\sum_{\overline{n} \in \Gamma_4(\Z)}\overline{\wh{f}(n_1)\wh{g}(n_2)\psi_k(n-n_1-n_2)\wh{f}(n_{31})\wh{g}(n_{32})\chi_k(n)n|\wh{w}(n)|^2}\\
					=&\sum_{\overline{n} \in \Gamma_4(\Z)}\overline{\wh{f}(n_1)\wh{g}(n_2)\psi_k(n-n_{31}-n_{32})\wh{f}(n_{31})\wh{g}(n_{32})\chi_k(n)n|\wh{w}(n)|^2}.
				\end{aligned}
			\end{equation}
The identity \eqref{eq:quintic resonant} is valid due to Remark \ref{rem:even real}. Above observation \eqref{eq:quintic resonant} with $f=g=v_\ell$ for \eqref{eq:quintic resonant1} or $f=v_1$ and $g = v_2$ for \eqref{eq:quintic resonant2} ensures that both \eqref{eq:quintic resonant1} and \eqref{eq:quintic resonant2} vanish, since those are purely real numbers. 

Moreover, one also shows similarly as \eqref{eq:quintic resonant} that the sum\footnote{It is possible because we chose ${\kappa}_{\ell,m}$ as the same number for all $1 \le \ell \le m \le 2$.} of \eqref{eq:quintic resonant3} and \eqref{eq:quintic resonant4} vanishes, indeed,
			\[\begin{aligned}
				&\overline{\sum_{\overline{n}\in \Gamma_4(\Z)}\wh{v}_1(n_1)\wh{v}_2(n_2)\chi_k(n-n_{31}-n_{32})\wh{v}_\ell(n_{31})\wh{v}_\ell(n_{32})\chi_k(n)n|\wh{w}(n)|^2}\\
				+&\overline{\sum_{\overline{n}\in \Gamma_4(\Z)}\wh{v}_\ell(n_1)\wh{v}_\ell(n_2)\chi_k(n-n_{31}-n_{32})\wh{v}_1(n_{31})\wh{v}_2(n_{32})\chi_k(n)n|\wh{w}(n)|^2}\\
				=&\sum_{\overline{n}\in \Gamma_4(\Z)}\wh{v}_\ell(n_1)\wh{v}_\ell(n_2)\chi_k(n-n_{31}-n_{32})\wh{v}_1(n_{31})\wh{v}_2(n_{32})\chi_k(n)n|\wh{w}(n)|^2\\
				+&\sum_{\overline{n}\in \Gamma_4(\Z)}\wh{v}_1(n_1)\wh{v}_2(n_2)\chi_k(n-n_{31}-n_{32})\wh{v}_\ell(n_{31})\wh{v}_\ell(n_{32})\chi_k(n)n|\wh{w}(n)|^2,
			\end{aligned}\]
for $\ell=1,2$. An analogous argument holds for \eqref{eq:quintic resonant5} by summing all terms. 
\end{rem}
	
\begin{rem}	
{\bf The case when $n_{32} + n = 0$.} For the other quintic resonances, it suffices to consider 
			\[QR(k) :=\left|\int_0^{t_k}\sum_{n,\overline{n}\in \Gamma_4(\Z)}\wh{v}_2(n_1)\wh{v}_2(n_2)\chi_k(n-n_{31}-n_{33})\wh{v}_2(n_{31})\wh{w}(n_{33})\chi_k(n)n\wh{v}_2(-n)\wh{w}(n)\; dt\right|.\]
Similarly as in \eqref{eq:energy1-1.5} with \eqref{eq:tri-block estimate-a1}, we obtain
\[\begin{aligned}
QR(k) \lesssim&~{} \sum_{k_1,k_2,k_{31},k_{33} \ge 0}2^{(k_{min}+k_{thd})/2}\norm{P_{k_1}v_2}_{F_{k_1}(T)}\norm{P_{k_2}v_2}_{F_{k_2}(T)}\norm{P_{k_{31}}v_2}_{F_{k_{31}}(T)}\\
				&\hspace{9em} \times\norm{P_{k_{33}}w}_{F_{k_{33}}(T)}\sum_{|k-k'|\le 5}2^{k}\norm{P_{k'}v_2}_{L_T^{\infty}L^2}\norm{P_{k'}w}_{L_T^{\infty}L^2}\\
				\lesssim&~{} \norm{v_2}_{F^{\frac12+}(T)}^3\norm{w}_{F^0(T)}\sum_{|k-k'|\le 5}2^{k}\norm{P_{k'}v_2}_{L_T^{\infty}L^2}\norm{P_{k'}w}_{L_T^{\infty}L^2},
			\end{aligned}\]
which, in addition to the embedding $F^s(T) \hookrightarrow C_TH^s(\T)$, implies
			\begin{equation*}
				\sum_{k\ge 1}2^{2sk}QR(k) \lesssim \norm{v_2}_{F^{\frac12+}(T)}^3\norm{w}_{F^0(T)}\norm{v_2}_{F^{s+1}(T)}\norm{w}_{F^{s}(T)},
			\end{equation*}
for any $s \ge 0$ and any $t_k \in (0,T]$.
			
The same argument also holds for the resonances in $I_{3,3}(t)$.
		\end{rem}
	\end{proof}
	
	The following lemma implies that ${E}_{T}^s(w)$ and $\norm{w}_{E^s(T)}$ are comparable.
	\begin{lem}[Lemma 6.11 of \cite{Kwak2018-2}]\label{lem:comparable energy1-2} Let $s > \frac12$ and $T \in (0,1]$. Then, for smooth solutions $v_1, v_2 \in C([-T,T];H^{\infty}(\T))$ to \eqref{eq:5mkdv4} and $w \in E^s(T) \cap  C([-T,T];H^{\infty}(\T))$ to \eqref{eq:5mkdv8}, there exists $0 < \delta \ll 1$ such that if $\norm{v_1}_{L_T^{\infty}H^s(\T)}, \norm{v_2}_{L_T^{\infty}H^s(\T)} \le \delta$, we have
		\[\frac12\norm{w}_{E^s(T)}^2 \le E_{T}^s(w) \le \frac32\norm{w}_{E^s(T)}^2.\]
	\end{lem}

	As a corollary to Lemma \ref{lem:comparable energy1-2} and Proposition \ref{prop:energy1-3}, we obtain an \emph{a priori bound} of $\norm{w}_{E^s(T)}$ for the difference of two solutions.
	\begin{cor}\label{cor:energy1-3}
		Let $s \ge 2$ and $T \in (0,1]$. Then, for smooth solutions $v_1, v_2 \in F^{2s}(T) \cap C([-T,T];H^{\infty}(\T))$ to \eqref{eq:5mkdv4} and $w \in F^s(T) \cap  C([-T,T];H^{\infty}(\T))$ to \eqref{eq:5mkdv8}, there exists $0 < \delta \ll 1$ such that if $\norm{v_1}_{L_T^{\infty}H^s(\T)}, \norm{v_2}_{L_T^{\infty}H^s(\T)} \le \delta$, we have
		\begin{equation}\label{eq:energy1-3.1.1}
			\begin{aligned}
				&\norm{w}_{E^0(T)}^2 \\
				&\lesssim (1+ \norm{v_{1,0}}_{H^s(\T)}^2+\norm{v_{1,0}}_{H^s(\T)}\norm{v_{2,0}}_{H^s(\T)}+\norm{v_{2,0}}_{H^s(\T)}^2)\norm{w_0}_{L^2}^2\\
				&\;\;+\left(\norm{v_1}_{F^{s}(T)}^2+\norm{v_1}_{F^{s}(T)}\norm{v_2}_{F^{s}(T)}+\norm{v_2}_{F^{s}(T)}^2\right)\norm{w}_{F^0(T)}^2\\
				&\;\;+\Big(\sum_{m=0}^{4}\norm{v_1}_{F^{s}(T)}^{4-m}\norm{v_2}_{F^{s}(T)}^m\Big)\norm{w}_{F^0(T)}^2\\
				&\;\;+\Big(\sum_{m=0}^{6}\norm{v_1}_{F^{s}(T)}^{6-m}\norm{v_2}_{F^{s}(T)}^m\Big)\norm{w}_{F^0(T)}^2
			\end{aligned} 
		\end{equation}
		and
		\begin{equation}\label{eq:energy1-3.2.1}
			\begin{aligned}
				&\norm{w}_{E^s(T)}^2 \\
				&\lesssim (1+ \norm{v_{1,0}}_{H^s(\T)}^2+\norm{v_{1,0}}_{H^s(\T)}\norm{v_{2,0}}_{H^s(\T)}+\norm{v_{2,0}}_{H^s(\T)}^2)\norm{w_0}_{H^s(\T)}^2\\
				&\;\;+\left(\norm{v_1}_{F^s(T)}^2+\norm{v_1}_{F^s(T)}\norm{v_2}_{F^s(T)}+\norm{v_2}_{F^s(T)}^2\right)\norm{w}_{F^s(T)}^2\\
				&\;\;+\left(\sum_{\ell,m=1,2}\norm{v_\ell}_{F^{\frac12}(T)}\norm{v_m}_{F^{2s}(T)}\right)\norm{w}_{F^0(T)}\norm{w}_{F^s(T)}\\
				&\;\;+\Big(\sum_{m=0}^{4}\norm{v_1}_{F^{s}(T)}^{4-m}\norm{v_2}_{F^{s}(T)}^m\Big)\norm{w}_{F^s(T)}^2\\
				&\;\;+\Big(\sum_{m=0}^{3}\norm{v_1}_{F^{s}(T)}^{3-m}\norm{v_2}_{F^{s}(T)}^m\Big)(\norm{v_1}_{F^{2s}(T)}+\norm{v_2}_{F^{2s}(T)})\norm{w}_{F^0(T)}\norm{w}_{F^s(T)}\\
				&\;\;+\Big(\sum_{m=0}^{6}\norm{v_1}_{F^{s}(T)}^{6-m}\norm{v_2}_{F^{s}(T)}^m\Big)\norm{w}_{F^s(T)}^2.
			\end{aligned}
		\end{equation} 
	\end{cor}

	\begin{rem}\label{rem:energysymm}
	By taking $v_1=0$, $v_2=v$, and $w=v$ in Corollary \ref{cor:energy1-3}, we obtain the energy estimates, an {\em a priori bound} of $\norm{v}_{E^s(T)}$ for a smooth solution $v$ to the equation \eqref{eq:5mkdv4}.

\end{rem}

	\begin{cor}\label{cor:energy1-2}
		Let $s \ge 0$ and $T \in (0,1]$. Then, for a smooth solution $v \in F^2(T)\cap F^s(T) \cap C([-T,T];H^{\infty}(\T))$ to \eqref{eq:5mkdv4}, there exists $0 < \delta \ll 1$ such that if $\norm{v}_{L_T^{\infty}H^s(\T)} \le \delta$, we have
		\begin{align}
			\begin{aligned}\label{eq:energy1-2.1.1}
				&\norm{v}_{E^s(T)}^2 \les (1+ \norm{v_0}_{H^s(\T)}^2)\norm{v_0}_{H^s(\T)}^2\\
				&\hspace{2.3cm} + \left(1 + \norm{v}_{F^{\frac12+}(T)}^2 + \norm{v}_{F^{\frac12+}(T)}^4\right)\norm{v}_{F^{2}(T)}^2\norm{v}_{F^s(T)}^2.
			\end{aligned}
		\end{align} 
	\end{cor}

	\section{Proof of Theorem \ref{thm:failure}}\label{sec:illposed}
In this section, we prove Theorem \ref{thm:failure}, that is, the flow map of \eqref{eq:5mkdv3_0} fails to be smooth in $H^s(\T)$ for any $s > 0$. First, we briefly introduce an argument described in \cite{Tzvetkov1999} (we also refer to \cite{BT2006} for more abstract scenario) for general nonlinear dispersive equations:
\begin{align}\label{eq:dispersive-maineq}
	\left\{		\begin{aligned} 
		\partial_t u  + L u &= \mathfrak{M}(u),\\
		u(0,\cdot) &= \de u_0,
	\end{aligned}\right.
\end{align}
for $\de \in  \R$, where  $L=ih(-i\nabla)$ for some real-valued polynomial $h$ and the degree of nonlinearity $\mathfrak{M}$ is $m$ and $u_0 \in H^s(\T)$ for some $s\in \R$. Suppose that $u(\de,t,x)$ is a solution to \eqref{eq:dispersive-maineq} and flow map $\de u_0 \mapsto u(t)$ to be $C^m$. By Duhamel's principle \eqref{Duhamel}, we have
\begin{align}\label{eq:dispersive-duhamel}
	u(\de,t,x) = \de S(t)u_0 + \int_0^t S(t-t') \mathfrak{M}(u(\de,t,x)) dt',
\end{align}
where $S(t)= e^{tL}$. 

 By a direct computation, we see that
\begin{align}\label{eq:dispersive-derive}
	\begin{aligned}
		\partial_\de u(\de,t,x) \Big|_{\de=0} &=  S(t)u_0 =:u_1(t,x),\\
		\partial_\de^\ell u(\de,t,x) \Big|_{\de=0} &=  0, \quad 2 \le \ell < m,\\
		\partial_\de^m u(\de,t,x) \Big|_{\de=0} &= m! \int_0^t S(t-t') \mathfrak{M}(u_1(t',x))dt' =:u_m(t,x).
	\end{aligned}
\end{align}
In a similar way to above, we can define $u_\ell$ for $\ell=m+1, m+2\cdots$, and with $u(0,t,x) = 0$, we formally write $u$ as a Taylor expansion
\begin{align*}
	u(\de,t,x) = \de u_1(t,x) + \de^m u_m(t,x) + \de^{m+1} u_{m+1}(t,x) + \cdots.
\end{align*}
By the hypothesis, the following estimate holds true:
\begin{align}\label{eq:ck-smoothness}
	\normo{u_m(t,\cdot)}_{H^s(\T)} \les \|u_0\|_{H^s(\T)}^m.
\end{align}
Therefore, the flow map of \eqref{eq:dispersive-maineq} is not $C^m$, if one constructs a suitable $u_0$ such that \eqref{eq:ck-smoothness} is valid.

\medskip

For the proof of Theorem \ref{thm:failure}, recall Duhamel's formula of \eqref{eq:5mkdv3_0}:
\begin{align}\label{eq:5mkdv-c5}
	\begin{aligned}
		\wh{v}(t,n) =&~{} \delta e^{it\mu(n)} v_0 \\
&~{}- 20i\int_0^t e^{i(t-t')\mu(n)} n^3|\wh{v}(n)|^2\wh{v}(n)dt'\\
		&~{}+10in \sum_{\N_{3,n}} \int_0^t e^{i(t-t')\mu(n)}\wh{v}(n_1)\wh{v}(n_2)n_3^2\wh{v}(n_3) dt'\\
		&~{}+10in \sum_{\N_{3,n}} \int_0^t e^{i(t-t')\mu(n)}\wh{v}(n_1)n_2\wh{v}(n_2)n_3\wh{v}(n_3) dt'\\
		&~{}+6i n\sum_{\N_{5,n}} \int_0^t e^{i(t-t')\mu(n)} \wh{v}(n_1)\wh{v}(n_2)\wh{v}(n_3)\wh{v}(n_4)\wh{v}(n_5) dt'
	\end{aligned}
\end{align}
with initial datum
\[
v(0,\cdot) =  \de v_0,
\]
for $\de \in \R$, where $\mu(n)=n^5 + c_1n^3 + c_2n$.

\begin{rem}
As is well-known, the flow map fails to be uniformly continuous when $s < \frac32$, if assuming the (local) well-posedness holds at the same level (see, for instance, \cite{BGT2002, TT2004, Kwak2020} for the same issues) due to the strong cubic resonance $in^3|\wh{v}(n)|^2\wh{v}(n)$. We, however, emphasize that even the equation \eqref{eq:5mkdv-c5} without its resonance cannot allow (at least) $C^5$-differentiability of the flow, because the strength of its non-resonant interactions are still much stronger than the dispersive smoothing effect. In what follows, to illustrate this phenomenon, we thus only consider the equation \eqref{eq:5mkdv-c5} without 
\[- 20i\int_0^t e^{i(t-t')\mu(n)} n^3|\wh{v}(n)|^2\wh{v}(n)dt'.\] 

On the other hand, a direct computation under \eqref{eq:counter-c5} gives\footnote{One needs to control
\[\int_0^t e^{i(t-t')\mu(n)} n^3|\wh{v_1}(n)|^2\wh{v_3}(n)dt',\]
where $v_\ell$, $\ell=1,3$, are defined as in \eqref{eq:dispersive-derive}.}   
\[\left\|\partial_\delta^5 \left( \int_0^t e^{i(t-t')\mu(n)} n^3|\wh{v}(n)|^2\wh{v}(n)dt'\right)\right\|_{H^s(\T)} \lesssim t^2N^{6-4s},\]
which can compete against our target \eqref{eq:tN} below $H^{\frac32}$.
\end{rem}

\begin{rem}\label{rem:C3}
Motivated by the failure of \eqref{eq:failure-xsb} for all $s \in \R$, the $C^3$-ill-posedness of \eqref{eq:5mkdv-gen} can be taken into account (Theorem \ref{thm:C^3 illposed}) due to the cubic (high-low) resonances. However, such resonant terms in \eqref{eq:5mkdv-c5} have already been removed by the re-normalization of the equation. Thus, for \eqref{eq:5mkdv-c5}, we will construct an initial data for which \eqref{eq:ck-smoothness}, $m=5$, fails. For Theorem \ref{thm:C^3 illposed}, we only provide a particular quadruple for frequencies, see Remark \ref{rem:C^3diffexample}.
\end{rem}

One can expect that the flow badly behaves, when being affected by the second nonlinear interaction in \eqref{eq:5mkdv-c5}, due to the presence of three derivatives in one component. It, thus, suffices to deal with
\begin{align}\label{eq:c5-failure}
	n \sum_{\N_{3,n}} \int_0^t e^{i(t-t')\mu(n)}\wh{v}(n_1)\wh{v}(n_2)n_3^2\wh{v}(n_3) dt'.
\end{align}

We define an interaction function $f(t) = S(-t)v(t)$. Then, a direct calculation leads us that
\begin{align}\label{eq:5mkdv-inter}
	\begin{aligned}
	\partial_t \wh{f(t)} &= e^{-it\mu(n)} \left(\partial_t \wh{v}(t,n) -i\mu(n)\wh{v}(t,n) \right)	\\
		&= e^{-it\mu(n)} 10in\sum_{\N_{3,n}} \wh{v}(n_1)\wh{v}(n_2)n_3^2\wh{v}(n_3) \\
		&\quad +e^{-it\mu(n)} 10i n\sum_{\N_{3,n}}  \wh{v}(n_1)n_2\wh{v}(n_2)n_3\wh{v}(n_3) \\
		&\quad +e^{-it\mu(n)} 6i  n\sum_{\N_{5,n}}  e^{i(t-t')\mu(n)} \wh{v}(n_1)\wh{v}(n_2)\wh{v}(n_3)\wh{v}(n_4)\wh{v}(n_5).
	\end{aligned}
\end{align}
Note that we rewrite \eqref{eq:c5-failure} as
\begin{equation}\label{eq:normal form}
\begin{aligned}
	\eqref{eq:c5-failure} 	&= 	e^{it\mu(n)}\sum_{\N_{3,n}} \int_0^t e^{it'\phi(\textbf{n})} n \wh{f}(n_1)\wh{f}(n_2)n_3^2\wh{f}(n_3)   dt'\\
	&= 	ie^{it\mu(n)}\sum_{\N_{3,n}} \int_0^t \frac{\partial_{t'} (e^{it'\phi(\textbf{n})})}{\phi(\textbf{n})} n \wh{f}(n_1)\wh{f}(n_2)n_3^2\wh{f}(n_3)   dt',
\end{aligned}
\end{equation}
where $\phi(\textbf{n}) = -\mu(n) + \mu(n_1) + \mu(n_2) + \mu(n_3)$ with $\textbf{n}=(n,n_1,n_2,n_3)$. Then, by the integration by parts with respect to the time variable $t'$, \eqref{eq:c5-failure} can be distributed as following terms: 

\begin{subequations}\label{eq:c5-normal}
	\begin{align}
		& \sum_{\N_{3,n}}\frac{ e^{it'\mu(n)}}{\phi(\textbf{n})} n \wh{v}(n_1)\wh{v}(n_2)n_3^2\wh{v}(n_3)\Big|_{t'=0}^{t'=t}, \label{eq:c5-normal-a}\\
		& \sum_{\N_{3,n}}	\int_0^t e^{it'\mu(n)} \frac{n}{\phi(\textbf{n})}e^{it'\mu(n_1)}\partial_{t'} (\wh{f}(n_1))\wh{v}(n_2)n_3^2\wh{v}(n_3)   dt', \label{eq:c5-normal-b}\\
		& \sum_{\N_{3,n}} \int_0^t e^{it'\mu(n)} \frac{n}{\phi(\textbf{n})}  \wh{v}(n_1)e^{it'\mu(n_2)}\partial_{t'}(\wh{f}(n_2))n_3^2\wh{v}(n_3)   dt',\label{eq:c5-normal-c}\\
		&\sum_{\N_{3,n}} \int_0^t e^{it'\mu(n)} \frac{n}{\phi(\textbf{n})} \wh{v}(n_1)\wh{v}(n_2)n_3^2 e^{it'\mu(n_3)}\partial_{t'}(\wh{f}(n_3) )  dt'.\label{eq:c5-normal-d}
	\end{align}
\end{subequations}

\begin{rem}\label{rem:remove}
On one hand, $\partial_\de^5 \wh{v}(\de,t,n)\Big|_{\de=0}$ for the cubic term \eqref{eq:c5-normal-a} contains $v_3$ (defined in \eqref{eq:dispersive-derive}), and it can be treated similarly as the other terms, see Remark \ref{rem:v_3} for the detailed expositions. On the other hand, distributions of $\partial_{t'}\hat{f}$ in \eqref{eq:c5-normal-b} and \eqref{eq:c5-normal-c} are better than \eqref{eq:c5-normal-d} in the sense that under a suitable frequency relation (will be given below), $H^s(\T)$-norm of \eqref{eq:c5-normal-b} and \eqref{eq:c5-normal-c} is much less than the same norm of \eqref{eq:c5-normal-d}. For the details, see Appendix \ref{sec:appendix-ill1}. 
\end{rem}

By Remark \ref{rem:remove}, we, here, only consider\eqref{eq:c5-normal-d}. Using \eqref{eq:5mkdv-inter}, we decompose \eqref{eq:c5-normal-d} into the following:
\begin{subequations}\label{eq:quintic-abcde}
	\begin{align}
%
		&\sum_{\N_{3,n}} \int_0^t e^{it'\mu(n)} \frac{n}{\phi(\textbf{n})}  \wh{v}(n_1) \wh{v}(n_2)n_3^3 \left[\sum_{\mathcal N_{3,n_3}} \wh{v}(n_{31})\wh{v}(n_{32})n_{33}^2\wh{v}(n_{33})\right]dt'\label{eq:quintic-d},\\
		& \sum_{\N_{3,n}} \int_0^t e^{it'\mu(n)} \frac{n}{\phi(\textbf{n})}\wh{v}(n_1) \wh{v}(n_2)n_3^3 \left[\sum_{\mathcal N_{3,n_3}} \wh{v}(n_{31})n_{32}\wh{v}(n_{32})n_{33}\wh{v}(n_{33})\right]dt'\label{eq:quintic-e},\\
		& \sum_{\N_{3,n}}\int_0^t e^{it'\mu(n)} \frac{n}{\phi(\textbf{n})} \wh{v}(n_1) \wh{v}(n_2)n_3^3 \left[\sum_{\mathcal N_{5,n_3}} \wh{v}(n_{31})\wh{v}(n_{32})\wh{v}(n_{33})\wh{v}(n_{34})\wh{v}(n_{35})\right]   dt', \label{eq:quintic-c}
	\end{align}
\end{subequations}
where
\begin{align*}
	\mathcal N_{3,n_3} &= \left\{ (n_{31},n_{32},n_{33}) \in \Z^3 : \;\;\begin{aligned} &n_{31} +n_{32} + n_{33} =n_3\\  &(n_{31}+n_{32})(n_{31}+n_{33})(n_{32}+n_{33})\neq 0
	\end{aligned}    \right\},\\
	\mathcal N_{5,n_3} &= \left\{ (n_{31},n_{32},n_{33},n_{34},n_{35}) \in \Z^5 : \;\;\begin{aligned} &n_{31} +n_{32} + n_{33}+ n_{34}+ n_{35} =n_3\\  &n_{3\ell}+n_{3j}+n_{3m}+n_{3p}\neq 0,\\
		&	1 \le \ell <  j < m < p \le 5
	\end{aligned}    \right\}.
\end{align*}

Note that the septic term \eqref{eq:quintic-c} can be removed (due to $u(0,t,x) = 0$) and the distribution of \eqref{eq:quintic-e} is better than \eqref{eq:quintic-d} (by the same reason in Remark \ref{rem:remove}). Thus, we only deal with \eqref{eq:quintic-d} here, and see Appendix \ref{sec:appendix-ill2} for the estimate of \eqref{eq:quintic-e}.

\medskip

A direct computation gives
\begin{align*}
	\partial_\de^5 \eqref{eq:quintic-d}\Big|_{\de=0} = 	\wh{\mathcal D(v_0)(t)}(n), 
\end{align*}
where
\begin{align*}
	\wh{\mathcal D(v_0)(t)}(n) &=  \sum_{(\mathcal N_{3,n}, \mathcal N_{3,n_3})}\int_0^t e^{-it'\mu(n)}\frac{n}{\phi(\textbf{n})}  n_3^{3}n_{33}^2e^{it'\mu(n_1)}\wh{v_0}(n_1)e^{it'\mu(n_2)}\wh{v_0}(n_2) \\
	&\hspace{3.6cm}\times e^{it'\mu(n_{31})}\wh{v_0}(n_{31})e^{it'\mu(n_{32})}\wh{v_0}(n_{32})e^{it'\mu(n_{33})}\wh{v_0}(n_{33})dt'.
\end{align*}
Here $(\mathcal N_{3,n}, \mathcal N_{3,n_j})$ means that  $(n_1,n_2,n_3) \in \mathcal N_{3,n}$ and $(n_{j1},n_{j2},n_{j3}) \in \mathcal N_{3,n_j}$. 

\medskip

Now we prove Theorem \ref{thm:failure} by constructing a suitable initial data for which the following estimate is not valid:
\begin{align}\label{eq:contraction-nonli}
	\normo{\mathcal D(v_0)(t)}_{H^s(\T)} \les \|v_0\|_{H^s(\T)}^5.
\end{align}
Let fix $0 < t \ll 1$, and let $ N \in \mathbb N$ be chosen later. We define the following sequence
\begin{align}\label{eq:counter-c5}
	a_n = \begin{cases} N^{-s},  &\mbox{if} \quad  n= N-1, N,\\
1,   &\mbox{if} \quad  n= -2,-1,1,2,\\
		0,       &\mbox{otherwise},
	\end{cases} 
\end{align}
and consider $\wh{v_0}(n) = a_n$. Then we have $v_0 \in H^s(\T)$ and $\|v_0\|_{H^s(\T)} \sim 1$. We rewrite $\mathcal D(v_0)(t)$ in a simple form
\[\wh{\mathcal D(v_0)(t)}(n)=\sum_{(\mathcal N_{3,n}, {\mathcal N_{3,n_3}})}\int_0^t e^{it'\vp(\textbf{m})} \frac{n}{\phi(\textbf{n})}n_3^3n_{33}^2 \wh{v_0}(n_1)\wh{v_0}(n_2)\wh{v_0}(n_{31})\wh{v_0}(n_{32})\wh{v_0}(n_{33})\;dt',\]
where
\[
\vp(\textbf{m}) = -\mu(n) +\mu(n_{1}) + \mu(n_{2}) + \mu(n_{31}) + \mu(n_{32}) + \mu(n_{33})
\]
for $\textbf{m} := (n,n_{1},n_{2},n_{31},n_{32},n_{33})$ satisfying $(n_1,n_2,n_3) \in \mathcal N_{3,n}$ and $(n_{31}, n_{32}, n_{33}) \in {\mathcal N_{3,n_3}}$. 	By the support property of $v_0$, we have
\[
\normo{\mathcal D_0(v_0)(t)}_{H^s(\T)} \les \normo{\mathcal D(v_0)(t)}_{H^s(\T)},
\]
where 
\begin{align*}
	&\wh{\mathcal D_{0}(v_0)(t)}(N) =\int_0^t \frac{n}{\phi(\textbf{n})}n_3^3n_{33}^2 \wh{v_0}(n_1)\wh{v_0}(n_2)\wh{v_0}(n_{31})\wh{v_0}(n_{32})\wh{v_0}(n_{33})dt' \Big|_{\textbf{m} = \textbf{m}_0}
\end{align*}
for $ \textbf{m}_0 :=(N,2,-1,-2,1,N)$. In particular, $\textbf{m}_0$ satisfies that $n_3 = N-1$,  $(n_1,n_2,n_3) \in \mathcal N_{3,n}$, and $(n_{31}, n_{32}, n_{33}) \in {\mathcal N_{3,n_3}}$. Note that $\vp(\textbf{m}_0) =0$ and $\phi(N,2,-1,N-1) \sim N^4$. This yields that
\begin{equation}\label{eq:tN}
\begin{aligned}
	\normo{\mathcal D_{0}(v_0)(t)}_{H^s(\T)}			&= \left|N^s\wh{\mathcal D_{0}(v_0)(t)}(N)\right|\\
	&=\left|N^s\int_0^t e^{it'\vp(\textbf{m}_0)} \frac{N}{\phi(\textbf{n})}(N-1)^3N^2 \wh{v_0}(2)\wh{v_0}(-1)\wh{v_0}(-2)\wh{v_0}(1)\wh{v_0}(N)dt'\right| \\
	&\sim N^s\int_0^t  \frac{N}{N^4}(N-1)^3N^2 N^{-s}dt'\\
	&\sim t N^2.
\end{aligned}
\end{equation}
By taking $N \in \mathbb N$ such that $t^{-\frac12} \ll N$, we have
\[1 \ll tN^2 \sim \|\mathcal D_{0}(v_0)(t)\|_{H^s(\T)} \les \normo{ \mathcal D(v_0)(t) }_{H^s(\T)},\]
which cannot guarantee that \eqref{eq:contraction-nonli} is valid, due to $\|v_0\|_{H^s(\T)}^5 \sim 1$. Therefore we complete the proof of Theorem \ref{thm:failure}.

\begin{rem}\label{rem:v_3}
A direct computation $\partial_\delta^5 v$ on \eqref{eq:5mkdv-c5} also proves Theorem \ref{thm:failure} and is equivalent to our method (normal form approach) due to the presence of $v_3$ in the cubic nonlinearities. Note that $v_3$ involves one more integral in terms of time variable, and this time integration compensates for $(\phi(\textbf{n}))^{-1}$ in \eqref{eq:normal form}. 
\end{rem}

\begin{rem}\label{rem:C^3diffexample}
As mentioned in Remark \ref{rem:C3}, \eqref{eq:ck-smoothness}, $m=3$, also fails for $S(t) = e^{t\partial_x^5}$ and $\mathfrak{M}(u) = u^2\px^3 u$ in \eqref{eq:dispersive-duhamel}. Such a phenomenon appears when taking $\wh{v}(t,n)=e^{itn^5}a_n$, where
\[a_n = \begin{cases} N^{-s},   &\mbox{if} \quad  n= N,\\
1,   &\mbox{if}  \quad n= \pm 1,\\
		0,      &\mbox{otherwise}.
	\end{cases}\]
The particular quadruple $(n,n_1,n_2,n_3)$ to make one similarly as \eqref{eq:tN} is $(N,1,-1,N)$.
\end{rem}

\appendix

\section{Proof of Theorem \ref{thm:main}}\label{sec:Appendix A}
This section devotes to proving Theorem \ref{thm:main} via the compactness argument. The argument used here is analogous to one in \cite{IKT2008}, also has been widely used in, for instance, \cite{GKK2013, KP2015, Zhang2016, Kwak2016, Kwak2018-2}. We provide a brief story of the proof here for self-containedness.

By re-scaling argument (see Remark \ref{rem:scaling} below), Theorem \ref{thm:main} is reduced to the following small data local well-posedness:
\begin{prop}\label{prop:small data1-3}
	Let $s \ge 2$ and $T \in (0,1]$. Suppose that the initial data $u_0 \in H^s(\T)$ is satisfying
	\begin{equation}\label{eq:specified1}
		\int_{\T} (u_0(x))^2 \; dx = \gamma_1, \hspace{2em} \int_{\T} (\px u_0(x))^2 + (u_0(x))^4 \; dx = \gamma_2,
	\end{equation}
	for some $\gamma_1, \gamma_2 \ge 0$, and $\norm{u_0}_{H^s(\T)} \le \delta_0 \ll 1$, for some sufficiently small $\delta_0 > 0$. Then, there exists   a unique solution $u(t)$ to \eqref{eq:5mkdv3}  with the initial data $u_0$ on $[-T,T]$ satisfying
	\begin{align}
		&u(t,x) \in C([-T,T];H^s(\T)),\nonumber\\
		&\eta(t)\sum_{n \in \Z} e^{i(nx - 20n\int_0^t \norm{u(s)}_{L^4}^4 \; ds)}\wh{u}(t,n) \in F^s(T) \cap C([-T,T];H^s(\T)) , \nonumber
	\end{align}
	where $\eta$ is any cut-off function in $C^{\infty}(\R)$ with $\mathrm{supp}\,\eta \subset [-T,T]$.
	
	Moreover, the flow map $S_T : H^s(\T) \to C([-T,T];H^s(\T))$ is continuous.
\end{prop}

\begin{rem}\label{rem:scaling}
	Proposition \ref{prop:small data1-3} also holds for arbitrary large initial data in $H^s(\T)$ by using a suitable scaling. Indeed, for a solution $v$ to \eqref{eq:5mkdv-gen} on $[-T,T]$ with an arbitrary initial data $v_0 \in H^s(\T)$, let us define $v_\lam(t,x) = \lam v(\lam^5 t, \lam x)$ for some $\lam >0$. Then $v_\lam$ is also the solution to \eqref{eq:5mkdv-gen} with initial data $v_\lam(0,x) = \lam v_0(\lam x)$. Moreover, we have
	\[
	\|v_\lam(0,\cdot)\|_{H^s(\T_\lam)} \les \lam^{\frac12} (1+\lam^s) \|v_0\|_{H^s(\T)},
	\]
	where $\T_\lam = \left[0, \frac{2\pi}{\lam}\right]$. 

On the other hand, nonlinear and energy estimates can be similarly obtained, but with $\lambda^{-\frac12}$ coefficient (see Section 7.2 in \cite{Kwak2018-2} for details). Nevertheless, by taking $\lam \sim \left(2\de_0\|v_0\|_{H^s(\T)}^{-1}\right)^6$, one has $\lam^{-\frac13}\|v_\lam(0)\|_{H^s(\T_\lam)} \le \de_0$, which recovers $\lambda^{-\frac12}$.
\end{rem}

\begin{prop}[Appendix A of \cite{GKK2013}]\label{prop:small data1-1}
	Let $s \ge 0$, $T \in (0,1]$. Then the following holds.
\begin{enumerate}
	\item 	For  $v \in F^s(T)$, we have
	\begin{equation}\label{eq:small data1.1}
		\sup_{t \in [-T,T]} \norm{v(t)}_{H^s(\T)} \lesssim \norm{v}_{F^s(T)}. 
	\end{equation} 
 \item For  $v,w \in C([-T,T];H^{\infty}(\T))$ satisfying 
	\[\pt \wh{v}(n) + i\mu(n) \wh{v}(n) = \wh{w}(n) \;\;\mbox{  on  } (-T,T) \times \Z,\]
we have
	\begin{equation}\label{eq:small data1.2}
		\norm{v}_{F^s(T)} \lesssim \norm{v}_{E^s(T)} + \norm{w}_{N^s(T)}.
	\end{equation} 
\end{enumerate}
\end{prop}

\begin{proof}[Proof of Proposition \ref{prop:small data1-3}]
	We set $s\ge2$. Exploiting the theory of complete integrability (or inverse spectral method), we see that there is a smooth solution $u$ to \eqref{eq:5mkdv-gen} with \eqref{eq:coefficient constraint} and  $u_0 \in H^{\infty}(\T)$.\footnote{Since we only consider the \eqref{eq:5mkdv-gen} with \eqref{eq:coefficient constraint} for $c_1=40$, one can guarantee the complete integrability (see the Remark \ref{rem:integrable is enough}). For non-integrable case, one can guarantee high regularity solutions from previous result \cite{Kwak2018-2}.} For a solution $v$ to \eqref{eq:5mkdv3}, since $\norm{u_0}_{H^s(\T)} = \norm{v_0}_{H^s(\T)}$, Proposition \ref{prop:small data1-1}, \ref{prop:nonlinear1} (a) and Corollary \ref{cor:energy1-2}  imply that
	\begin{align*}
		\left \{
		\begin{aligned}
			\norm{v}_{F^{s}(T')}&\lesssim \norm{v}_{E^s(T')} + \sum_{j=1}^4\norm{N_{j}(v)}_{N^s(T')},\\
			\sum_{j=1}^4\norm{N_{j}(v)}_{N^s(T')} &\les (1+\norm{v}_{F^s(T')}^2)\norm{v}_{F^s(T')}^3,\\
			\norm{v}_{E^s(T')}^2 &\lesssim (1+ \norm{v_0}_{H^s(\T)}^2)\norm{v_0}_{H^s(\T)}^2 + (1 + \norm{v}_{F^{s}(T)}^2 + \norm{v}_{F^{s}(T)}^4)\norm{v}_{F^{s}(T)}^4,
		\end{aligned}
		\right.
	\end{align*}
	for any $T' \in [0,T]$. Let 
	$$
	X(T') = \norm{v}_{E^s(T')} + \sum_{j=1}^4\norm{N_{j}(v)}_{N^s(T')}.
	$$
	Then, we get that $X(T')$ is non-decreasing and continuous on $[0,T]$. By the bootstrap argument,   we obtain $X(T') \lesssim \norm{v_0}_{H^s(\T)}$ (see \cite{KP2015, Zhang2016} for details), and hence
	\begin{equation}\label{eq:a priori 0}
		\norm{v}_{F^s(T')} \lesssim \norm{v_0}_{H^s(\T)},
	\end{equation}
	for all $T' \in [0,T]$ and enough small $\delta_0$. Then, Proposition \ref{prop:small data1-1} leads us that
	\[
	\sup_{t \in [-T,T]}\norm{v}_{H^s(\T)} \lesssim \norm{v_0}_{H^s(\T)}.
	\]

	We fix $u_0 \in H^s(\T)$ with $\norm{u_0}_{H^s(\T)} \le \delta_0 \ll 1$. And let us  choose a sequence of functions $\set{u_{0,\ell}}\subset H^{\infty}(\T)$ such that $u_{0,\ell}$ satisfies \eqref{eq:specified1} and $u_{0,\ell} \to u_0$ in $H^s(\T)$ as $\ell \to \infty$. Let $u_\ell(t) \in H^{\infty}(\T)$ be a solution to \eqref{eq:5mkdv-gen} with \eqref{eq:coefficient constraint} and the initial data $u_{0,\ell}$. Then, it suffices to show that the sequence $\set{v_\ell}$ is a Cauchy sequence in $C([-T,T];H^s(\T))$.  For the purpose,  set $v_\ell^K = P_{\le K}v_\ell$ for $K \in \mathbb N$. Then, $v_\ell^K$ satisfies the following frequency localized equation:
	\begin{equation*}
		\begin{split}
			\pt\chi_{\le K}(n)\Big(\wh{v}_\ell(n) - i\mu(n)\wh{v}_\ell(n)\Big)=&-20i\chi_{\le K}(n)n^3|\wh{v}_\ell(n)|^2\wh{v}_\ell(n)\\
			&+6i \chi_{\le K}(n)n\sum_{\N_{5,n}} \wh{v}_\ell(n_1)\wh{v}_\ell(n_2)\wh{v}_\ell(n_3)\wh{v}_\ell(n_4)\wh{v}_\ell(n_5) \\
			&+10i\chi_{\le K}(n)n \sum_{\N_{3,n}} \wh{v}_\ell(n_1)\wh{v}_\ell(n_2)n_3^2\wh{v}_\ell(n_3) \\
			&+5i\chi_{\le K}(n)n \sum_{\N_{3,n}} (n_1+n_2)\wh{v}_\ell(n_1)\wh{v}_\ell(n_2)n_3\wh{v}_\ell(n_3),
		\end{split}
	\end{equation*}
	with the initial data $v_\ell^K(0,\cdot) = v_{0,\ell}^K$. By the triangle inequality, we have
	\begin{align}\label{eq:cauchy}
\begin{aligned}
		\sup_{t \in [-T,T]}\norm{v_\ell - v_m}_{H^s(\T)} \le&~{}\sup_{t \in [-T,T]}\norm{v_\ell - v_\ell^K}_{H^s(\T)}+\sup_{t \in [-T,T]}\norm{v_\ell^K - v_m^K}_{H^s(\T)}\\
		&~{}+\sup_{t \in [-T,T]}\norm{v_m^K - v_m}_{H^s(\T)}.
	\end{aligned}
	\end{align}
Let us consider the first and third terms in the right-hand side of \eqref{eq:cauchy}. Using \eqref{eq:small data1.1} and \eqref{eq:a priori 0}, we see that
	\begin{equation}\label{eq:C1}
		\begin{aligned}
			\sup_{t \in [-T,T]}\norm{v_\ell^K(t) - v_\ell(t)}_{H^s(\T)} &\lesssim \norm{(I-P_{\le K}) v_\ell}_{F^s(T)} \les  \norm{(I-P_{\le K})v_{0,\ell}}_{H^s(\T)},
		\end{aligned}
	\end{equation}
	for any $\ell, K \in \mathbb N$. To estimate the second term in right-hand side of \eqref{eq:cauchy}, let us now consider the two different solutions $v_1,v_2$ to \eqref{eq:5mkdv3} with respective initial data $v_{1,0}, v_{2,0}$. By \eqref{eq:small data1.2}, \eqref{eq:nonlinear2}, \eqref{eq:energy1-3.1.1}, and \eqref{eq:a priori 0}, we have 
	\[
	\norm{v_1-v_2}_{F^0(T)} \lesssim \norm{v_{1,0} - v_{2,0}}_{L^2(\T)}.
	\]
	With this, by \eqref{eq:small data1.2}, nonlinear estimates \eqref{eq:nonlinear1}, and energy estimates \eqref{eq:energy1-3.2.1} and \eqref{eq:energy1-2.1.1} with \eqref{eq:a priori 0}, we obtain
	\[
	\norm{v_1-v_2}_{F^s(T)} \lesssim \norm{v_{1,0} - v_{2,0}}_{H^s(\T)} + \Big(\norm{v_{1,0}}_{H^{2s}(\T)}+\norm{v_{2,0}}_{H^{2s}(\T)}\Big)\norm{v_{1,0}-v_{2,0}}_{L^2(\T)}.
	\]
	Then,  we estimate
	\begin{align}\label{eq:C2}
		\begin{aligned}
			&\sup_{t \in [-T,T]}\norm{v_\ell^K(t) - v_m^K(t)}_{H^s(\T)} \\
			&\quad\lesssim \norm{v_\ell^K - v_m^K}_{F^s(T)}\\
			&\quad\lesssim \norm{v_{0,\ell}^K - v_{0,m}^K}_{H^s(\T)} + \Big(\norm{v_{0,\ell}^K}_{H^{2s}(\T)}+\norm{v_{0,m}^K}_{H^{2s}(\T)}\Big)\norm{v_{0,\ell}^K-v_{0,m}^K}_{L^2(\T)}\\
			&\quad\lesssim \norm{v_{0,\ell}^K - v_{0,m}^K}_{H^s(\T)} + K^s\norm{v_{0,\ell}^K-v_{0,m}^K}_{L^2(\T)},
		\end{aligned}
	\end{align}
	for any $\ell,m,K \in \mathbb N$.
Since $v_0 \in H^s(\T)$,  we have for sufficiently large $K$,
	\begin{equation}\label{eq:limit1-1}
		\norm{(I-P_K)v_0}_{H^s(\T)} \ll 1.
	\end{equation}
Let us fix the $K$ such that \eqref{eq:limit1-1} holds.	Moreover,  the fact  that $\{v_{0,\ell}\}$ converges to $v_0$ in $H^s(\T)$ as $\ell \to \infty$, implies that $\{v_{0,\ell}\}$ and $\{v_{0,\ell}^K\}$ are Cauchy sequence in $H^s(\T)$. This yields that
\begin{equation*}
		\norm{v_{0,\ell} - v_0}_{H^s(\T)}=o(1) \mbox{ as } \ell \to \infty. 
	\end{equation*}
 Using these facts with \eqref{eq:C1}, we estimate
	\begin{align}\label{eq:limit argument1}
		\begin{aligned}
			\sup_{t \in [-T,T]}\norm{v_\ell^K - v_\ell}_{H^s(\T)}
			&\les  \Big(\norm{v_{0,\ell}-v_{0}}_{H^s(\T)} +\norm{v_0 - v_0^K}_{H^s(\T)} + \norm{v_0^K - v_{0,\ell}^K}_{H^s(\T)}\Big)\\
		&= o(1) \;\;\mbox{ as }\;\; \ell \to \infty.
		\end{aligned}
	\end{align}
In a similar way above with \eqref{eq:C2}, \eqref{eq:limit1-1}, and the fact that $\{v_{0,\ell}^K\}$ is Cauchy sequence in $L^2(\T)$,  we obtain that for large enough $\ell,m \gg 1$,
	\begin{equation*}
	\begin{aligned}
K^s\norm{v_{0,\ell}^K - v_{0,m}^K}_{L^2(\T)} &\ll 1,
	\end{aligned}
\end{equation*}
implies that
	\begin{equation}\label{eq:limit argument2}
		\begin{aligned}
			\sup_{t \in [-T,T]}\norm{v_\ell^K - v_m^K}_{H^s(\T)} &\ll 1
		\end{aligned}
	\end{equation}
for large enough $\ell,m \gg 1$. Therefore, by \eqref{eq:cauchy}, \eqref{eq:limit argument1} and \eqref{eq:limit argument2} lead us  that $\{v_\ell\}$ is Cauchy sequence in $C([-T,T]; H^s(\T))$.

The standard limit argument ensures the existence of the solution, and the rest of the proof follows an analogous way. To complete the proof, it remains to verify that the nonlinear transformation \eqref{eq:modified solution} is bi-continuous in $C([-T,T];H^s(\T))$ with $s \ge 2$, and it is done by Lemma \ref{lem:bi-continuity} below.
\end{proof}

	\begin{lem}\label{lem:bi-continuity}
		Let $s \ge \frac14$ and $0 < T < \infty$. Then, $\NT(u)$ defined as in \eqref{eq:modified solution} is bi-continuous from a ball in $C([-T,T];H^s(\T))$ to itself. 
	\end{lem}
	\begin{proof}
One can find the proof in \cite{Kwak2018-2}, thus we only point out crucial parts of the proof. Moreover, an essential part for the completeness of Proposition \ref{prop:small data1-1} is the continuity of $\NT^{-1}$ (in fact, the proof of the other part is analogous). Thus, we only prove
		\[
		u_\ell := \NT^{-1}(v_\ell) \to \NT^{-1}(v) =: u\hspace{1em}\mbox{in}\hspace{1em} C([-T,T];H^s(\T)), \hspace{1em} \mbox{as}\hspace{1em} \ell \to \infty.
		\]
Let fix $s \ge 1/4$, $T \in (0,\infty)$, and $t \in [-T,T]$, and assume that $\norm{v}_{L_T^{\infty}H^s(\T)} \les 1$. Consider
\begin{align}\label{eq:convergence}
		\norm{u_\ell(t) - u(t)}_{H^s(\T)}^2.
		\end{align}
Using \eqref{eq:modified solution}, a direct computation gives
			\begin{align*}
		\wh{u}_\ell(n) -\wh{u}(n) =&~{} e^{20in\int_0^t \norm{u_\ell(s)}_{L^4}^4 \; ds}\wh{v}_\ell(n) - e^{20in\int_0^t \norm{u(s)}_{L^4}^4 \; ds}\wh{v}(n) \\
		=&~{}e^{20in\int_0^t \norm{u(s)}_{L^4}^4 \; ds}\left[e^{20in\int_0^t \norm{u_\ell(s)}_{L^4}^4 - \norm{u(s)}_{L^4}^4 \; ds} - 1 \right]\wh{v}_\ell(n) \\
		&~{}+e^{20in\int_0^t \norm{u(s)}_{L^4}^4 \; ds}(\wh{v}_\ell(n) - \wh{v}(n)),
	\end{align*}
which divides \eqref{eq:convergence} into the following three terms:
\begin{subequations}
		\begin{align}
			& |\wh{v}_\ell(0) - \wh{v}(0)|^2, \label{eq:bi-conti1}\\
			&2^{s+1}\sum_{|n| \ge 1}\left|e^{20in\int_0^t \norm{u_\ell(s)}_{L^4}^4 - \norm{u(s)}_{L^4}^4 \; ds} - 1\right|^2|n|^{2s}|\wh{v}_\ell(n)|^2, \label{eq:bi-conti2}\\
			&2^{s+1}\sum_{|n| \ge 1}|n|^{2s}|\wh{v}_\ell(n) - \wh{v}(n)|^2 \label{eq:bi-conti3}.
		\end{align}
			\end{subequations}
By the hypothesis, we have
	\begin{equation}\label{eq:limit3}
		\norm{v_\ell - v}_{L_T^{\infty}H^s(\T)} =o(1) \mbox{ as } \ell \to \infty,
	\end{equation}
which immediately controls \eqref{eq:bi-conti1} and \eqref{eq:bi-conti3}. For the rest \eqref{eq:bi-conti2}, a direct computation gives
		\begin{equation}\label{eq:bi-conti7}
			\begin{aligned}
				\Big|\int_0^t&\norm{v_\ell(s)}_{L^4}^4 - \norm{v(s)}_{L^4}^4 \; dt \Big| \\
				&\le\int_0^T\norm{v_\ell(t) - v(t)}_{L^4}\Big(\norm{v_\ell(t)}_{L^4}+\norm{v(t)}_{L^4}\Big)\Big(\norm{v_\ell(t)}_{L^4}^2+\norm{v(t)}_{L^4}^2\Big) \; dt\\
				&\lesssim T \left(\sup_{t \in [-T,T]}\norm{v_\ell}_{H^s(\T)}^3+\sup_{t \in [-T,T]}\norm{v}_{H^s(\T)}^3\right)\sup_{t \in [-T,T]}\norm{v_\ell-v}_{H^s(\T)}\\
				&\lesssim T\sup_{t \in [-T,T]}\norm{v_\ell-v}_{H^s(\T)}.
			\end{aligned}
		\end{equation}
Note that from $\norm{v}_{L_T^{\infty}H^s(\T)} \les 1$ and \eqref{eq:limit3}, one can take $M>0$ independent on $\ell$ such that
		\begin{equation}\label{eq:limit2}
			\sum_{|n| > M} |n|^{2s}|\wh{v}_\ell(n)|^2 \ll 1.
		\end{equation}
For this $M$, we divide the summation in \eqref{eq:bi-conti2} into
		\begin{align}\label{eq:bi-conti-sum}
		\sum_{1 \le |n| \le M} (\cdots)+ \sum_{|n| > M}(\cdots) .
		\end{align}
Then, the second summation in \eqref{eq:bi-conti-sum} is immediately controlled by \eqref{eq:limit2}. For the case $|n| \le M$, since $\|u_\ell^2\|_{L^2} = \|v_\ell^2\|_{L^2}$ due to the Plancherel's theorem, \eqref{eq:bi-conti7} in addition to \eqref{eq:limit3} ensures for sufficiently large $\ell$ that
\[
		\left|n\int_0^t \norm{u_\ell(s)}_{L^4}^4 - \norm{u(s)}_{L^4}^4 \; ds\right| \ll 1, 
		\]
		which implies
		\[
		\left|e^{20in\int_0^t \norm{u_\ell(s)}_{L^4}^4 - \norm{u(s)}_{L^4}^4 \; ds} - 1\right|^2 \ll 1.
		\]
This controls the first term in \eqref{eq:bi-conti-sum}, and we indeed obtain
		\[
		\norm{u_\ell - u}_{H^s(\T)} = o(1) \mbox{ as } k \to \infty.
		\]
	\end{proof}

\section{Computations of \eqref{eq:c5-normal-b}, \eqref{eq:c5-normal-c},  and \eqref{eq:quintic-e}}\label{appendix:calculation}

\subsection{Computations for \eqref{eq:c5-normal-b} and \eqref{eq:c5-normal-c}}\label{sec:appendix-ill1}
Similarly as the case of \eqref{eq:c5-normal-d}, we know
\begin{align*}
	\partial_\de^5 \left[\eqref{eq:c5-normal-b} + \eqref{eq:c5-normal-c} \right] \Big|_{\de=0} = \sum_{\ell=1}^2 \left[\wh{\mathcal B_\ell(v_0)(t)}(n) + \wh{\mathcal C_\ell(v_0)(t)}(n) \right] ,
\end{align*}
where
\begin{align*}
	\wh{\mathcal B_1(v_0)(t)}(n) &=   \sum_{(\mathcal N_{3,n}, \mathcal N_{3,n_1})}\int_0^t e^{it'\vp_1(\textbf{m}_1)}\frac{n}{\phi(\textbf{n})}  n_1n_{13}^2n_3^{2}\wh{v_0}(n_{11})\wh{v_0}(n_{12})\wh{v_0}(n_{13})\wh{v_0}(n_2)\wh{v_0}(n_3)  dt',\\
	\wh{\mathcal B_2(v_0)(t)}(n)  &=   \sum_{ (\mathcal N_{3,n}, \mathcal N_{3,n_1})}\int_0^t e^{it'\vp_1(\textbf{m}_1)}\frac{n}{\phi(\textbf{n})} n_1n_{12}n_{13} n_3^{2}\wh{v_0}(n_{11})\wh{v_0}(n_{12})\wh{v_0}(n_{13})\wh{v_0}(n_2)\wh{v_0}(n_3) dt',
\end{align*}
and
\begin{align*}
	\wh{\mathcal C_1(v_0)(t)}(n) &=   \sum_{(\mathcal N_{3,n}, \mathcal N_{3,n_2})}\int_0^t e^{it'\vp_2(\textbf{m}_2)}\frac{n}{\phi(\textbf{n})}  n_2n_{23}^2n_3^{2}\wh{v_0}(n_1)\wh{v_0}(n_{21}) \wh{v_0}(n_{22})\wh{v_0}(n_{23})\wh{v_0}(n_3)  dt',\\
	\wh{\mathcal C_2(v_0)(t)}(n)  &=   \sum_{(\mathcal N_{3,n}, \mathcal N_{3,n_2})}\int_0^t e^{it'\vp_2(\textbf{m}_2)}\frac{n}{\phi(\textbf{n})} n_2n_{22}n_{23} n_3^{2}\wh{v_0}(n_1) \wh{v_0}(n_{21})\wh{v_0}(n_{22}) \wh{v_0}(n_{23})\wh{v_0}(n_3) dt',
\end{align*}
where
\[
\vp_1(\textbf{m}_1) = -\mu(n) + \sum_{\ell=1}^3 \mu(n_{1\ell}) +\mu(n_2) +\mu(n_3), \quad \mbox{for} \quad \textbf{m}_1 := (n,n_{11},n_{12},n_{13},n_2,n_3),
\]
and
\[
\vp_2(\textbf{m}_2) = -\mu(n) +\mu(n_1) + \sum_{\ell=1}^3 \mu(n_{2\ell})  +\mu(n_3), \quad \mbox{for} \quad \textbf{m}_2 := (n,n_1,n_{21},n_{22},n_{23},n_3).
\]
Here $\mathcal N_{3,n_\ell}$ and $(\mathcal N_{3,n}, \mathcal N_{3,n_\ell})$ are defined similarly as before.
By the symmetry between $\mathcal B_\ell$ and $\mathcal C_\ell$, it is enough to consider only $\mathcal B_\ell$. With $v_0$ defined by using \eqref{eq:counter-c5} (but we only deal with that the components of $\textbf{m}_1$ take either $1$ or $N$, except for $n$), we have
 \begin{align}
	\begin{aligned}\label{eq:appen-b1}
		\normo{\mathcal B_1(v_0)(t)}_{H^s(\T)} &= 	\normo{n^s\wh{\mathcal B_1(v_0)(t)}(n)}_{\ell^2} \\
		&\les \sum_{(\mathcal N_{3,n}, \mathcal N_{3,n_1})}   \int_0^t \left\| n^{s}  \frac{n}{\phi(\textbf{n})} n_1n_{13}^2n_3^{2}\wh{v_0}(n_{11})\wh{v_0}(n_{12})\wh{v_0}(n_{13})\wh{v_0}(n_2)\wh{v_0}(n_3)   \right\|_{\ell^2} dt'\\
		&\les t \left(1+ N^{-s} + N^{-2s} \right) (1+ N^{-1-s} + N^{1-s} + N^{1-2s} )\\
		&\les t \max(N^{1-s},1).
	\end{aligned}
\end{align} 
For the other cases of $\textbf{m}_1$, $H^s(\T)$-norm of $\mathcal B_1$ is smaller than or similar to the right-hand side of \eqref{eq:appen-b1}. Similarly, one also has 
\[
\normo{\mathcal B_{2}(v_0)(t)}_{H^s(\T)} \les t\max(N^{1-2s},1),
\]
which particularly comes from the cases $( n_{11},n_{12},n_{13},n_2,n_3) = (1,N,N,1,N)$ and $( n_{11},n_{12},n_{13},n_2,n_3) = (1,1,1,1,1)$. Therefore, we conclude that
\begin{align*}
	\sup_{t \in [0,T]} \normo{\sum_{\ell=1}^2 \mathcal B_\ell(v_0)(t) + \mathcal C_\ell(v_0)(t)  }_{H^s(\T)} \ll tN^2.
\end{align*}

$|n| \sim N$ and rough integration in time

\subsection{Computations for \eqref{eq:quintic-e}}\label{sec:appendix-ill2}
In a similar way to Appendix \ref{sec:appendix-ill1}, we also get 
\begin{align*}
	\partial_\de^5  \eqref{eq:quintic-e} \Big|_{\de=0} =   \wh{ \mathcal D_1(v_0)(t)}(n),
\end{align*}
where
\begin{align*}
	\wh{\mathcal D_1(v_0)(t)}(n)  &=   \sum_{\mathcal N_{3,n}\cap \mathcal N_{3,n_3}}\int_0^t e^{it'\vp(\textbf{m})}\frac{n}{\phi(\textbf{n})}  n_3^{3}n_{32}n_{33}\wh{v_0}(n_1)\wh{v_0}(n_2)\wh{v_0}(n_{31})\wh{v_0}(n_{32})\wh{v_0}(n_{33})dt'.
\end{align*}
Here $\vp$ and $\textbf{m}$ are defined in Section \ref{sec:illposed}. By considering the cases when $(n_1,n_2,n_3) = (1,1,N)$ and $(n_1,n_2, n_{31},n_{32},n_{33}) = (1,1,1,N,N)$, we obtain
 \begin{align}
 	\begin{aligned}\label{eq:ill-upperbound4}
 		\normo{	\mathcal D_{1}(v_0)(t)}_{H^s(\T)}  &\les t\max(N^{2-s},1),
 	\end{aligned}
 \end{align}
Therefore, \eqref{eq:ill-upperbound4} implies that
\begin{align*}
	\sup_{t \in [0,T]} \normo{ \mathcal D_1(v_0)(t)  }_{H^s(\T)} \ll tN^2.
\end{align*}


\begin{thebibliography}{00}
\bibitem{AS1981} M. Ablowitz and H. Segur, Solitons and inverse scattering transform, in ``SIAM Studies in Applied Mathematics" Vol. 4, SIAM, Philadelphia (1981). 

%

\bibitem{BT2006} J. Bejenaru and T. Tao, \emph{Sharp well-posedness and ill-posedness results for a quadratic nonlinear Schr\"odinger equation}, J. Funct. Anal. 233 (2006), 228--259.


\bibitem{BBM} T. B. Benjamin, J. L. Bona, and J. J. Mahony, \emph{Model Equations for Long Waves in Nonlinear Dispersive Systems}, Philosophical Transactions of the Royal Society of London. Series A, Mathematical and Physical Sciences, 272 (1220) (1972), 47--78.





\bibitem{BCS1} J. L. Bona, M. Chen, and J.-C. Saut, \emph{Boussinesq equations and other systems for small-amplitude long waves in nonlinear dispersive media. I: Derivation and linear theory}, J. Nonlinear. Sci. 12 (2002), 283--318.


\bibitem{BCL} J. L. Bona, T. Colin, and D. Lannes, \emph{Long wave approximations for water waves}, Arch. Ration. Mech. Anal. 178 (3) (2005), 373--410.

\bibitem{BLS} J. L. Bona, D. Lannes, and J.-C. Saut, \emph{Asymptotic models for internal waves}, J. Math. Pures Appl. (9) 89 (6) (2008),  538--566.
	
\bibitem{BS1976} J. L. Bona and R. Scott, \emph{Solutions of the Korteweg-de Vries equation in fractional order Sobolev spaces}, Duke Math. J. 43 (1976), 87--99. 

	\bibitem{BS1975} J. L. Bona and  R. S. Smith, \emph{The initial value problem for the Korteweg-de Vries equation},  Philos. Trans. Roy. Soc. London Ser. A 278 (1287) (1975),  555--601.
	
	\bibitem{Bourgain1993} J. Bourgain, \emph{Fourier transform restriction phenomena for certain lattice subsets and applications to nonlinear evolution equations. Parts I, II}, Geom. Funct. Anal. 3 (1993), 107--156, 209--262.
	


\bibitem{BKV2021} B. Bringmann, R. Killip,  and M. Visan, \emph{Global well-posedness for the fifth-order KdV equation in $H^{-1}(\R)$}, Ann. PDE 7 (2) (2021),  21--46.

\bibitem{BGT2002} N. Burq, P. G\'erard, and N. Tzvetkov, \emph{An instability property of the nonlinear Schr\"odinger equation on $S^d$}, Math. Res. Lett. 9 (2-3) (2002), 323--335.

%
%
%

\bibitem{CTD2009} A. Choudhuri, B. Talukdar, and U. Das, \emph{The Modified Korteweg-de Vries Hierarchy: Lax Pair Representation and Bi-Hamiltonian Structure}, Zeitschrift f\"ur Naturforschung A 64 (2009), 171--179.

%
	
\bibitem{CKSTT2003}  J. Colliander, M. Keel, G. Staffilani, H. Takaoka, and T. Tao, \emph{Sharp global well-posedness for KdV and modified KdV on $\R$ and $\T$}, J. Amer. Math. Soc. 16 (2003), 705--749.

	\bibitem{CKSTT2004} J. Colliander, M. Keel, G. Staffilani, H. Takaoka, and T. Tao, \emph{Multilinear estimates for periodic KdV equations, and applications}, J. Funct. Anal. 211 (2004), 173--218.


%

\bibitem{Deift} P. A. Deift, Three Lectures on “Fifty Years of KdV: An Integrable System”. In: Miller, P., Perry, P., Saut, JC., Sulem, C. (eds) ``Nonlinear Dispersive Partial Differential Equations and Inverse Scattering" Fields Institute Communications, vol 83. Springer, New York, NY. https://doi.org/10.1007/978-1-4939-9806-7\underline{\;}1

%
%
%
%

	\bibitem{GGKM1967} C. Gradner, J. Greene, M. Kruskal, and R. Miura, \emph{A method for solving the Korteweg-de Vries equation}, Phys. Rev. Lett. 19 (1967), 1095--1097.


\bibitem{Grunrock} A. Gr\"unrock, \emph{On the hierarchies of higher order mKdV and KdV equations}, Centr. Eur. J. Math. 8   (2010), 500--536.


\bibitem{Guo2009} Z. Guo, \emph{Global well-posedness of Korteweg-de Vries equation in $H^{-3/4}(\R)$}, J. Math. Pures Appl. (9) 91 (6) (2009), 583--597.
	
	
	\bibitem{guo2012} Z. Guo, \emph{Local well-posedness for dispersion generalized Benjamin-Ono equations in Sobolev spaces}, J. Differential Equations 252 (2012), 2053--2084. 
	
	\bibitem{GPWW2011} Z. Guo, L. Peng, B. Wang, and Y. Wang, \emph{Uniform well-posedness and inviscid limit for the Benjamin-Ono-Burgers equation}, Advances in Mathematics 228 (2011), 647--677.
	
	\bibitem{GKK2013} Z. Guo, C. Kwak, and S. Kwon, \emph{Rough solutions of the fifth-order KdV equations}, J. Funct. Anal. 265 (2013), 2791--2829.
	
%
	
\bibitem{Hadamard} J. Hadamard, Sur les probl\`emes aux d\'eriv\'ees partielles et leur signification physique. Princeton University Bulletin. (1902), 49--52.

\bibitem{Hasimoto1970}  H. Hasimoto, \emph{Water waves}, Kagaku, 40 (1970), 401--408.


%
%
%
%

\bibitem{IK2007JMAS}  A. Ionescu and C. Kenig, \emph{Global well-posedness of the Benjamin–Ono equation in low-regularity spaces}, J. Amer. Math. Soc. 20 (3) (2007) 753--798.
	
	\bibitem{IKT2008} A. Ionescu, C. Kenig, and D. Tataru, \emph{Global well-posedness of the KP-I initial-value problem in the energy space}, Invent. Math. 173 (2) (2008), 265--304.
	

\bibitem{KM2018} T. Kappeler and J.-C. Molnar, \emph{On the wellposedness of the KdV/KdV2 equations and their frequency maps},  Ann. Inst. H. Poincar\'e C Anal. Non Lin\'eaire 35 (1) (2018), 101--160.

	\bibitem{KP2003} T. Kappeler and  J. P\"oschel, \emph{On the Korteweg-de Vries equation and KAM theory}, Geometric analysis and nonlinear partial differential equations, Springer, Berlin (2003), 397--416.
	
	\bibitem{KT2005} T. Kappeler and P. Topalov, \emph{Global wellposedness of mKdV in $L^2(\T,\R)$}, Comm. Partial Differential Equations 30 (2005), 435--449. 
	
	\bibitem{KT2006} T. Kappeler and P. Topalov, \emph{Global wellposedness of KdV in $H^{-1}(\T,\R)$}, Duke Math. J. 135 (2006), 327--360. 
	

%

%

\bibitem{Kato1979} T. Kato, \emph{On the Korteweg-de Vries equation}, Manuscripta Math. 29 (1979), 89--99. 



	
	
	\bibitem{KPV1991} C. Kenig, G. Ponce, and L. Vega, \emph{Oscillatory integrals and regularity of dispersive equations}, Indiana U. Math. J 40 (1991), 33--69.

\bibitem{KPV1993} C. Kenig, G. Ponce, and L. Vega, \emph{Well-posedness and scattering results for the generalized Korteweg-de Vries equation via the contraction principle}, Comm. Pure Appl. Math. 46 (1993), 527--620.
	
	\bibitem{KPV1996} C. Kenig, G. Ponce, and L. Vega, \emph{A bilinear estimate with applications to the KdV equation}, J. Amer. Math. Soc. 9 (1996), 573--603.
	
	\bibitem{KP2015} C. Kenig and D. Pilod, \emph{Well-posedness for the fifth-order KdV equation in the energy space}, Trans. Amer. Math. Soc. 367 (2015), 2551-2612.
	
	\bibitem{KP2016} C. Kenig and D. Pilod, \emph{Local well-posedness for the KdV hierarchy at high regularity}, Adv. Diff. Eq. 21 (2016), 801--836.
	

\bibitem{KV2019} R. Killip and M. Visan, \emph{KdV is well-posed in $H^{-1}$}, Ann. of Math. (2) 190 (1) (2019),  249--305.

%

\bibitem{Kishimoto2009} N. Kishimoto, \emph{Well-posedness of the Cauchy problem for the Korteweg-de Vries equation at the critical regularity}, Differential Integral Equations 22 (5/6) (2009),  447--464.

	
\bibitem{KdV1895} D. J. Korteweg and G. de Vries, \emph{On the change of form of long waves advancing in a rectangular canal, and on a new type of long stationary waves}, Phil. Mag. 39 (1895), 422--443. 	


	\bibitem{Kwak2016} C. Kwak, \emph{Local well-posedness for the fifth-order KdV equations on $\T$}, J. Differential Equations 260 (2016), 7683--7737.

\bibitem{Kwak2018-1} C. Kwak, \emph{Periodic fourth-order cubic NLS : Local well-posedness and Non-squeezing property}, J. Math. Anal. Appl.  461 (2) (2018), 1327--1364.


\bibitem{Kwak2018-2} C. Kwak, \emph{Low regularity Cauchy problem for the fifth-order modified KdV equations on $\T$}, J. Hyperbolic Differ. Equ. 15 (3) (2018), 463--557. 

\bibitem{Kwak2020} C. Kwak, \emph{Well-posedness issues on the periodic modified Kawahara equation}, Annales de l'Institut Henri Poincar\'e C, Analyse Non Lin\'eaire 37 (2020), 373--416.
	
		
	\bibitem{Kwon2008-1} S. Kwon, \emph{On the fifth order KdV equation: Local well-posedness and lack of uniform continuity of the solution map} J. Differential Equations 245 (2008), 2627--2659.
	
	\bibitem{Kwon2008-2} S. Kwon, \emph{Well-posedness and ill-posedness of the fifth-order modified KdV equation}, Electron. J. Differential Equations 2008 (2008), 1--15.

%
	
	\bibitem{Lax1968} P. Lax, \emph{Integrals of nonlinear equations of evolution and solitary waves}, Comm. Pure Appl. Math 21 (1968), 467--490.
	
%
%
	
\bibitem{Liu2015} B. Liu, \emph{A priori bounds for KdV equation below $H^{-\frac34}$} J. Funct. Anal. 268 (3) (2015),  501--554.

	\bibitem{Magri1978} F. Magri, \emph{A simple model of the integrable Hamiltonian equation}, J. Math. Phys. 19 (1978), 1156--1162.
	
%

	\bibitem{Miura1968} R. Miura, \emph{Korteweg-de Vries equation and generalizations, I. A remarkable explicit nonlinear transformation}, J. Math. Phys. 9 (1968), 1202--1204.

\bibitem{MGK1968} R. Miura, C. Gradner, and M. Kruskal, \emph{Korteweg-de Vries equation and generalizations, II.Existence of conservation laws and constants of motion}, J. Math. Phys. 9 (1968), 1204--1209.


\bibitem{Molinet2011} L. Molinet, \emph{A note on ill posedness for the KdV equation}, Differential Integral Equations 24 (7/8) (2011), 759--765. 

\bibitem{MPV2018} L. Molinet, D. Pilod, and S. Vento, \emph{On unconditional well-posedness for the periodic modified Korteweg-de Vries equation},  J. Math. Soc. Japan 741 (1) (2019), 147--201.


\bibitem{MST2001} L. Molinet, J. C. Saut, and N. Tzvetkov, \emph{Ill-posedness issues for the Benjamin–Ono and related equations}, SIAM J. Math.
Anal. 33 (2001), 982--988.




\bibitem{NTT2010} K. Nakanish, H. Takaoka, and Y. Tsutsumi, \emph{Local well-posedness in low regularity of the mKdV equation with periodic boundary condition}, DCDS 28 (4) (2010), 1635--1654.

%
%
%


	
	\bibitem{Ponce1993} G. Ponce, \emph{Lax pairs and higher order models for water waves}, J. Differential Equations 102 (2) (1993), 360--381.
	

\bibitem{ST1976} J. C. Saut and R. Temam, \emph{Remarks on the Korteweg-de Vries equation}, Israel J. Math. 24 (1976), 78--87. 

%
%

	\bibitem{Staffilani1997} G. Staffilani, \emph{On solutions for periodic generalized KdV equations}, Int. Math. Res. Not. 18 (1997), 899--917.
	
\bibitem{TT2004} H. Takaoka and Y. Tsutsumi, \emph{Well-posedness of the Cauchy problem for the modified KdV equation with periodic boundary condition}, Int. Math. Res. Not. 56 (2004), 3009--3040.

	\bibitem{Tao2001} T. Tao, \emph{Multilinear weighted convolution of $L^2$ functions and applications to nonlinear dispersive equations}, Amer. J. Math. 123 (5) (2001), 839--908.
	
%
%

\bibitem{Tzvetkov1999} N. Tzvetkov, \emph{Remark on the local ill-posedness for KdV equation}, C.R. Acad. Sci. Paris S\'er. I Math. 329 (1999),  1043--1047.

%
%
%
%


\bibitem{Zhang2016} Y. Zhang, \emph{Local well-posedness of KP-I initial value problem on torus in the Besov space},  Comm. Partial Differential Equations 41 no. 2 (2016), 256--281.

	
\end{thebibliography}
\end{document}